\documentclass[pdflatex,sn-mathphys-num]{sn-jnl}

\usepackage{graphicx}%
\usepackage{multirow}%
\usepackage{amsmath,amssymb,amsfonts}%
\usepackage{esint}
\usepackage{amsthm}%
\usepackage{mathrsfs}%
\usepackage[title]{appendix}%
\usepackage{xcolor}%
\usepackage{textcomp}%
\usepackage{manyfoot}%
\usepackage{booktabs}%
\usepackage{algorithm}%
\usepackage{algorithmicx}%
\usepackage{algpseudocode}%
\usepackage{listings}%
\usepackage[style=numeric]{biblatex}
\usepackage{enumitem}
\newtheorem{theorem}{Theorem}[section]
\newtheorem{definition}{Definition}[section]
\newtheorem{lemma}[theorem]{Lemma}
\newtheorem{corollary}[theorem]{Corollary}
\newtheorem{proposition}[theorem]{Proposition}
\newtheorem{remark}[theorem]{Remark}

\begin{document}

\title[Ancient mean curvature flow asymptotic to a minimal quadratic cone]{Ancient mean curvature flow asymptotic to a minimal quadratic cone}
\author*[1]{\fnm{Junyoung} \sur{Park}}\email{jp2453@math.rutgers.edu}

\affil*[1]{\orgdiv{Department of Mathematics}, \orgname{Rutgers University}, \orgaddress{\street{110 Frelinghuysen
Road}, \city{Piscataway}, \postcode{08854-8019}, \state{NJ}, \country{USA}}}


\abstract{In this paper, we prove that a smooth, properly embedded ancient mean curvature flow that is asymptotic to an $O(n)\times O(m)$ symmetric minimal quadratic cone for $n+m \geq 10$ and lies on one side of the cone has to have unique asymptotics in the parabolic region. When additionally assuming mean convexity, we upgrade unique asymptotics to full uniqueness, and show that such a flow has to be a stationary flow given by one of the leaves of the Hardt-Simon foliation. This is the first rigidity/unique asymptotics theory for ancient mean curvature flow with a \textbf{singular minimal cone as the asymptotic model}.}

\keywords{Mean curvature flow, ancient flow, minimal quadratic cone}



\maketitle
\tableofcontents
\section{Introduction}
A smooth, one parameter family of hypersurfaces $\mathcal{M} = (M_t \subset \mathbf{R}^{n+1})_{t \in I}$ is a solution to mean curvature flow if its normal velocity is equal to the mean curvature vector, i.e
\begin{equation}
    (\partial_tx_t)^{\perp} = \Vec{H}_t = -H_t\nu_t,
\end{equation}
where $x_t$ is the position, $\Vec{H}_t = -H_t\nu_t$ is the mean curvature vector at $x_t$. Mean curvature flow is a nonlinear heat type equation acting on submanifolds, and one can hope that the flow will even out the curvature, and in the end `converge' to a surface with simple geometry. Unfortunately, due to the nonlinearity, the flow becomes singular in finite time in many typical cases such as the evolution of closed hypersurfaces. There has been extensive effort to understand the behavior near singularities, in hopes of possibly extending the flow past singular times in a canonical way.\\

A general method of investigating the singular behavior of the flow is via blowup analysis. By suitably rescaling the flow near the singularity, one can obtain a model mean curvature flow which approximates the original flow near the singularity. It turns out that these model solutions are defined for all negative times; in other words, they are ancient flows. Therefore if one has a full classification of ancient flows, then one can determine the blowup limits, which leads to an understanding of the local geometry of the original flow near the singularity. This connection with singularity analysis is an important reason why ancient flows have received much attention in the theory.\\

It seems that classifying \textit{all} ancient flows is very difficult, if not impossible. Therefore, one typically further assumes a natural geometric assumption and obtains a classification of ancient flows with the extra condition. Typical assumptions include various types of convexity, symmetry assumptions, and a non-collapsing assumption \cite{daskalopoulos2010classification}, \cite{bourni2019convex}  \cite{huisken2015convex}, \cite{LangfordLynch2020}, \cite{haslhofer2015uniqueness}, \cite{haslhofer2016ancient}, \cite{angenent2020uniqueness}, \cite{BrendleChoi2021}, \cite{choi2022translators}, \cite{choi2024classification}, \cite{choi2022bubblesheet}, \cite{bourni2021collapsing}, \cite{Clutterbuck_2007}  .\\

Another possible geometric assumption is prescribing the asymptotic behavior at infinity. If $\mathcal{M} = (M_t)_{t \in (-\infty, 0)}$ is an ancient flow, then for any sequence of $\lambda_i \to 0$, we can consider the sequence of rescaled flows
\begin{equation}
    M^i_t = \lambda_iM_{\lambda_i^{-2}t}, \ t \in (-\infty, 0).
\end{equation}
It turns out that in many cases, one can take a (subsequential) limit of the above sequences (possibly in a weak sense), and obtain a limit ancient flow $\mathcal{M}^{\infty}$ which is called the blowdown (or tangent flow at infinity) of $\mathcal{M}$. By Huisken's monotonicity formula \cite{huisken1990asymptotic}, it is well known that $\mathcal{M}^{\infty}$ has to be self-shrinking, i.e.
\begin{equation}
    M^{\infty}_t = \sqrt{-t}M^{\infty}_{-1.}
\end{equation}
Then one can prescribe $M^{\infty}_{-1}$, and try to classify all ancient flows with this prescribed asymptotic behavior. \\

The most extensively studied case is when $M^{\infty}_{-1} = \mathbf{R}^k \times \mathbf{S}^{n-k}(\sqrt{2(n-k)})$, in other words when the flow is asymptotic to multiplicity one shrinking cylinders at infinity. One reason why this has been widely studied is due to its connection with Ilmanen's mean convex neighborhood conjecture (see \cite{haslhofer2026mcf} for details). There has been a long line of work in this setting by various authors, and we now have a full classification of all the asymptotically cylindrical flows \cite{hershkovits2021translators}, \cite{choi2022ancientlowentropy}, \cite{ancientasymptotical}, \cite{DuZhu2025}, \cite{bamler2025pde}, \cite{bamler2025classification}. Some other addressed cases are when $M^{\infty}_{-1}$ is an asymptotically conical self-shrinker \cite{Chodosh2024-qv}, or a line with possibly higher multiplicity \cite{choi2026classification}. In the Lagrangian mean curvature flow setting, there has been work dealing with the case when the blowdown is a union of transverse planes of multiplicity one \cite{LotaySchulzeSzekelyhidi2024}, \cite{lambert2021ancient}. \\

A natural question is what other classes of self-shrinkers can be prescribed for $M^{\infty}_{-1}$. Motivated by the asymptotically conical case, one can consider the situation where $M^{\infty}_{-1}$ is a regular hypercone. Since $M^{\infty}_{-1}$ has to be a self-shrinker, this means $M^{\infty}_{-1}$ is a regular, minimal cone. In particular, we can consider the simplest case when $M^{\infty}_{-1}$ is an $O(n) \times O(m)$ symmetric minimal quadratic cone. Thus one can ask which ancient flows are asymptotic to a multiplicity one $O(n) \times O(m)$ symmetric minimal quadratic cone. \\

This question seems interesting for several reasons. First, it is a natural parabolic analogue of the well known classification problem of minimal hypersurfaces asymptotic to quadratic cones \cite{simon1986minimal}, \cite{Mazet2014}, \cite{edelen2023regularity}. Another reason is that minimal quadratic cones have a singularity at the origin, which is in stark contrast to all previous works except the Lagrangian case, where the prescribed model is smooth everywhere. This difference introduces new fundamental difficulties. For example, in the asymptotically conical case, the flow can be expressed as a \textit{global} normal graph defined over the shrinker, whereas in our case, the singularity at the origin does not allow such a nice representation of the flow. In fact this issue is always present whenever the model shrinker is not smooth everywhere. Since isolated singularities are one of the simplest forms of singularities, investigation of the question might lead to new ideas that may be useful when dealing with other model shrinkers with more complicated singular behavior.  \\

We also note that minimal cones themselves are an important class of self-shrinkers, and have been studied by several authors. Vel{\'a}zquez \cite{velazquez1994curvature} showed that there exists a family of mean curvature flow that develops a finite time singularity modeled over the $O(n) \times O(n)$ symmetric Simons cone. Later, Guo-\v{S}e\v{s}um \cite{guo2018analysis}, and Stolarski \cite{stolarski2023existence} studied Vel{\'a}zquez solutions more carefully, providing further analytical and geometrical properties of these solutions. In particular, Stolarski \cite{stolarski2023existence} proved that the mean curvature of some of Vel{\'a}zquez solution remains bounded up to first singular time, hence resolving the extension problem in high dimensions, which asks whether or not the mean curvature remains bounded up to first blowup time. As for more general cones, Stolarski \cite{stolarski2026structure} showed that whenever the mean curvature is `under control' up to first blowup time, then the singularity must be modeled on a minimal cone. Also, Huang \cite{huang2025mean} proved that any area minimizing hypercone can arise as a singularity model.\\

In this paper, we consider a smooth, properly embedded, ancient mean curvature flow $(M_t)_{t \in (-\infty, 0)}$ which satisfies (some of) the following assumptions:
\begin{enumerate}[label=(A\arabic*)]
    \item \label{strong convergence to simonscone assumption1} (Minimal quadratic cone as unique tangent flow at infinity) For any sequence $\lambda_i \to 0$, define a sequence of mean curvature flows
    \begin{equation}
        M^i_t = \lambda_iM_{\lambda_i^{-2}t}.
    \end{equation}
    Then 
\begin{equation}
   (M^i_t)_{t \in (-\infty, 0)} \to (\sqrt{-t}\Sigma \equiv \Sigma)_{t \in (-\infty, 0)} \text{ as }i \to \infty,
\end{equation}
as integral Brakke flows (see Ilmanen \cite{ilmanen1994elliptic}), where $\Sigma$ is an $O(n) \times O(m)$ symmetric minimal quadratic cone with $d+1=n+m \geq10$ (with unit multiplicity).
    \item \label{additional geoemtric assumption}
    $(M_t)_{t \in (-\infty, 0)}$ lies on one side of $\Sigma$. In particular, we assume without loss of generality that
    \begin{equation}
        M_t \subset \{ (x,y) \in \mathbf{R}^n \times \mathbf{R}^m\ | \ |y| \geq \sqrt{\frac{m-1}{n-1}}|x|\}.
    \end{equation}
    \item \label{mean convexity assumption4}
    In addition to assumptions \ref{strong convergence to simonscone assumption1} and \ref{additional geoemtric assumption}, $H_t = \textup{div}(\nu_{t}) \geq 0$, where $\nu_{t}$ is the unit normal vector field that is pointing away from the domain which contains the origin. 
\end{enumerate}
 We note that under assumptions \ref{strong convergence to simonscone assumption1} and \ref{additional geoemtric assumption}, $M_t$ is connected (corollary \ref{connectedness}), and $0 \notin M_t$ for all $t < 0$.
This implies that $\mathbf{R}^{d+1} \setminus M_t$ has two components, and we can make a `continuous' choice of component $0 \in \Omega_{t} \subset \mathbf{R}^{d+1} \setminus M_t$. Then $\nu_{t}$ in assumption \ref{mean convexity assumption4} is chosen so that it is pointing out of $\Omega_{t}$. \\

Our first result is that under assumptions \ref{strong convergence to simonscone assumption1} and \ref{additional geoemtric assumption}, the rescaled mean curvature flow $\Tilde{\mathcal{M}} = (\Tilde{M}_{\tau})_{\tau \in (-\infty, 0)}$ 
given by the formula
\begin{equation}
    \Tilde{M}_{\tau} = e^{\frac{\tau}{2}}M_{-e^{-\tau}}
\end{equation}
has unique asymptotics. 
\begin{theorem}\label{maintheorem : uniqueasmptotics} Let $(M_t)_{t \in (-\infty, 0)}$ be a smooth, properly embedded ancient mean curvature flow which satisfies assumptions \ref{strong convergence to simonscone assumption1}, \ref{additional geoemtric assumption}, and let $(\Tilde{M}_{\tau})_{\tau \in (-\infty, 0)}$ be the rescaled mean curvature flow. Then there exist $c(\Tilde{\mathcal{M}}) > 0$, $p(n,m) > 0$, $r_0(\Tilde{\mathcal{M}}) > 0$ so that for all sufficiently negative $\tau$, $\Tilde{M}_{\tau} \setminus B(0, r_0e^{\frac{\tau}{2}})$ is a normal graph of $u_{\tau}$ defined over $\Sigma  \setminus B(0, 2r_0e^{\frac{\tau}{2}}) \subset \mathcal{D}_{\tau} \subset \Sigma.$ Moreover, for given $k \in \mathbf{N}$, $0 < r < R < \infty$, $u_{\tau}$ satisfies
\begin{equation}
    \|u_{\tau}(x) - ce^{\frac{1 - \alpha}{2}\tau}|x|^{\alpha} \|_{C^{k}(\Sigma \cap (B(0, R) \setminus B(0, r)))} \leq O(e^{\frac{(1 - \alpha)(1 + p)}{2}\tau}) \text{ as }\tau \to -\infty.
\end{equation}
Here, 
\begin{equation}
     \alpha = \alpha(d) = \frac{-(d-2) + \sqrt{d^2 - 8d+8}}{2} \in (-2, -1), \ d+1 = m+n.
\end{equation}
\end{theorem}
\begin{remark}
    The function $ce^{\frac{1 - \alpha}{2}\tau}|x|^{\alpha}$ is precisely the asymptotic formula of the normal graph function of the type I rescaling of the stationary solution given by one of the leaves in the Hardt-Simon foliation. Thus, theorem \ref{maintheorem : uniqueasmptotics} asserts that any smooth, one-sided flow asymptotic to the minimal quadratic cone has to agree with the stationary solution at the asymptotic level in the parabolic region. We believe the one-sidedness assumption \ref{additional geoemtric assumption} to be optimal, and expect that there should be a nonstationary ancient mean curvature flow asymptotic to the minimal quadratic cone.
\end{remark}
When we additionally assume mean convexity \ref{mean convexity assumption4}, then the unique asymptotics result in theorem \ref{maintheorem : uniqueasmptotics} is upgraded to a full uniqueness result. 
\begin{theorem}\label{maintheorem : uniqueness in mean convex case}
    Let $(M_t)_{t \in (-\infty, 0)}$ be a smooth, properly embedded, ancient mean curvature flow which satisfies assumptions \ref{strong convergence to simonscone assumption1}, \ref{additional geoemtric assumption}, \ref{mean convexity assumption4}. Then $(M_t)_{t \in (-\infty, 0)}$ is a stationary flow given by one of the leaves of the Hardt-Simon foliation associated to the minimal quadratic cone.  
\end{theorem}
\begin{remark}
    In view of theorem \ref{maintheorem : uniqueasmptotics}, we believe that theorem \ref{maintheorem : uniqueness in mean convex case} should continue to hold without mean convexity. 
\end{remark}
We briefly discuss the key ideas in the proof. The proof of theorem \ref{maintheorem : uniqueasmptotics} follows the same line of argument as in previous works dealing with the cylindrical case. Assumption \ref{strong convergence to simonscone assumption1} implies that for sufficiently negative times, we can express large parts of the flow as a normal graph of a function with small norm defined over evolving domains in the minimal quadratic cone that eventually cover all of $\Sigma \setminus \{0\}$. As mentioned before, we cannot hope that such functions are defined globally in $\Sigma \setminus \{0\}$; hence we truncate the function by a cutoff function. By using the rescaled mean curvature flow equation, we can compute the evolution equation of the truncated graph function $\Tilde{u}_{\tau} = u_{\tau}\xi_{\tau}$ which is of the form
\begin{equation}
    \partial_{\tau}\Tilde{u}_{\tau} = L\Tilde{u}_{\tau} + Q\Tilde{u}_{\tau} + E_{\tau},
\end{equation}
where $L$ is the linearized rescaled mean curvature operator at the minimal quadratic cone, $Q$ is the at least quadratic error term arising from the nonlinearity of the equation, and $E_{\tau}$ is the error term coming from the cutoff function. The goal is to show that the error terms $Q\Tilde{u}_{\tau} + E_{\tau}$ can be controlled by higher powers of $\|\Tilde{u}_{\tau}\|$. To control $Q\Tilde{u}_{\tau}$, we prove an inner outer estimate in our current setting (proposition \ref{inner outer estimate}), which is an inverse Poincar\'e type inequality for the graph function $u_{\tau}$ that was first proved in the cylindrical case \cite{angenent2019unique}.\\

In previous works dealing with cylindrical ancient flows, the above inner outer estimate was sufficient to control $E_{\tau}$ as well; hence led to a successful asymptotics analysis. However, due to the different nature of the `bad' region where graphicality is lost, similar arguments do not give good control over $E_{\tau}$ in our setting, and thus we require some new ideas. To overcome this difficulty, we prove a graphical radius estimate (proposition \ref{graphical radius estimate proposition}), which estimates the size of the domain on which $u_{\tau}$ is defined in terms of the $L^2$ integral of $u_{\tau}$ over a \textit{fixed} annulus region on the cone. Both the inner outer estimate and graphical radius estimate are based on a new family of self-shrinkers we construct in proposition \ref{inner family of self shrinkers}. By choosing the correct cutoff function $\xi_{\tau}$, we can indeed control $Q\Tilde{u}_{\tau} + E_{\tau}$ essentially by $\|\Tilde{u}_{\tau}\|^{1 + p}$ for some $p=p(n,m) > 0$. Once the error terms are under control, we can use the Merle-Zaag ODE lemma (lemma A.1 in \cite{merle1998optimal}) together with one-sidedness \ref{additional geoemtric assumption} to isolate the dynamics onto the first eigenmode of $L$, which leads to the asymptotics given in theorem \ref{maintheorem : uniqueasmptotics}. \\

To prove theorem \ref{maintheorem : uniqueness in mean convex case}, we use theorem \ref{maintheorem : uniqueasmptotics} and the mean convexity assumption \ref{mean convexity assumption4} to first show that the unrescaled flow $\mathcal{M} = (M_t)_{t \in (-\infty, 0)}$ smoothly converges to one of the leaves of the Hardt-Simon foliation as $t \to -\infty$. This allows us to write $\mathcal{M}$ as a \textit{global} normal graph over the leaf for all sufficiently negative times. By using the Liouville type theorems that were essentially proved by Stolarski \cite{stolarski2023existence} (proposition \ref{liouvile for simons cone}, \ref{liouvlie for hardtsimonleaf}), we show that the time derivative of the normal graph function vanishes, which implies that $\mathcal{M}$ is a stationary flow. This means that the time slices of $\mathcal{M}$ are equal to its backward limit which is one of the leaves of the Hardt-Simon foliation, thus proving theorem \ref{maintheorem : uniqueness in mean convex case}.\\

The paper is organized as follows. In section \ref{preliminaries}, we collect several preliminaries, and set up notations that will be used frequently throughout the paper. In section \ref{graphical radius estimatesection}, we prove the graphical radius estimate. In section \ref{innerouterestimatesection}, we prove the inner outer estimate. In section \ref{uniqueasymptoticssection}, we prove theorem \ref{maintheorem : uniqueasmptotics}. In section \ref{uniquenesssection}, we prove theorem \ref{maintheorem : uniqueness in mean convex case}. Finally, in sections \ref{innershrinkersection} and \ref{trumpetsection}, we construct two families of $O(n)\times O(m)$ symmetric self-shrinkers with boundary that play a crucial role in the proof of the main theorems.
\section*{Disclosure of AI} Ai (ChatGPT 5.6 Sol / Astra) was used to proofread the manuscript for minor mathematical errors such as routine computation mistakes, and for grammatical and spelling errors. Ai was not used in the development of the mathematical ideas in this paper. All the mathematical ideas and proofs are due to the author. This paper contains no AI generated text.
\section{Preliminaries}\label{preliminaries}
In this section, we collect several concepts and results that will be frequently used.
\subsection{Integral Brakke flow}
We first recall several basic concepts related to integral $n$-Brakke flows. We refer the reader to the notes \cite{schulze2021introduction} for details. A one parameter family of Radon measures $(\mu_t)_{t \in I}$ is an integral $n$-Brakke flow in $\mathbf{R}^{n+1}$ if 
    \begin{itemize}
        \item For almost every $t \in I$, $\mu_t$ is a pushforward of an integral $n$-varifold $V(t)$ with locally bounded first variation, and no generalized boundary. 
        \item For any bounded interval $[t_1, t_2]$, compact $K \subset \mathbf{R}^{n+1}$, 
        \begin{equation}
            \int_{t_1}^{t_2}\int_{K}(1 + |H|^2)d\mu_tdt < \infty.
        \end{equation}
        \item For any $[t_1, t_2] \subset I$, $f \in C^{\infty}_c(\mathbf{R}^{n+1}\times [t_1, t_2] ; \mathbf{R}_+)$, 
        \begin{equation}
            \int f_{t_2}d\mu_{t_2} - \int f_{t_1}d\mu_{t_1} \leq \int_{t_1}^{t_2}\int -|H|^2f_t + \langle H, \nabla f_t\rangle + \partial_tf_td\mu_tdt.
        \end{equation}
    \end{itemize}
Colding - Minicozzi \cite{coldingminicozzi} introduced the notion of entropy which is given by
\begin{equation}
        \lambda(\mu) = \sup_{x_0 \in \mathbf{R}^{n+1}, r > 0}F_{x_0, r}(\mu),
    \end{equation}
    with
    \begin{equation}
        F_{x_0, r}(\mu) = \int_{\mathbf{R}^{n+1}}\frac{1}{(4\pi r^2)^{n/2}}\exp{(-\frac{|x - x_0|^2}{4r^2})}d\mu.
    \end{equation}
    We say that an integral Brakke flow $(\mu_t)_{t \in I}$ has finite entropy if 
    \begin{equation}
\sup_{t \in I}\lambda(\mu_t) < \infty.
    \end{equation}
We can also define the Gaussian density of an integral Brakke flow $\mathcal{M} = (\mu_t)_{t \in I}$ at a point $X_0$ at scale $r > 0$ by
\begin{equation}
        \Theta(X_0, \mathcal{M}, r) = \int \frac{1}{(4\pi r^2)^{n/2}}\exp{(-\frac{|x - x_0|^2}{4r^2})}d\mu_{t_0 - r^2},
\end{equation}
whenever $t_0 - r^2 \in I$. 
Huisken's monotonicity formula \cite{huisken1990asymptotic} implies that
\begin{equation}
    r  \to \Theta(X_0, \mathcal{M}, r) 
\end{equation}
is non-decreasing. In particular, whenever the flow is defined for all negative time, and has bounded entropy, then
\begin{equation}
    \Theta(X_0, \mathcal{M}) = \lim_{r \to 0}\Theta(X_0, \mathcal{M}, r)  \leq \lim_{r \to \infty}\Theta(X_0, \mathcal{M}, r)  = \sup_{t \in (-\infty, 0)}\lambda(\mu_t).
\end{equation}
Using the Gaussian density, we now define the support of $\mathcal{M} = (\mu_t)_{t \in I}$, and the notion of unit regular Brakke flows. Let $\mathcal{M} = (\mu_t)_{t \in I}$ be an integral $n$-Brakke flow in $\mathbf{R}^{n+1}$ with bounded entropy. The support of $\mathcal{M}$ is defined by
    \begin{equation}
        \operatorname{spt}(\mathcal{M}) = \bigcup_{t \in I}\operatorname{spt}(\mathcal{M})_t \times \{t\} = \{ X \ | \ \Theta(X, \mathcal{M}) \geq 1\}.
    \end{equation}
    $\mathcal{M} = (\mu_t)_{t \in I}$ is unit regular if $X \in \operatorname{spt}(\mathcal{M})$ is regular, i.e one can find a spacetime neighborhood of $X$ where $\mathcal{M}$ is smooth, whenever $\Theta(X, \mathcal{M}) = 1$. In other words, the set of all regular points can be defined as
    \begin{equation}
        (\mathcal{M}^{\infty})^{\textup{reg}} = \{ X \in \operatorname{spt}(\mathcal{M}^{\infty}) \ | \ \Theta(X, \mathcal{M}^{\infty}) = 1\}.
    \end{equation}
    Any integral Brakke flow $\mathcal{M}$ arising as a Brakke flow limit of a sequence of smooth flows is unit regular.\\ 

We end this subsection by recalling a strong maximum principle for integral Brakke flows that was proved in \cite{ancientasymptotical}. 
\begin{theorem}[Theorem 3.4 in \cite{ancientasymptotical}]\label{strongmaximumprinciplebrakkeflows} Let $\mathcal{M}_1$ be a smooth mean curvature flow defined in a parabolic ball $P(X_0,r)$, where $X_0 \in \operatorname{spt}\mathcal{M}_1$. Assume that $r>0$ is sufficiently small so that $\operatorname{spt}\mathcal{M}_1$ separates $P(X_0,r)$ into two open connected components, denoted by $U$ and $U'$. Let $\mathcal{M}_2$ be a unit regular integral Brakke flow defined in $P(X_0,r)$, with $X_0 \in \operatorname{spt}\mathcal{M}_2.$ Suppose that the Gaussian density of $\mathcal{M}_2$ at $X_0$ satisfies 
\begin{equation}
    \Theta(X_0, \mathcal{M}_2)<2
\end{equation}
and that \begin{equation} \operatorname{spt}\mathcal{M}_2 \subseteq U \cup \operatorname{spt}\mathcal{M}_1.  \end{equation} Then $X_0$ is a regular point of $\mathcal{M}_2$. Moreover, there exists $\varepsilon>0$ such that \begin{equation} \operatorname{spt}\mathcal{M}_2 \cap P(X_0,\varepsilon) = \operatorname{spt}\mathcal{M}_1 \cap P(X_0,\varepsilon).  \end{equation} 
Here, $P(X_0, \epsilon) = B(x_0, \epsilon_0) \times (t_0 - \epsilon_0^2, t_0]$ for $X_0 = (x_0, t_0)$.
\end{theorem}
\subsection{Minimal quadratic cone}\label{simons cone related result prelim}
In this subsection, we collect some quantitative information related to minimal quadratic cones, especially when $n+m\geq 10$ with $n,m \geq 2$. We will always denote the $O(n)\times O(m)$ symmetric minimal cone by $\Sigma$.\\

We first consider the entropy of the minimal cone. In appendix A in \cite{bernstein2025lower}, the authors give an explicit formula of the entropy of an $O(n) \times O(m)$ symmetric quadratic cone, which is given by
\begin{equation}
    \lambda(C_{n,m}) = \frac{\sigma_{n-1}\sigma_{m-1}}{\sigma_{n+m-2}}(\frac{n-1}{n+m-2})^{\frac{n-1}{2}}(\frac{m-1}{n+m-2})^{\frac{m-1}{2}},
\end{equation}
with
\begin{equation}
    \sigma_{n} = (n+1)\frac{\pi^{\frac{n+1}{2}}}{\Gamma(\frac{n+3}{2})}.
\end{equation}
In particular, for every $n,m \geq 2$, one can directly check that
\begin{equation}
    \lambda(\Sigma) < 2.
\end{equation}
We now recall the definition, and several quantitative properties of the Hardt-Simon foliation associated to the minimal quadratic cone when either $n,m \geq 2$ with $n +m\geq 9$, or $(n,m) = (3,5), (4,4), (5,3)$  \cite{hardt1985area}. One can find these results in section 2 in \cite{edelen2023regularity}. \\

There exists a unique function $\psi_1 : [0, \infty) \to \mathbf{R}_+$ so that it satisfies the initial value problem
    \begin{equation}\label{minimal surface equation}
        \begin{cases}
            \frac{\psi_1''}{1 + (\psi_1')^2} + \frac{n-1}{y}\psi_1' - \frac{m-1}{\psi_1} = 0 \\ \psi_1(0) = 1, \psi_1'(0) = 0.
        \end{cases}
    \end{equation}
    This profile function generates a smooth minimal hypersurface in $\mathbf{R}^{n+m}$ via the map
    \begin{equation}\label{Hardt-Simon foliation for Simons coneeq}
        X : [0, \infty) \times \mathbf{S}^{n-1}\times \mathbf{S}^{m-1} \ni (y, w_1, w_2) \to (yw_1, \psi_1(y)w_2) \in \Sigma^+_1 \subset  \mathbf{R}^{n+m}.
    \end{equation}
    The profile function $\psi_1$ satisfies the asymptotic formula
    \begin{equation}\label{asymptotic behavior of hardsimonleafformula}
        \psi_1(y) = \sqrt{\frac{m-1}{n-1}}y + c_1(n,m)y^{\alpha} + O(y^{\alpha -\delta}) \text{ as }y \to \infty
    \end{equation}
   in $C^2$ for some $c_1(n,m), \delta > 0$, and 
    \begin{equation}\label{the formula for alpha}
        \alpha =  \frac{-(d-2) + \sqrt{d^2 - 8d+8}}{2} \in [-2, -1), \ d = n+m-1.
    \end{equation}
    For each $c > 0$, we define
    \begin{equation}
        \psi_c(x) = c\psi_1(\frac{x}{c}),
    \end{equation}
    which is the profile function of the rescaled minimal hypersurface 
    \begin{equation}
        \Sigma^+_c = c\Sigma^+_1.
    \end{equation}
    One can alternatively parametrize $\Sigma^+_c$ as a normal graph over the quadratic cone. There exists $\hat{\psi}_c : [\sqrt{\frac{m-1}{n+m-2}}c, \infty) \to \mathbf{R}_+$ so that
    \begin{equation}\label{graphical normal parametrization of hardtsimonleaf}
        \Sigma^+_c = \{ v_{n,m}r + \nu_{n,m}\hat{\psi}_c(r) \ | \ (r, w_1, w_2) \in [\sqrt{\frac{m-1}{n+m-2}}c, \infty) \times \mathbf{S}^{n-1} \times \mathbf{S}^{m-1} \}
    \end{equation}
    with
    \begin{equation}
        v_{n,m} = (\sqrt{\frac{n-1}{n+m-2}}w_1, \sqrt{\frac{m-1}{n+m-2}}w_2), \ \nu_{n,m} = (-\sqrt{\frac{m-1}{n+m-2}}w_1, \sqrt{\frac{n-1}{n+m-2}}w_2).
    \end{equation}
    Moreover, $\hat{\psi}_c$ satisfies the asymptotic formula
    \begin{equation}
        \hat{\psi}_c(r) = c_0(n,m)c^{1 - \alpha}r^{\alpha} + O(r^{\alpha - \delta}) \text{ as }r \to \infty
    \end{equation}
    in $C^2$ for $c_0(n,m), \delta > 0$.\\

We now consider the linearization of \eqref{minimal surface equation} at $\psi_1$. For each $u \in C^2([0, \infty))$, the linearized minimal surface equation at $\psi_1$ is given by
\begin{equation}\label{linearized minimal surface equation at leaf}
        \mathcal{L}u = \frac{u''}{1 + (\psi_1')^2} + \frac{n-1}{y}u' + \frac{(m-1)u}{\psi_1^2} - \frac{2\psi_1'\psi_1''u'}{(1 + (\psi_1')^2)^2}.
    \end{equation}
We now state an inverse operator to $\mathcal{L}$ which is essentially constructed in  \cite{stolarski2023existence}. For any $y_* > 0$, $f \in C^0([0, y_*])$ so that
\begin{equation}
        \sup_{y \in [0, y_*]}|\frac{f(y)}{(1 + y)^{\alpha}}| < \infty,
    \end{equation}
one can slightly modify the inverse operator (definition 4.11 in \cite{stolarski2023existence}) to explicitly solve the initial value problem
\begin{equation}\label{inhomogenous initial valueproblem for lineaarized minimal surf}
    \begin{cases}
        \mathcal{L}u = f \\ u(0) = u'(0) = 0.
    \end{cases}
\end{equation}
The unique solution is given by the formula
    \begin{equation}\label{Solvability of linearized operator definition eq}
        u(y) = (\psi_1 - y\psi_1')\int_0^{y}\frac{1}{\phi^2(s)\mathcal{J}(s)}\int_{0}^{s}f(v)\mathcal{J}(v)(\psi_1 - v\psi_1')dvds
    \end{equation}
    where
    \begin{equation}
        \phi(y) = \frac{\psi_1 - y\psi_1'}{\sqrt{1 + (\psi_1')^2}}, \ \mathcal{J}(y) = \frac{y^{n-1}\psi_1^{m-1}}{\sqrt{1 + (\psi_1')^2}}.
    \end{equation}\\
    
We now consider the linearized rescaled mean curvature operator at the minimal quadratic cone. The details of the spectral decomposition of the operator that we will state can be found in section 2 in \cite{huang2025mean}.
\begin{definition}\label{linearized rescaled mean curvature flow equation over the cone}
    Let $\Sigma$ be the $O(n) \times O(m)$ symmetric minimal quadratic cone with $n,m \geq 2$, and $n+m \geq 10.$ For each $u \in C^{\infty}(\mathcal{D})$ with $\mathcal{D} \subset \Sigma \setminus \{0\}$, the linearized rescaled mean curvature operator is given by
    \begin{equation}\label{linearzed RMCF at cone}
        Lu = \Delta_{\Sigma}u - \frac{1}{2}\langle x, \nabla^{\Sigma}u \rangle + (\frac{1}{2} + \frac{n+m-2}{|x|^2})u,
    \end{equation}
    where $x \in \mathcal{D} \subset \Sigma$, $\Delta_{\Sigma}$, $\nabla^{\Sigma}$ denote the Laplacian / covariant derivative with respect to the induced metric on $\Sigma$.
\end{definition}
When analyzing \eqref{linearzed RMCF at cone}, one typically considers the Gaussian weighted $L^2$ space since the operator is symmetric with respect to the weighted inner product.
\begin{definition}\label{gaussian weighted sobolev space}
    The Gaussian weighted $L^2$ space is defined by 
    \begin{equation}
        L^2_{w} = \overline{C_c^{\infty}(\Sigma \setminus \{0\})}^{\|\cdot \|_{L^2_{w}}},
    \end{equation}
    where the norm is given by the inner product
    \begin{equation}
        \langle f, g \rangle_{L^2_{w}} = \int_{\Sigma}f(x)g(x)e^{-\frac{|x|^2}{4}}d\mathcal{H}^{n+m-1}(x).
    \end{equation}
    The Gaussian weighted Sobolev space $H^1_{w}$ is defined analogously with the inner product
    \begin{equation}
        \langle f, g \rangle_{H^1_{w}} = \int_{\Sigma}(fg + \nabla^{\Sigma}f \cdot \nabla^{\Sigma}g)
        e^{-\frac{|x|^2}{4}}d\mathcal{H}^{n+m-1}(x).
    \end{equation}
\end{definition}
By using the fact $\Sigma$ is strictly stable when $n +m \geq 10$, the operator \eqref{linearzed RMCF at cone} is coercive, i.e it satisfies
\begin{equation}\label{spectral decomposition}
     -\langle Lu, u \rangle_{L^2_w} \geq \epsilon_0\|\nabla u\|_{L^2_w}^2 - \frac{1}{\epsilon_0}\|u\|_{L^2_w}^2
 \end{equation}
 for some $\epsilon_0(n,m) > 0$. As a result, we can find $\lambda_1 > \lambda_2 > ... \to -\infty$, and a countable orthonormal basis of $L^2_w$ consisting of smooth functions in $\Sigma \setminus \{0\}$ denoted by $\{\varphi_{i,j}\}_{i \in \mathbf{N}, \ 1 \leq j \leq J(i)< \infty}$ so that
    \begin{equation}
        L\varphi_{i,j} = \lambda_i\varphi_{i,j},
    \end{equation}
    with $L$ given in \eqref{linearzed RMCF at cone}. $J(i)$ is the multiplicity of $\lambda_i$. \\

For the sake of our purpose, we explicitly state the first four eigenvalues and corresponding eigenfunctions in the case when $n+m \geq 10$. The first four eigenvalues are 
    \begin{equation}\label{first four eigenvalues eq}
        \lambda_1 = \frac{1 - \alpha}{2}, \ \lambda_2 = \frac{1}{2}, \lambda_3 = \frac{-1 - \alpha}{2}, \lambda_4 = 0,
    \end{equation}
    and $\lambda_i \leq \lambda_5(n,m) < 0$ for $i \geq 5$. As for the eigenfunctions, for each $x \in \Sigma \setminus \{0\}$, let $\nu_{\Sigma}(x)$ be a smooth unit normal vector field. Set $(e_j)_{j = 1}^{n+m}$ to be the standard $j^{\textup{th}}$ unit vector in $\mathbf{R}^{n+m}$. For each $i = 1,2,3,4$, the corresponding eigenfunctions are as follows. 
  \begin{itemize}\label{eigenfunctions list}
  \item $i = 1$ : $J(1) = 1$, $\varphi_1(x) =\varphi_{1,1}(x) = c(n,m)|x|^{\alpha}$.
  \item $i = 2$ : $J(i) = n+m$, $\varphi_{2,j}(x) = c(n,m)\nu(x) \cdot e_j$.
  \item $i = 3$ : $J(3) = 1$, $\varphi_{3}(x) = \varphi_{3,1}(x) = c(n,m)|x|^{\alpha}(1 - \frac{1}{2(n+m-1 + 2\alpha)}|x|^2)$.
  \item $i = 4$ : $J(4) = nm$. By associating each index $1 \leq j \leq nm$ with a pair $(s,v)$ with $1 \leq s \leq n$, $1 \leq v \leq m$, $\varphi_{4,j}(x) = c(n,m)\frac{(x\cdot e_{s})(x\cdot e_{n +v})}{r}$.
  \end{itemize}
  Here, $c(n,m) > 0$ are normalizing constants that depend on the eigenfunction, and $\alpha \in (-2,-1)$ is given by \eqref{the formula for alpha}.\\

  We end this subsection by stating two Liouville type theorems for ancient solutions to linear parabolic equations related to the minimal quadratic cone $\Sigma$, and the smooth minimal surface $\Sigma^+_1$ introduced in \eqref{Hardt-Simon foliation for Simons coneeq}. These theorems are simple modifications of theorem 5.2, and corollary 5.5 in \cite{stolarski2023existence}. 
\begin{proposition}\label{liouvile for simons cone}
    Let $n,m \geq 2$ with $n+m \geq 10$, and let $\Sigma$ be the $O(n)\times O(m)$ symmetric minimal quadratic cone. Let $u \in C^{\infty}(\Sigma \setminus \{0\} \times (-\infty, 0))$ so that
    \begin{equation}
        \partial_tu_t = \Delta_{\Sigma}u_t + |A_{\Sigma}|^2u_t,
    \end{equation}
    and
    \begin{equation}\label{spatial decay for liouvile theorem simons cone}
        u_t \geq 0, \ \ \sup_{t < 0}\sup_{x \in \Sigma \setminus \{0\}}|u_t(x)||x|^{|\alpha| + \delta} < \infty
    \end{equation}
    for some $\delta \in (0, 3)$. Then $u \equiv 0$. Here, $\alpha\in (-2, -1)$ is given in \eqref{the formula for alpha}.
\end{proposition}
\begin{proof}[Proof of proposition \ref{liouvile for simons cone}]
    Parameterizing $\Sigma \setminus \{0\}$ via the map
    \begin{equation}
        (0, \infty) \times \mathbf{S}^{n-1} \times \mathbf{S}^{m-1} \ni (y, w_1, w_2) \to (yw_1, \sqrt{\frac{m-1}{n-1}}yw_2) \in \mathbf{R}^{n+m},
    \end{equation}
    we can define the `averaged' function $\hat{u}_t$ as
    \begin{equation}
        \hat{u}_t(y) = \fint_{\mathbf{S}^{n-1} \times \mathbf{S}^{m-1}}u_t(y,w_1,w_2)d\sigma^n(w_1)d\sigma^m(w_2),
    \end{equation}
    where $d\sigma^n$ is the standard spherical measure on $\mathbf{S}^{n-1}$. Noting that rotations along $\mathbf{S}^{n-1} \times \mathbf{S}^{m-1}$ are isometries from $\Sigma$ to itself, we see that
    \begin{equation}
        \partial_t\hat{u}_t = \Delta_{\Sigma}\hat{u}_t + |A_{\Sigma}|^2\hat{u}_t.
    \end{equation}
    Moreover the assumption \eqref{spatial decay for liouvile theorem simons cone} translates to 
    \begin{equation}\label{spatial decay of reduced function}
        \sup_{t < 0}\sup_{y > 0}|\hat{u}_t(y)||y|^{|\alpha| + \delta} < \infty
    \end{equation}
    for some $\delta \in (0, 3)$. If $\hat{u}_t \equiv 0$, then since $u_t \geq 0$, we can conclude that $u_t \equiv 0$, thus proving proposition \ref{liouvile for simons cone}.\\

    To prove that $\hat{u}_t \equiv 0$, we follow the proof of corollary 5.5 in \cite{stolarski2023existence}. By following the computations in appendix A in \cite{stolarski2023existence}, we see that 
    \begin{equation}
        \partial_t\hat{u}_t(y) = \frac{n-1}{n+m-2}\hat{u}_t''(y) + \frac{n-1}{y}\hat{u}_t'(y) + \frac{n-1}{y^2}\hat{u}_t(y) \text{ for }(y,t) \in (0, \infty) \times (-\infty, 0).
    \end{equation}
    By defining a new time parameter $s = \frac{2(n-1)}{n+m-2}t<0$, we have
    \begin{equation}
        \partial_s\hat{u}_s = \frac{1}{2}\hat{u}_s'' + \frac{N-1}{y}\hat{u}_s' + \frac{N-1}{y^2}\hat{u}_s
    \end{equation}
    with $N = \frac{n+m}{2}$. By exactly following the arguments in section 5.2 in \cite{stolarski2023existence}, we see that \eqref{spatial decay of reduced function} together with $n + m \geq 10$ implies that $\hat{u}\equiv 0$, thus proving proposition \ref{liouvile for simons cone}.
\end{proof}
\begin{proposition}\label{liouvlie for hardtsimonleaf}
     Let $n,m \geq 2$ with $n+m \geq 10$. Let $\Sigma^+_1$ be a leaf of the Hardt-Simon foliation defined in \eqref{Hardt-Simon foliation for Simons coneeq}. Let $u \in C^{\infty}(\Sigma^+_1 \times (-\infty, 0))$ so that 
    \begin{equation}
        \partial_tu_t = \Delta_{\Sigma^+_1}u_t + |A_{\Sigma^+_1}|^2u_t,
    \end{equation}
    and
    \begin{equation}\label{decay at infinity forhardsimonleafliouviel}
        u_t \geq 0, \ \ \sup_{t < 0}\sup_{x \in \Sigma^+_1}|u_t(x)|(1 + |x|)^{a} < \infty
    \end{equation}
    for some $a > |\alpha|$. Then $u \equiv 0$. Here, $\alpha\in (-2,-1) $ is given by \eqref{the formula for alpha}.
\end{proposition}
\begin{proof}[Proof of proposition \ref{liouvlie for hardtsimonleaf}]
    Parameterizing $\Sigma^+_1$ via the map
    \begin{equation}
        [0, \infty) \times \mathbf{S}^{n-1}\times \mathbf{S}^{m-1} \ni (y, w_1, w_2) \to (yw_1, \psi_1(y)w_2) \in \mathbf{R}^{n+m},
    \end{equation}
    where $\psi_1$ is the solution to the initial value problem \eqref{minimal surface equation}, we can define the `averaged' function $\hat{u}_t$ as
    \begin{equation}
        \hat{u}_t(y) = \fint_{\mathbf{S}^{n-1} \times \mathbf{S}^{m-1}}u_t(y,w_1, w_2)d\sigma^n(w_1)d\sigma^m(w_2),
    \end{equation}
    where $d\sigma^n$ is the standard spherical measure. Noting that rotations along $\mathbf{S}^{n-1} \times \mathbf{S}^{m-1}$ are isometries from $\Sigma^+_1$ to itself, we see that 
    \begin{equation}
        \partial_t\hat{u}_t = \Delta_{\Sigma^+_1}\hat{u}_t + |A_{\Sigma^+_1}|^2\hat{u}_t.
    \end{equation}
    Moreover the assumption \eqref{decay at infinity forhardsimonleafliouviel} translates to 
    \begin{equation}
        \sup_{t < 0}\sup_{ y \geq 0}|\hat{u}_t(y)|(1 + y)^{a} < \infty
    \end{equation}
    for some $a > |\alpha|$. Then one can follow the same line of arguments as in theorem 5.2 in \cite{stolarski2023existence} with the changed numbers and show that $\hat{u}_t \equiv 0$. Since $u \geq 0$, this implies that $u \equiv 0$, thus proving proposition \ref{liouvlie for hardtsimonleaf}.
    \end{proof}
\subsection{Domains and norms}
In this subsection, we define notations for specific domains and weighted norms that we will frequently use throughout the paper. \\

We first define a special form of subdomains of the minimal quadratic cone which will be frequently used. 
\begin{definition}\label{definiton of conical domain}
    Let $\Sigma$ be the $O(n)\times O(m)$ symmetric minimal quadratic cone. For each $0 < r < L \leq \infty$, we define
    \begin{equation}
        \Sigma_{r, L} = \Sigma 
        \cap (B(0, L) \setminus B(0, r)).
    \end{equation}
    Here, $B(0, r)$ is an open ball in $\mathbf{R}^{n+m}$ of radius $r > 0$ centered at the origin. 
\end{definition}

We now define the weighted $C^k$ norm for functions over $\Sigma$. 
\begin{definition}\label{Weighted Ck norms}
    For any $0 < r < R \leq \infty$, $u \in C^k(\Sigma_{r, R})$, the weighted $C^k$ seminorm is defined by
    \begin{equation}\label{definition of weighted Ck seminorm}
        [u]^*_{C^k(\Sigma_{r, R})} = \sup_{x \in \Sigma_{r, R}}|x|^{k - 1}|\nabla^ku(x)|,
    \end{equation}
    where $\nabla$ denotes the covariant derivative with respect to the induced metric on $\Sigma$. 
    The weighted $C^k$ norm is defined by 
    \begin{equation}\label{definition of weighted ck norm}
        \|u\|^*_{C^k(\Sigma_{r, R})} = \sum_{i = 0}^k[u]^*_{C^i(\Sigma_{r, R})}.
    \end{equation}
    $\Sigma_{r,R}$ is given by definition \ref{definiton of conical domain}.
\end{definition}
One can directly check that for the rescaling
\begin{equation}
    u_{\lambda}(x) = \lambda^{-1} u(\lambda x),
\end{equation}
the seminorms given by \eqref{definition of weighted Ck seminorm} are invariant i.e
\begin{equation}
    [u_{\lambda}]^*_{C^k(\Sigma_{r,R})} = [u]^*_{C^k(\Sigma_{\lambda r, \lambda R})}.
\end{equation}
The weighted seminorms also satisfy the following form of interpolation inequalities.
\begin{proposition}\label{interpolation inequalities of weighted Ck norms}
    Following the notations in definition \ref{Weighted Ck norms}, for any $0 < j < k$, $0 < \delta_1 < \delta_2 < L_2 < L_1 < \infty$, there exists a constant $c= c(j,k,n,m, \delta_1, \delta_2, L_1, L_2) > 0$ so that 
    \begin{equation}
        [u]^*_{C^j(\Sigma_{\delta_2r, L_2r})} \leq c(([u]^*_{C^k(\Sigma_{\delta_1r, L_1r})})^{\frac{j}{k}}(\|u\|^*_{C^0(\Sigma_{\delta_1r, L_1r})} )^{1 - \frac{j}{k}} + \|u\|^*_{C^0(\Sigma_{\delta_1r, L_1r})} )
    \end{equation}
    for any $r > 0$. In case $L_1 = L_2 = \infty$, we can find some different constant $c = c(j,k,n,m,\delta_1, \delta_2)>0$ so that the above inequality holds with $L_1 = L_2 = \infty$. 
\end{proposition}
\begin{proof}[Proof of proposition \ref{interpolation inequalities of weighted Ck norms}]
   By considering 
    \begin{equation}
        u^r(x) = \frac{1}{r}u(rx), \ x\in \Sigma_{\delta_1, L_1},
    \end{equation}
    the scaling invariance of the seminorms implies that it is enough to prove the case when $r = 1$. We now choose any $x \in \Sigma_{\delta_2, L_2}$, and define
    \begin{equation}
        u^x(y) = \frac{1}{|x|}u(|x|y), \ y \in \Omega,
    \end{equation}
    where
    \begin{equation}
        \Omega = 
        \begin{cases}
            \Sigma_{\frac{\delta_1 }{ \delta_2}, \frac{L_1}{L_2}} \text{ if }L_2 < L_1 < \infty \\ 
             \Sigma_{\frac{\delta_1 }{ \delta_2}, 2} \text{ if }L_1 = L_2 = \infty.
        \end{cases}
    \end{equation}
    Then by standard interpolation inequalities of the unweighted norms, we can find $c > 0$ with the prescribed dependence so that we have
    \begin{align*}
        |x|^{-1 + j}|\nabla^ju|(x) = |\nabla^ju^x|(\frac{x}{|x|}) &\leq c([u^x]_{C^k(\Omega)}^{\frac{j}{k}}\|u^x\|_{C^0(\Omega)}^{1 - \frac{j}{k}} + \|u^x\|_{C^0(\Omega)}) \\ & \leq c(([u]^*_{C^k(\Sigma_{\delta_1, L_1})})^{\frac{j}{k}}(\|u\|^*_{C^0(\Sigma_{\delta_1, L_1}))} )^{1 - \frac{j}{k}} + \|u\|^*_{C^0(\Sigma_{\delta_1, L_1}))} ),
    \end{align*}
    where the last inequality is due to the definition of the weighted norm, and the choice of domain $\Omega$. Taking the supremum over all $x \in \Sigma_{\delta_2, L_2}$ gives us the desired interpolation inequality.
\end{proof}
\section{Graphical radius estimate}\label{graphical radius estimatesection}
In this section, we consider a smooth, properly embedded ancient mean curvature flow $(M_t)_{t \in (-\infty, 0)}$ which satisfies assumptions \ref{strong convergence to simonscone assumption1}, \ref{additional geoemtric assumption}. The main goal of this section is to establish an optimal graphical radius estimate (proposition \ref{graphical radius estimate proposition}). This will be achieved by a pseudolocality theorem (proposition \ref{pseudolocality theorem}), together with the comparison principle using the inner self-shrinkers we construct in proposition \ref{inner family of self shrinkers}.\\

We first establish a local smooth convergence of the rescaled mean curvature flow to the minimal quadratic cone away from the origin, and an explicit entropy bound. 
\begin{lemma}\label{local smooth convergence away from origin lemma}
    Let $(M_t)_{t < 0}$ be a smooth, properly embedded ancient mean curvature flow which satisfies assumption \ref{strong convergence to simonscone assumption1}. Define the rescaled mean curvature flow $\Tilde{\mathcal{M}} = (\Tilde{M}_{\tau})_{\tau \in (-\infty, 0)}$ to be
    \begin{equation}
        \Tilde{M}_{\tau} = e^{\frac{\tau}{2}}M_{-e^{-\tau}}.
    \end{equation}
    Then
    \begin{equation}\label{local smooth convergence away from origin}
        \Tilde{M}_{\tau} \to \Sigma \text{ locally smoothly in }\mathbf{R}^{n+m}\setminus \{0\} \text{ as }\tau \to -\infty.
    \end{equation}
    Moreover, \begin{equation}\label{entropy bound proposition eq}
        \lambda(M_t) = \lambda(\Tilde{M}_{\tau}) \leq \lambda(\Sigma) < 2
    \end{equation}
    for any $t = -e^{-\tau} \in (-\infty, 0)$.
\end{lemma}
\begin{proof}[Proof of lemma \ref{local smooth convergence away from origin lemma}]
Local smooth convergence away from the origin is immediate from White's local regularity theorem. Indeed, for any $\tau_i \to -\infty$, if we set $\lambda_i = e^{\frac{\tau_i}{2}} \to 0$, then
\begin{equation}
    \Tilde{M}_{\tau_i} = \lambda_iM_{-\lambda^{-2}_i}= M^i_{-1}.
\end{equation}
Since the flow
\begin{equation}
    \sqrt{-t}\Sigma \equiv \Sigma
\end{equation}
is smooth in $\mathbf{R}^{n+m}\setminus \{0\} \times (-\infty, 0)$, 
assumption \ref{strong convergence to simonscone assumption1} together with White's local regularity theory \cite{White2005-nl} implies that
\begin{equation}
   \Tilde{M}_{\tau_i} = M^i_{-1} \to \Sigma \text{ locally smoothly in }\mathbf{R}^{n+m}\setminus \{0\}.
\end{equation}
Since $\tau_i \to -\infty$ is arbitrary, we have \eqref{local smooth convergence away from origin}.\\

We now prove \eqref{entropy bound proposition eq}. Take any $x_0 \in \mathbf{R}^{n+m}$, $t_0 < 0$, $r > 0$, and $\lambda_i \to 0$. Set $d = n+m-1$. Choose any large $L > 0$. By (localized) Huisken's monotonicity formula \cite{huisken1990asymptotic}, we have for any $\lambda_i > 0$ small so that $-\lambda_i^{-2} < t_0$, 
\begin{align*}
   &\int_{M_{t_0}}\frac{1}{(4\pi r^2)^{d/2}}\exp(-\frac{|x - x_0|^2}{4r^2})(1 - \frac{|x-x_0|^2}{L\lambda_i^{-2}})_+^3d\mathcal{H}^{d}(x) \\&\leq\int_{M_{t_0}}\frac{1}{(4\pi r^2)^{d/2}}\exp(-\frac{|x - x_0|^2}{4r^2})(1 - \frac{|x-x_0|^2- 2dr^2}{L\lambda_i^{-2}})_+^3d\mathcal{H}^{d}(x)  \\ &\leq \int_{M_{-\lambda_i^{-2}}}\frac{1}{(4\pi(t_0 + r^2 + \lambda_i^{-2}))^{d/2}}\exp{(-\frac{|x - x_0|^2}{4(t_0 + r^2 + \lambda_i^{-2})})}(1 - \frac{|x-x_0|^2 - 2d(\lambda_i^{-2} + r^2 + t_0)}{L\lambda_i^{-2}})_+^3d\mathcal{H}^{d}(x)\\ & = \int_{M^i_{-1}}\frac{1}{(4\pi(1 + r_i))^{d/2}}\exp(-\frac{|x - x_i|^2}{4(1 + r_i)})(1 + \frac{2d}{L} + \frac{2dr_i}{L} - \frac{|x - x_i|^2}{L})_+^3d\mathcal{H}^{d}(x) = I_i,
\end{align*}
with $r_i = \lambda_i^2(t_0 + r^2)$, $x_i = \lambda_ix_0$, and
\begin{equation}
    M^i_t = \lambda_iM_{\lambda^{-2}_it}.
\end{equation}
Since the integrand in the previous integral is compactly supported on some fixed large ball for large enough $i$, and converges uniformly to
\begin{equation}
    \frac{1}{(4\pi)^{d/2}}\exp(-\frac{|x|^2}{4})(1 + \frac{2d}{L} - \frac{|x|^2}{L})_+^3,
\end{equation}
by assumption \ref{strong convergence to simonscone assumption1}, we have
\begin{equation}
    \lim_{i\to \infty}I_i = \int_{\Sigma}\frac{1}{(4\pi)^{d/2}}\exp(-\frac{|x|^2}{4})(1 + \frac{2d}{L} - \frac{|x|^2}{L})_+^3d\mathcal{H}^{d}(x).
\end{equation}
At the same time, by Fatou's lemma, 
\begin{equation}
    F_{x_0, r}(M_{t_0}) \leq \liminf_{i\to \infty}\int_{M_{t_0}}\frac{1}{(4\pi r^2)^{d/2}}\exp(-\frac{|x - x_0|^2}{4r^2})(1 - \frac{|x-x_0|^2}{L\lambda_i^{-2}})_+^3d\mathcal{H}^{d}(x).
\end{equation}
Therefore, we obtain for each $r > 0$, $x_0 \in \mathbf{R}^{n+m}$, $t_0 < 0$, $L > 0$, 
\begin{equation}
    F_{x_0, r}(M_{t_0}) \leq \int_{\Sigma}\frac{1}{(4\pi)^{d/2}}\exp(-\frac{|x|^2}{4})(1 + \frac{2d}{L} - \frac{|x|^2}{L})_+^3d\mathcal{H}^{d}(x).
\end{equation}
Thus, letting $L \to \infty$, and taking the supremum in $(x_0, r)$ implies that
\begin{equation}
    \sup_{t < 0}\lambda(M_t) \leq \lambda(\Sigma) < 2.
\end{equation}
Since entropy is invariant under rescaling, we obtain \eqref{entropy bound proposition eq}.
\end{proof}
Following the arguments in lemma 7.18 in \cite{Chodosh2024-qv}, we now establish rough graphicality of the ancient flow over the minimal quadratic cone away from the origin.
\begin{proposition}\label{Rough graphicality property}
    Let $(M_t)_{t \in (-\infty, 0)}$ be a smooth, properly embedded rescaled mean curvature flow that satisfies assumption \ref{strong convergence to simonscone assumption1}, and let $(\Tilde{M}_{\tau})_{\tau \in (-\infty, 0)}$ be the rescaled mean curvature flow.
    There exists $\overline{\eta}(n,m) \in (0, \frac{1}{100})$ so that for every $\eta \in (0, \overline{\eta}(n,m))$, $r > 0$, there exists $ \overline{\tau} < 0$ so that for all $\tau \leq \overline{\tau}$, $\Tilde{M}_{\tau} \setminus B(0, r)$ can be written as a normal graph of a smooth function $u_{\tau} \in C^{\infty}(\mathcal{D}_{\tau})$. Moreover, $\Sigma \setminus B(0, 2r) 
    \subset \mathcal{D}_{\tau}$, and $\|u_{\tau}\|^*_{C^3(\Sigma \setminus B(0, 2r))} \leq \eta$. Here, the weighted $C^3$ norm $\|u\|^*_{C^3}$ is given by definition \ref{Weighted Ck norms}.
\end{proposition}
One corollary of proposition \ref{Rough graphicality property} is that $(\Tilde{M}_{\tau})_{\tau \in (-\infty, 0)}$ is connected. 
\begin{corollary}\label{connectedness}
    Any smooth, properly embedded mean curvature flow $(M_t)_{t \in (-\infty, 0)}$ which satisfies assumption \ref{strong convergence to simonscone assumption1} is connected.
\end{corollary}
\begin{proof}[Proof of corollary \ref{connectedness}]
It is enough to prove connectedness of the rescaled mean curvature flow $(\Tilde{M}_{\tau})_{\tau \in (-\infty, 0)}$. 
Since the flow is smooth, it is enough to show connectedness for sufficiently negative time slices. Suppose this is not true. Let $\Tilde{M}^{main}_{\tau}$ be the component which is graphical over $\Sigma$ outside a compact set. Note that such a component exists due to proposition \ref{Rough graphicality property}. If we let $N_{\tau}$ be any other component, then proposition \ref{Rough graphicality property} implies that for sufficiently negative $\tau$, $N_{\tau} \subset B(0, 1)$. Note that the sphere $\partial B(0,1)$ contracts down to the origin in finite time under rescaled mean curvature flow. Therefore if we denote by $\tau_0(n,m) > 0$ to be the time required for $\partial B(0,1)$ to contract to a point, then the avoidance principle implies that $N_s$ cannot be smooth for all $s \in [\tau, \tau + \tau_0]$. By taking $\tau \leq -2\tau_0$ sufficiently negative, we see that the flow $(\Tilde{M}_{\tau})_{(-\infty, 0)}$ cannot be smooth which is a contradiction. This proves corollary \ref{connectedness}.
\end{proof}
\begin{proof}[Proof of proposition \ref{Rough graphicality property}]
    We follow the arguments in lemma 7.18 of \cite{Chodosh2024-qv}. For each $\tau_0 < 0$, define the smooth mean curvature flow 
    \begin{equation}
        S^{\tau_0}_t = \sqrt{-t}\Tilde{M}_{\tau_0 - \ln (-t)}, \ t < -e^{\tau_0}.
    \end{equation}
    Then due to lemma \ref{local smooth convergence away from origin lemma}, $S^{\tau_0}_t$ converges locally smoothly in $\mathbf{R}^{n+m} \setminus \{0\} \times (-\infty, 0)$ to a self similar flow $\sqrt{-t}\Sigma \equiv \Sigma$ as $\tau_0 \to -\infty$. This implies that for any $\epsilon > 0$, $\xi \in (0,1)$, one can find $\overline{\tau}_0 < 0$ so that for any $\tau_0 \leq \overline{\tau}_0$, $S^{\tau_0}_{t} \cap (B(0, 3) \setminus B(0, \frac{1}{2}))$ is graphical over $\Sigma$ with $C^{10}$ norm bounded by $\xi$ for all $t \in [-1, -\epsilon]$. By using pseudolocality of mean curvature flow (theorem 1.5 in \cite{pseudolocality}), we see that by choosing $\epsilon = \epsilon(n, m,\eta) > 0$, $\xi = \xi(n, m,\eta) > 0$, we can ensure that for all $\tau_0 \leq \overline{\tau}_0$, $t \in [-1, -e^{\tau_0}]$,  
    \begin{equation}
        S^{\tau_0}_t \cap (B(0, 2) \setminus B(0, 1))
    \end{equation}
    is a normal graph over $\Sigma$ with $C^3$ norm bounded by $\frac{\eta}{100}$. By choosing $\overline{\eta}(n,m) \in (0, \frac{1}{100})$ small enough, and rescaling back to $\Tilde{M}_{\tau}$, we see that for any $\tau \in [\tau_0, 0)$, $\tau_0 \leq \overline{\tau}_0$, 
    \begin{equation}\label{graphicality of rescaled mean curvautre flow on annulus regions}
        \Tilde{M}_{\tau} \cap (B(0, 2e^{\frac{\tau - \tau_0}{2}}) \setminus B(0, e^{\frac{\tau - \tau_0}{2}}))
    \end{equation}
    is graphical over $\Sigma$, with the weighted $C^3$ norm of the graph function bounded by $\eta$. By fixing any $\tau \leq \overline{\tau}_0$, and varying $\tau_0 \in (-\infty, \tau]$, we can paste together the regions given in \eqref{graphicality of rescaled mean curvautre flow on annulus regions} and conclude that there exists $\overline{\tau}_0 < 0$ so that for all $\tau \leq \overline{\tau}_0$, $\Tilde{M}_{\tau} \setminus B_1$ is graphical over $\Sigma$ with the weighted $C^3$ norm bounded by $\eta > 0$. \\

    As for $\Tilde{M}_{\tau}\cap (B(0,1) \setminus B(0, r))$, we use assumption \ref{strong convergence to simonscone assumption1} directly. This implies that there exists $\overline{\tau}_1 < 0$ so that for all $\tau \leq \overline{\tau}_1$, $\Tilde{M}_{\tau}\cap (B(0,10) \setminus B(0, r))$ is graphical over $\Sigma$ with weighted $C^3$ norm bounded by $\eta > 0$. By combining this with the previous conclusion, we see that for all $\tau \leq \overline{\tau} = \min\{\overline{\tau}_0, \overline{\tau}_1\} < 0$, $\Tilde{M}_{\tau} \setminus B(0, r)$ is graphical over $\Sigma$ with weighted $C^3$ norm bounded by $\eta$. Denoting the domain of the graph function by $\mathcal{D}_{\tau}$, we may further conclude that $\Sigma \setminus B(0, 2r) \subset \mathcal{D}_{\tau}$ 
    for all $\tau \leq \overline{\tau}$ by possibly shrinking $\overline{\eta}(n,m) \in (0, \frac{1}{100})$. This proves proposition \ref{Rough graphicality property}.
\end{proof}
Proposition \ref{Rough graphicality property} implies that the rescaled flow can be expressed using the graph function $u_{\tau}$ outside any small compact set containing the origin for all sufficiently negative times. However, in order to obtain unique asymptotics, we need a more precise estimate on the size of the compact set for $\textit{each}$ $\tau$. To do so, we need the following version of pseudolocality of mean curvature flow lying on one side of the minimal quadratic cone. 
\begin{proposition}\label{pseudolocality theorem}
    Let either $n, m\geq 2$ with $n+m \geq 9$, or $(n,m) = (3,5), (4,4), (5,3)$, and let $\Sigma$ be the $O(n)\times O(m)$ symmetric minimal quadratic cone in $\mathbf{R}^{n+m}$. For any given $\eta > 0, \ k \geq 2, \ l > 0, \ 0 < \delta <1 < 10 < L$, there exist $\epsilon = \epsilon(n,m,\eta, k, l, \delta, L) > 0$, and $K = K(n,m,\eta, k, l, \delta, L) > 10$ so that the following holds: Let $r > 0$, and suppose a smooth, properly embedded mean curvature flow $(M_t)_{t \in [-2lr^2, 2lr^2]}$ satisfies the entropy bound
    \begin{equation}\label{entropy bound proof of pseudolocality}
        \lambda(M_t) \leq \lambda(\Sigma) < 2
    \end{equation}
    for all $t  \in [-2lr^2, 2lr^2]$, $(M_t)_{t \in [-2lr^2, 2lr^2]}$ lies on one side of $\Sigma$, 
     and $M_t \cap (B(0,{Kr}) \setminus B(0, r))$ is graphical over $\mathcal{D}_t \subset \Sigma$ for all $t \in [-2lr^2, 0]$ with 
    \begin{equation}
       \Sigma_{2r, (K-1)r} \subset \mathcal{D}_t, \ \|u_t\|^{*}_{C^1( \Sigma_{2r, (K-1)r})} \leq \epsilon, \ \|u_t\|^{*}_{C^2( \Sigma_{2r, (K-1)r})} \leq \eta.
    \end{equation}
    Then for any $t \in [-lr^2, lr^2]$, $M_t\cap (B(0, Lr) \setminus B(0, \delta r))$ is graphical over $\mathcal{D}_t \subset \Sigma$ with 
    \begin{equation}
        \Sigma_{2\delta r, (L-1)r} \subset \mathcal{D}_t , \ \|u_t\|^*_{C^k(\Sigma_{2\delta r, (L-1)r})} \leq \eta.
    \end{equation}
    The norms are given in definition \ref{Weighted Ck norms}.
\end{proposition}
\begin{remark}
    The usual pseudolocality theorem in mean curvature flow (e.g. theorem 1.5 in \cite{pseudolocality}) controls the flow in a smaller domain at later time provided good control at a larger domain at the initial time. Proposition \ref{pseudolocality theorem} on the other hand allows us to control the flow in regions near the origin at later times just by assuming closeness in regions far from the origin at prior time. 
\end{remark}
\begin{proof}[Proof of proposition \ref{pseudolocality theorem}]
    Since proposition \ref{pseudolocality theorem} is scaling invariant, it is enough to prove it for $r = 1$. Suppose the theorem is false. Then we can find a sequence of smooth, properly embedded mean curvature flows $\mathcal{M}^j = (M^j_t)_{t \in [-2l, 2l]}$, and $t_j \in [-l, l] $ so that the following hold.
    \begin{itemize}
        \item $\sup_{|t| \leq 2l}\lambda(M^j_t) \leq \lambda(\Sigma) < 2$. 
        \item $M^j_t$ lies on one side of $\Sigma$ for all $|t| \leq 2l$.
        \item $M^j_t \cap (B(0, j) \setminus  B(0, 1))$ is graphical over $\Sigma_{2, (j-1)} \subset \mathcal{D}^j_t$ for all $t \in [-2l, 0]$. 
        \item $\|u^j_t\|_{C^1(\Sigma_{2, (j-1)})}^* \leq \frac{1}{j}, \ |u^j_t\|_{C^2(\Sigma_{2, (j-1)})}^* \leq \eta$ for all $t  \in [-2l, 0]$. 
        \item Either $M^j_{t_j}\cap (B(0, L) \setminus  B(0, \delta ))$ is not graphical over $\Sigma$, or $\Sigma_{2\delta , (L-1)} \not\subset \mathcal{D}^j_{t_j}$, or $\|u^j_{t_j}\|^*_{C^k(\Sigma_{2\delta , (L-1)})} \geq \eta$.
    \end{itemize}
    Set 
    \begin{equation}\label{two components of complement of cone}
    \mathbf{R}^{n+m} \setminus \Sigma = U^+ \cup U^-
\end{equation}
with
\begin{equation}
    U^+ = \{ (x,y) \in \mathbf{R}^n\times \mathbf{R}^m \ | \ |y| > \sqrt{\frac{m-1}{n-1}}|x|\}, \  U^- = \{ (x,y) \in \mathbf{R}^n\times \mathbf{R}^m \ | \ |y| < \sqrt{\frac{m-1}{n-1}}|x|\}.
\end{equation}
Since $\mathcal{M}^j$ are smooth and lie on one side of $\Sigma$, after possible reflection and swapping $(n,m)$, we may without loss of generality further assume that
\begin{equation}\label{one sidedness of sequence of mean curvauture flow}
    \operatorname{spt}(\mathcal{M}^j) \subset U^+ \times [-2l, 2l].
\end{equation}
    Due to the uniform entropy bound, by possibly passing to a subsequence, there exists a unit-regular, integral $(n+m-1)$-Brakke flow $\mathcal{M}^{\infty} = (\mu^{\infty}_t)_{t \in [-2l, 2l]}$ so that $\mathcal{M}^j\to \mathcal{M}^{\infty}$ as integral Brakke flows defined over $I = [-2l, 2l]$ (see theorem 5.8 in \cite{schulze2021introduction}). \\

    We show that for each $|t| \leq \frac{3}{2}l$, 
    \begin{equation}
        \mu^{\infty}_t \equiv \mathcal{H}^{n+m-1}|_{\Sigma}.
    \end{equation}
The assumption 
\begin{equation}\label{good c1c2bound onsequence of flows}
    \|u^j_t\|_{C^1(\Sigma_{2, (j-1)})}^* \leq \frac{1}{j}, \|u^j_t\|_{C^2(\Sigma_{2, (j-1)})}^* \leq \eta \text{ for all }
    t  \in [-2l, 0],
\end{equation}
together with pseudolocality of mean curvature flow (theorem 1.5 in \cite{pseudolocality}), and interior regularity theory \cite{Ecker1991InteriorEF} implies that there exists $R = R(l, \eta, n,m) > 0$ so that for $j$ large, $(M^j_t \cap (B(0, \frac{j}{2})\setminus B(0, R)))_{t \in [-\frac{7}{4}l, \frac{7}{4}l]}$ is a normal graph over $\Sigma$ with the normal graph function having uniformly bounded weighted $C^q$ norm for each $q$ on each compact set. This implies that outside $B(0, R)$, $t \in [-\frac{3}{2}l, \frac{3}{2}l]$, $\mathcal{M}^{\infty}$ is a smooth mean curvature flow that is a normal graph over $\Sigma$. Also note that since we have
\begin{equation}
    \sup_{t \in [-2l, 0]}\|u^j_t\|_{C^1(\Sigma_{2, (j-1)})}^* \leq \frac{1}{j},
\end{equation}
we see that outside $B(0, R)$, $t \in [-\frac{3}{2}l, 0]$, $\mathcal{M}^{\infty}$ is identical to the multiplicity one minimal quadratic cone. This in particular implies that for each $|t| \leq \frac{3}{2}l$, 
\begin{equation}
    \operatorname{spt}(\mathcal{M}^{\infty})_t \setminus B(0, R) \subset (\mathcal{M}^{\infty})^{\textup{reg}}_t,
\end{equation}
and for $t \in [-\frac{3}{2}l, 0]$,
\begin{equation}
    \Sigma \setminus B(0, R) \subset (\mathcal{M}^{\infty})^{\textup{reg}}_t.
\end{equation}

We first claim that for each $t \in [-\frac{3}{2}l, 0]$, 
\begin{equation}
     \Sigma \setminus \{0\} \subset (\mathcal{M}^{\infty})^{\textup{reg}}_t.
\end{equation}
The previous discussion implies that there exists $R > 0$ so that 
\begin{equation}
    \Sigma \setminus B(0, R) \subset (\mathcal{M}^{\infty})^{\textup{reg}}_t
\end{equation}
for each $t \in [-\frac{3}{2}l, 0]$. We now select $R_0 \geq 0$ to be the infimum of all such $R > 0$. Then our claim follows if $R_0 = 0$. So for the sake of contradiction, assume that $R_0 > 0$. Take any $X_0 = (x_0, t_0) \in \Sigma \cap \partial B(0, R_0) \times [-\frac{3}{2}l, 0]$. By the choice of $R_0 > 0$, we see that 
\begin{equation}
    X_0 \in \operatorname{spt}(\mathcal{M}^{\infty}) \cap \operatorname{spt}((\mathcal{H}^{n+m-1}|_{\Sigma})_{[-\frac{3}{2}l, \frac{3}{2}l]}).
\end{equation}
Since $R_0 > 0$, we can find $\epsilon_0 > 0$ so that $\Sigma \cap B(x_0, \epsilon_0)$ is smooth, and $\Sigma$ divides $B(x_0, \epsilon_0)$ into two components, namely $U^+ \cap B(x_0, \epsilon_0)$, and $U^- \cap B(x_0, \epsilon_0)$, where $U^+, U^-$ are given by \eqref{two components of complement of cone}. Due to \eqref{one sidedness of sequence of mean curvauture flow}, we see that for all $t \in [t_0 - \epsilon_0^2, t_0]$, 
\begin{equation}
    \operatorname{spt}(\mathcal{M}^{\infty})_t \cap B(x_0, \epsilon_0)\subset (U^+ \cup \Sigma) \cap B(x_0, \epsilon_0).
\end{equation}
By Huisken's monotonicity formula together with the entropy bound on $\mathcal{M}^j$, we see that
\begin{equation}
    \Theta(X_0, \mathcal{M}^{\infty}) \leq \lambda(\Sigma) < 2.
\end{equation}
We can therefore apply the strong maximum principle for integral Brakke flows (theorem \ref{strongmaximumprinciplebrakkeflows}) and conclude that after possibly shrinking $\epsilon_0 > 0$, we have
\begin{equation}
    X_0 \in (\mathcal{M}^{\infty})^{\textup{reg}}, \ \operatorname{spt}(\mathcal{M}^{\infty})_t \cap B(x_0, \epsilon_0) \equiv \Sigma \cap B(x_0, \epsilon_0)
\end{equation}
for all $t \in [t_0 - \epsilon_0^2, t_0]$. Since $\Sigma \cap \partial B(0, R_0) \times [-\frac{3}{2}l, 0]$ is compact, we conclude that there exists some $\epsilon_1 > 0$ so that
\begin{equation}
    \Sigma \setminus B(0, R_0 - \epsilon_1) \subset (\mathcal{M}^{\infty})^{\textup{reg}}_t
\end{equation}
for $t \in [-\frac{3}{2}l, 0]$. This contradicts the minimality of $R_0 > 0$. Therefore $R_0 = 0$, hence
\begin{equation}
    \Sigma \subset \operatorname{spt}(\mathcal{M}^{\infty})_t, \ \Sigma \setminus \{0\} \subset (\mathcal{M}^{\infty})^{\textup{reg}}_t
\end{equation}
for each $t \in [-\frac{3}{2}l, 0]$. \\

Once we have $\Sigma \subset \operatorname{spt}(\mathcal{M}^{\infty})_t, \ \Sigma \setminus \{0\} \subset (\mathcal{M}^{\infty})^{\textup{reg}}_t$, the entropy bound
\begin{equation}\label{entropy bound}
    \lambda(\mu^{\infty}_t) \leq \lambda(\Sigma)
\end{equation}
implies that
\begin{equation}
    \mu^{\infty}_t \equiv\mathcal{H}^{n+m-1}|_{\Sigma}  
\end{equation}
for all $t \in [-\frac{3}{2}l, 0]$, or else any excess measures will strictly contribute to the entropy of $\mu^{\infty}_t$, and violate \eqref{entropy bound}. \\

To control $t \in [0, \frac{3}{2}l]$, we first note that
\begin{equation}
    \operatorname{spt}(\mathcal{M}^{\infty})_t \subset \Sigma
\end{equation}
for all $t \in [0, \frac{3}{2}l]$. This is because we can apply the avoidance principle (appendix E in \cite{Chodosh2023-yl}) to $\mathcal{M}^{\infty}$ against any leaf of the Hardt-Simon foliation. Since we already know that outside $B(0, R)$, $t \in [-\frac{3}{2}l, \frac{3}{2}l]$, $\mathcal{M}^{\infty}$ is a smooth normal graph over $\Sigma$, we see that for each $t \in [-\frac{3}{2}l, \frac{3}{2}l]$, 
\begin{equation}
    \Sigma \setminus B(0, R) \subset (\mathcal{M}^{\infty})^{\textup{reg}}_t.
\end{equation}
By repeating the previous argument using the strong maximum principle, we can conclude that for each $t \in [-\frac{3}{2}l, \frac{3}{2}l]$, 
\begin{equation}
        \mu^{\infty}_t \equiv \mathcal{H}^{n+m-1}|_{\Sigma}.
\end{equation}
White's local regularity theory \cite{White2005-nl} implies that $(M^j_t)_{t \in [-l, l]} \to (\Sigma)_{t \in [-l,l]}$ locally smoothly away from $\{0\} \times [-l,l]$, which contradicts the last assumption on $\mathcal{M}^j$. This completes the proof of proposition \ref{pseudolocality theorem}.
\end{proof}
By using proposition \ref{pseudolocality theorem}, and the self-shrinkers we constructed in proposition \ref{inner family of self shrinkers}, we can estimate the graphical radius. \\

We first define the precise notion of graphical radius.
\begin{definition}\label{definition of graphical radius}
    Let $(\Tilde{M}_{\tau})_{(-\infty, 0)}$ be as in proposition \ref{Rough graphicality property}. For each $\eta \in (0, \overline{\eta}(n,m))$,  $\tau \leq \overline{\tau}$ for some $\overline{\tau} < 0$, we define $ r(\tau;\eta)$ to be the infimum among all $r > 0$ so that for each $s \leq \tau$, $\Tilde{M}_{s} \setminus B(0, r)$ is graphical over $\Sigma_{2r, \infty} \subset \mathcal{D}_{s} \subset \Sigma$ with the normal graph function satisfying $\|u_{s}\|^*_{C^3(\Sigma_{2r, \infty})} \leq \eta$.
    Here, $\overline{\eta}(n,m) \in (0, \frac{1}{100})$ is given by proposition \ref{Rough graphicality property}, and the norms are given by definition \ref{Weighted Ck norms}.
\end{definition}
\begin{remark}
    Note that by proposition \ref{Rough graphicality property}, we can find $\overline{\tau} < 0$ so that $r(\tau ; \eta) $ is well defined for all $\tau \leq \overline{\tau}$. In fact, we have 
    \begin{equation}
        \lim_{\tau \to -\infty}r(\tau ; \eta) = 0
    \end{equation}
    for every $\eta > 0$.
\end{remark}
The main result of this section is the following graphical radius estimate. This  allows us to estimate the graphical radius with respect to the $L^2$ closeness of the flow to the cone in a \textit{fixed} annulus region.
\begin{proposition}\label{graphical radius estimate proposition} There exist $\overline{r} = \overline{r}(n,m) > 0$, $\overline{\eta} = \overline{\eta}(n,m) \in (0, \frac{1}{100})$ so that the following holds: let $(M_t)_{t \in (-\infty, 0)}$ be a smooth, properly embedded, ancient mean curvature flow that satisfies assumptions \ref{strong convergence to simonscone assumption1}, \ref{additional geoemtric assumption}, and let $(\Tilde{M}_{\tau})_{\tau \in (-\infty, 0)}$ be the rescaled mean curvature flow. Define 
    \begin{equation}
        \rho_0(\tau) = \sup_{s \leq \tau}\| u_{s}\|_{L^2_{w}(\Sigma_{\frac{\overline{r}}{2}, 2\overline{r}})}, 
    \end{equation}
    where the norm is the Gaussian weighted $L^2$ norm given in definition \ref{gaussian weighted sobolev space}.
    Then for any $\eta \in (0, \overline{\eta})$, there exist $\overline{\tau} < 0$, and $C_0 = C_0(n,m,  \eta) > 0$ so that for all $\tau \leq \overline{\tau}$, 
    \begin{equation}\label{graphical radius estimate}
        r(\tau ; \eta) \leq C_0\rho_0(\tau)^{\frac{1}{1 - \alpha}}.
    \end{equation}
    Here, \begin{equation}
        \alpha \in (-2, -1)
    \end{equation}
    is given by \eqref{the formula for alpha}.
\end{proposition}
\begin{remark}
    Proposition \ref{graphical radius estimate proposition} holds when either $n,m \geq 2$ with $n+m \geq 9$, or $(n,m) = (3,5), (4,4), (5,3)$ as well. Moreover this estimate is optimal up to a constant $C_0$ in view of the flow generated by a leaf of the Hardt-Simon foliation, e.g $(e^{\frac{\tau}{2}}\Sigma^+_1)_{\tau \in (-\infty, 0)}$. Direct computation shows that for each small $\eta > 0$, both $r(\tau ; \eta)$ and $\rho_0(\tau)^{\frac{1}{1 - \alpha}}$ are comparable to $e^{\tau/2}$ uniformly in $\tau$. 
\end{remark}
We now prove proposition \ref{graphical radius estimate proposition}. We first prove a slight variant of proposition \ref{graphical radius estimate proposition} with the $L^2$ norm replaced by an $L^{\infty}$ norm. 
\begin{lemma}\label{graphical radius estimate supversion}Under the assumptions of proposition \ref{graphical radius estimate proposition}, there exist $\overline{r}=\overline{r}(n,m)  > 0$, $\overline{\eta} = \overline{\eta}(n,m) \in (0, \frac{1}{100})$ so that for each $\eta \in (0, \overline{\eta})$, there exist constants $\overline{\tau}< 0$, and  $C_0 = C_0(\eta, n,m) > 0$ so that for all $\tau \leq \overline{\tau}$
    \begin{equation}
        r(\tau ; \eta) \leq C_0(\sup_{s \leq \tau, |x| = \overline{r}, x \in \Sigma} |u_{s}(x)|)^{\frac{1}{1 - \alpha}}.
    \end{equation}
    Here, $\alpha \in (-2, -1)$ is given in proposition \ref{graphical radius estimate proposition}.
\end{lemma}
\begin{proof}[Proof of lemma \ref{graphical radius estimate supversion}]
We first choose the constants $\overline{r}$, and $\overline{\tau}$. Recall that for $0 < \theta \leq \theta_0(n,m)$, $\hat{u}_{\theta}$ is the normal graph function of the inner self-shrinker constructed in proposition \ref{inner family of self shrinkers} (see corollary \ref{graphicality of self shrinkers over simons cone}). By proposition \ref{approximate formula for inner self shrinkers}, for any $\epsilon > 0$, there exists $\delta = \delta(n,m, \epsilon) > 0$ so that if
    \begin{equation}
        x \in \Sigma \setminus\{0\}, |x| \leq \delta, \frac{\theta}{|x|} \leq \delta,
    \end{equation}
    then 
    \begin{equation}\label{asymptotic formula}
        |\hat{u}_{\theta}(x) - c_0\theta^{1 - \alpha}|x|^{\alpha}| \leq \epsilon \theta^{1 - \alpha}|x|^{\alpha},
    \end{equation}
    where $c_0 = c_0(n,m) > 0$. Fix $\epsilon = \epsilon(n,m) = \frac{c_0}{2} > 0$, and define $\overline{r} = \overline{r}(n,m) > 0$ to be the corresponding $\delta = \delta(n,m) >0$. \\
    
Define $K = K(n,m, \eta)> 10$ to be the constant from proposition \ref{pseudolocality theorem} with $k = 3, l = 1, L = 15, \delta = \frac{1}{100}$.
    Choose $\overline{\tau} < 0$ so that for all $\tau \leq \overline{\tau}$, $r(\tau; \eta)$ is well defined, and is less than $\min\{\frac{1}{10},\frac{\overline{r}}{100K}\}$. Note that such a choice of $\overline{\tau}$ is possible due to proposition \ref{Rough graphicality property}.\\

    We now prove the desired estimate. For simplicity, let 
    \begin{equation}
        f(\tau) = \sup_{s \leq \tau, |x| = \overline{r}, x \in \Sigma}|u_{s}(x)|.
    \end{equation}
    Suppose that lemma \ref{graphical radius estimate supversion} is false. Then, there exists $\tau \leq \overline{\tau}$ so that 
    \begin{equation}\label{contradiction assumption graphical radius}
        r(\tau ; \eta) > C_0f(\tau)^{\frac{1}{1 - \alpha}}
    \end{equation}
    for large $C_0 > 0$ to be determined. \\
    
    By the avoidance principle against the self-shrinker constructed in proposition \ref{inner family of self shrinkers}, we have
    \begin{equation}
        |u_{s}(x)| \leq |\hat{u}_{\theta}(x)| \text{ for all } x \in \Sigma_{2r(\tau;\eta) , \overline{r}}, \ s \leq \tau
    \end{equation}
  whenever $\theta > 0$ satisfies
    \begin{equation}
        f(\tau) \leq |\hat{u}_{\theta}(x)| \leq 10f(\tau)
    \end{equation}
    for any $|x| = \overline{r}$. 
    By choosing $\overline{\eta} = \overline{\eta}(n,m) \in (0, \frac{1}{100})$ so that $f(\tau) \leq c(n,m)\overline{\eta}$ is small enough, we can use the formula \eqref{asymptotic formula} to find the optimal $\theta$, which is given by
    \begin{equation}
        \theta  = c_1(n,m)f(\tau)^{\frac{1}{1 - \alpha}}
    \end{equation}
    for some $c_1(n,m) > 0$. 
   Note that for all $x \in \Sigma_{2r(\tau ; \eta), \overline{r}}$, we have
    \begin{equation}
        \frac{\theta}{|x|} \leq \frac{\theta}{r(\tau ; \eta)}\leq \frac{c_1(n,m)}{C_0} \leq \overline{r}(n,m),
    \end{equation}
    provided $C_0 \geq \underbar{C}_0(n,m)$. Then the formula \eqref{asymptotic formula} is applicable for $x \in \Sigma_{2r(\tau ; \eta), \overline{r}}$, $\epsilon = \frac{c_0}{2}$; hence we see that for all $x \in \Sigma_{2r(\tau ; \eta) , 10Kr(\tau ; \eta)} \subset \Sigma_{2r(\tau;\eta), \overline{r}}$, $s \leq \tau$ 
    \begin{equation}\label{improved C0}
        \frac{|u_{s}(x)|}{|x|} \leq 2c_0\theta^{1 - \alpha}|x|^{\alpha - 1} \leq 2c_0\theta^{1 - \alpha}r(\tau ; \eta)^{\alpha - 1} \leq2c_0(\frac{c_1}{C_0})^{1 - \alpha} \leq \hat{\epsilon}
    \end{equation}
    provided $C_0 \geq \underbar{C}_0(n,m, \hat{\epsilon})$. By the definition of $r(\tau ; \eta)$, we have the $C^2$ bound
    \begin{equation}
        \|u_{s}\|^*_{C^2(\Sigma_{2r(\tau ; \eta), 10Kr(\tau ; \eta)})} \leq \eta
    \end{equation}
    for all $s \leq \tau$.
    By the interpolation inequality (proposition \ref{interpolation inequalities of weighted Ck norms}) with the above $C^2$ bound, we have
    \begin{equation}\label{improved C1 norm}
        \|u_{s}\|^*_{C^1(\Sigma_{3r(\tau ; \eta), 9Kr(\tau ; \eta)})} \leq \epsilon
    \end{equation}
    for all $s \leq \tau$, 
    provided $\hat{\epsilon} > 0$ is small enough depending on $n,m$, $\eta$, and $\epsilon > 0$. Select $\epsilon = \epsilon( n,m, \eta) > 0$ from proposition \ref{pseudolocality theorem} with $k= 3, \delta = \frac{1}{100}, l = 1, L = 15$.\\

    We will now use proposition \ref{pseudolocality theorem}. For each $s \leq \tau$, define the mean curvature flow
\begin{equation}
    M^s_{t} = \sqrt{1-t}\Tilde{M}_{s - \ln(1-t)}, \ t \in [-8r^2(\tau ; \eta), 8r^2(\tau ; \eta)].
\end{equation}
Recalling that $r(\tau ; \eta) \leq \frac{1}{10}$ by our choice of $\overline{\tau}$, the previous discussion on $\Tilde{M}_{s}$ implies that for each $t \in [-8r^2(\tau ; \eta), 0]$, 
\begin{equation}
    M^s_t \setminus B(0,2r(\tau ; \eta)) \text{ can be written as a normal graph over }\mathcal{D}_t \subset \Sigma,
\end{equation}
with
\begin{equation}
    \Sigma_{4r(\tau ; \eta), 9Kr(\tau ; \eta)} \subset \mathcal{D}_t, \ \|\hat{u}^s_t\|^*_{C^1(\Sigma_{4r(\tau ; \eta), 9Kr(\tau ; \eta)})} \leq \epsilon, \ \|\hat{u}^s_t\|^*_{C^2(\Sigma_{4r(\tau ; \eta), 9Kr(\tau ; \eta)})} \leq \eta.
\end{equation}
The required entropy bound \eqref{entropy bound proof of pseudolocality} follows from \eqref{entropy bound proposition eq}, and one sidedness follows from assumption \ref{additional geoemtric assumption}. Thus we can apply proposition \ref{pseudolocality theorem} to $M^s_t$ with $r = 2r(\tau ; \eta)$, $k = 3, \delta = \frac{1}{100}, l = 1, L = 15$. Taking $t = 0$, and noting that $M^s_0 = \Tilde{M}_s$, we see that $\Tilde{M}_s \cap (B(0, 30r(\tau ; \eta) \setminus B(0, \frac{1}{50}r(\tau ; \eta))$ is a normal graph over $\Sigma_{\frac{r(\tau ; \eta)}{25}, 28r(\tau ; \eta)} \subset \mathcal{D}_0$, with $\|u_s\|^*_{C^3(\Sigma_{\frac{r(\tau ; \eta)}{25}, 28r(\tau ; \eta)})} \leq 
\eta$ for all $s \leq \tau$. Since we already have graphicality of $\Tilde{M}_{s}$ over $\Sigma$ outside $B(0, r(\tau ; \eta))$ with weighted $C^3$ norm bounded by $\eta$ due to the definition of $r(\tau ; \eta)$, the minimality of $r(\tau ; \eta)$ implies that
    \begin{equation}
        \frac{1}{25}r(\tau ; \eta) \geq r(\tau ; \eta)
    \end{equation}
    which is a contradiction since $r(\tau ; \eta) > 0$ by \eqref{contradiction assumption graphical radius}. Thus lemma \ref{graphical radius estimate supversion} is true for any $C_0 \geq \underbar{C}_0(n,m,  \eta)$, in particular for $C_0 = \underbar{C}_0(n,m,  \eta) > 0$. 
\end{proof}
Once we have lemma \ref{graphical radius estimate supversion}, by possibly shrinking $\overline{\eta}(n,m)$, proposition \ref{graphical radius estimate proposition} immediately follows. This is because we can apply standard interior Schauder estimates to replace the $L^{\infty}$ norm with the $L^2$ norm, and use the fact that in $\Sigma_{\frac{\overline{r}}{2}, 2\overline{r}}$, the weighted $L^2$ norm is equivalent to the standard $L^2$ norm up to a constant only depending on $n,m$. \\

Note that in the proof of proposition \ref{graphical radius estimate proposition}, we used the avoidance principle to obtain a weighted $C^0$ estimate of $u_{\tau}$ near the origin given by \eqref{improved C0}. It turns out that this estimate is useful for later purposes, so we record it separately. The constants below are essentially those given in proposition \ref{graphical radius estimate proposition} with $\eta = \overline{\eta}(n,m) \in (0, \frac{1}{100})$. 
\begin{proposition}\label{inner region weighted C0 estimate}
    Let $ (\Tilde{M}_{\tau})_{\tau \in (-\infty, 0)}$ be a smooth, properly embedded ancient rescaled mean curvature flow which satisfies assumptions \ref{strong convergence to simonscone assumption1} and \ref{additional geoemtric assumption}. Then there exist $\overline{r}(n,m) > 0$, $c_0(n,m), \ C_0(n,m) > 0$, $\overline{\tau}< 0$ so that for all $\tau \leq \overline{\tau}$, 
    \begin{equation}
        x \in \Sigma, \ C_0(\sup_{s \leq \tau}\|u_s\|_{L^2_w(\Sigma_{\frac{\overline{r}}{2}, 2\overline{r}})})^{\frac{1}{1 - \alpha}} \leq |x| \leq \overline{r},
    \end{equation}
    the following weighted $C^0$ estimate
    \begin{equation}
        \frac{|u_{\tau}(x)|}{|x|} \leq c_0\sup_{s \leq \tau}\|u_s\|_{L^2_w(\Sigma_{\frac{\overline{r}}{2}, 2\overline{r}})}|x|^{\alpha - 1}
    \end{equation}
    holds. Here, $\alpha\in (-2, -1)$ is given by \eqref{the formula for alpha}.
\end{proposition}

\section{Inner outer estimate}\label{innerouterestimatesection}
In this section, we assume $n,m \geq 2$ with $n+m \geq 10$, and set $\Sigma$ to be the $O(n)\times O(m)$ symmetric minimal quadratic cone. By using the fact that the inner self-shrinkers constructed in proposition \ref{inner family of self shrinkers} foliate a fixed small ball centered at the origin (theorem \ref{inner region foliation}), we establish an inner outer estimate, which is an inverse Poincar\'e type inequality analogous to the estimate introduced in section 4.5 in \cite{angenent2019unique} (proposition \ref{inner outer estimate}). This inner outer estimate, and the graphical radius estimate (proposition \ref{graphical radius estimate proposition}) will serve as key tools in later analysis.\\

To state the inner outer estimate, we need some definitions. We first define Huisken's functional. 
\begin{definition}\label{Huisken's functional}
    For a properly embedded hypersurface $M \subset \mathbf{R}^{n+m}$, the Huisken functional is defined by
    \begin{equation}
        H(M) = \int_{M}e^{-\frac{|x|^2}{4}}d\mathcal{H}^{n+m-1}(x).
    \end{equation}
\end{definition}
We next define graphical radius for general hypersurfaces, analogous to that in definition \ref{definition of graphical radius}. 
\begin{definition}\label{def of graphical radius for general hypersurface}
    Let $M^{n+m-1} \subset \mathbf{R}^{n+m}$ be a smooth, properly embedded hypersurface. For each $\eta \in (0, \frac{1}{100})$, we define
    \begin{equation*}
        r(\eta ; M) = \inf\{ r > 0 \ | \ M \setminus B(0, r) \text{ is graphical over }\Sigma_{2r, \infty} \subset \mathcal{D}\text{ with }\|u\|^*_{C^1(\Sigma_{2r, \infty})} \leq \eta \}.
    \end{equation*}
    If the set on the right hand side is empty, we then define $r = \infty$. Here, the weighted norm is given by definition \ref{Weighted Ck norms}.
    \end{definition}
    We now prove the inner outer estimate. 
    \begin{proposition}\label{inner outer estimate}
        Let $n ,m\geq 2$ with $n+m \geq 10$, $\Sigma$ be the $O(n)\times O(m)$ symmetric minimal quadratic cone, and $M \subset \mathbf{R}^{n+m}$ be a smooth, properly embedded hypersurface so that 
        \begin{equation}
            H(M) \leq H(\Sigma),
        \end{equation}
        where $H$ is the Huisken functional given in definition \ref{Huisken's functional}. Let $d = n+m-1$. There exist $\overline{r}(n,m) > 0$, $C(n,m) > 0$, and $\overline{\eta} = \overline{\eta}(n,m)  \in (0, \frac{1}{100})$ so that if the graphical radius of $M$ given by definition \ref{def of graphical radius for general hypersurface} satisfies $r(\eta ; M) \leq \frac{\overline{r}}{10}$ for some $\eta \in (0, \overline{\eta})$, then if we let $u$ denote the normal graph function of $M$ over $\Sigma$, we have for any $3r(\eta ; M) \leq r \leq \overline{r}(n,m)$, 
      \begin{align*}
           \frac{1}{r^2}&\int_{\Sigma_{r,2r}}u^2e^{-\frac{|x|^2}{4}}d\mathcal{H}^{d} + \int_{\Sigma_{r, \infty}}|\nabla^{\Sigma}u|^2e^{-\frac{|x|^2}{4}}d\mathcal{H}^{d} 
           \\ & + \int_{\Sigma_{r, \infty}}u^2(|x|^2 + \frac{1}{|x|^2})e^{-\frac{|x|^2}{4}}d\mathcal{H}^{d} \leq C(n,m)\int_{\Sigma_{2r,\infty}}u^2e^{-\frac{|x|^2}{4}}d\mathcal{H}^{d}.
      \end{align*}
    \end{proposition}
    Recall by proposition \ref{Rough graphicality property}, the graphical radius of the time slices of the rescaled mean curvature flow that satisfies assumption \ref{strong convergence to simonscone assumption1} can be made arbitrarily small for all sufficiently negative time. Moreover, the Huisken functional bound
 $H(\Tilde{M}_{\tau}) \leq H(\Sigma)$ is satisfied by Huisken's monotonicity formula \cite{huisken1990asymptotic}.
 Thus, combining this fact with the inner outer estimate (proposition \ref{inner outer estimate}), we obtain the following corollary. 
    \begin{corollary}\label{inner outer estimate for RMCF corollary}
        Let $(M_t)_{t \in (-\infty, 0)}$ be a smooth, properly embedded mean curvature flow that satisfies assumption \ref{strong convergence to simonscone assumption1}, and let $(\Tilde{M}_{\tau})_{\tau \in (-\infty, 0)}$ be the rescaled mean curvature flow. 
        Then there exist $C(n,m) > 0, \ \overline{r}(n,m) > 0$,  $\overline{\eta}(n,m) \in (0, \frac{1}{100})$, $\overline{\tau} < 0$ so that for any $\tau \leq \overline{\tau}$, 
        \begin{align*}
           \frac{1}{r^2}&\int_{\Sigma_{r,2r}}u_{\tau}^2e^{-\frac{|x|^2}{4}}d\mathcal{H}^{d} + \int_{\Sigma_{r, \infty}}|\nabla^{\Sigma}u_{\tau}|^2e^{-\frac{|x|^2}{4}}d\mathcal{H}^{d} 
           \\ & + \int_{\Sigma_{r, \infty}}u_{\tau}^2(|x|^2 + \frac{1}{|x|^2})e^{-\frac{|x|^2}{4}}d\mathcal{H}^{d} \leq C(n,m)\int_{\Sigma_{2r,\infty}}u_{\tau}^2e^{-\frac{|x|^2}{4}}d\mathcal{H}^{d}
      \end{align*}
        holds for any $3r(\tau ; \eta) \leq r \leq \overline{r}(n,m)$, where $u_{\tau}$ is the normal graph function of $\Tilde{M}_{\tau}$ over $\Sigma$. 
    \end{corollary}
    In the proof of proposition \ref{inner outer estimate}, we need the following Poincar\'e type inequalities.  
\begin{lemma}\label{Poincare type inequality for u^2/r^2}
    Let $u$ be a smooth function defined on $\Sigma_{r, \infty}$ for any $r > 0$. Assume that $u(x)$ has at most polynomial growth in $|x|$ as $|x| \to \infty$. Let $d= n+m-1$. Then 
\begin{align*}
    \frac{d}{2}\int_{\Sigma_{r, \infty}}\frac{u^2}{|x|^2}e^{-\frac{|x|^2}{4}}d\mathcal{H}^{d} \leq & \int_{\Sigma_{r, \infty}}u^2e^{-\frac{|x|^2}{4}}d\mathcal{H}^{d} + \frac{8}{3d - 8}\int_{\Sigma_{r, \infty}}|\nabla u|^2e^{-\frac{|x|^2}{4}}d\mathcal{H}^{d} \\ &- 2\int_{\partial \Sigma_{r, \infty}}\frac{u^2}{|x|}e^{-\frac{|x|^2}{4}}d\mathcal{H}^{d-1}.
\end{align*}
Also, 
\begin{align*}
    \int_{\Sigma_{r, \infty}}u^2|x|^2e^{-\frac{|x|^2}{4}}d\mathcal{H}^{d} \leq 4\int_{\partial \Sigma_{r, \infty}}|x|u^2e^{-\frac{|x|^2}{4}}d\mathcal{H}^{d-1} + 4d\int_{\Sigma_{r, \infty}}(u^2 + |\nabla u|^2)e^{-\frac{|x|^2}{4}}d\mathcal{H}^{d}.
\end{align*}
\end{lemma}
\begin{proof}[Proof of lemma \ref{Poincare type inequality for u^2/r^2}]
    We parametrize $\Sigma_{r, \infty}$ as 
    \begin{equation}
        (r, \infty) \times (\sqrt{\frac{n-1}{n+m-2}}\mathbf{S}^{n-1} \times \sqrt{\frac{m-1}{n+m-2}}\mathbf{S}^{m-1}) \ni (s,\omega) \to s\cdot \omega \in \Sigma_{r, \infty}. 
    \end{equation}
    We denote the volume measure of $T_{n,m} =\sqrt{\frac{n-1}{n+m-2}}\mathbf{S}^{n-1} \times \sqrt{\frac{m-1}{n+m-2}}\mathbf{S}^{m-1}$ by $d\sigma(\omega)$. Further set $d = n+m-1$. Then by integration by parts together with the growth assumption on $u$ as $|x| \to \infty$, we have
    \begin{align*}
        \int_{\Sigma_{r, \infty}}u^2e^{-\frac{|x|^2}{4}}d\mathcal{H}^{d}(x) & = \int_{r}^{\infty}\int_{T_{n,m}}u^2(s\omega)s^{d-1}e^{-\frac{s^2}{4}}d\sigma(\omega)ds \\ & = \int_{T_{n,m}}\int_{r}^{\infty}-2u^2s^{d-2}\frac{d}{ds}(e^{-\frac{s^2}{4}})dsd\sigma(\omega) \\ & = 2\int_{T_{n,m}}u^2(r\omega)r^{d-2}e^{-\frac{r^2}{4}}d\sigma(\omega) +\int_{r}^{\infty}\int_{T_{n,m}}(2d-4)u^2s^{d-3}e^{-\frac{s^2}{4}}d\sigma(\omega)ds \\ & \ \ \ \ + \int_{r}^{\infty}\int_{T_{n,m}}4u\partial_su s^{d-2}e^{-\frac{s^2}{4}}d\sigma(\omega)ds.
    \end{align*}
    Now, applying Young's inequality
    \begin{equation}
        |4u\partial_su s^{d-2}| \leq (\frac{3d}{2}-4)u^2s^{d-3} + \frac{8}{3d-8}(\partial_su)^2s^{d-1},
    \end{equation}
    we obtain
    \begin{align*}
        \int_{\Sigma_{r, \infty}}u^2e^{-\frac{|x|^2}{4}}d\mathcal{H}^{d} & \geq \frac{d}{2}\int_{\Sigma_{r, \infty}}\frac{u^2}{|x|^2}e^{-\frac{|x|^2}{4}}d\mathcal{H}^{d} - \frac{8}{3d - 8}\int_{\Sigma_{r, \infty}}|\nabla u|^2e^{-\frac{|x|^2}{4}}d\mathcal{H}^{d} +2\int_{\partial \Sigma_{r, \infty}}\frac{u^2}{|x|}e^{-\frac{|x|^2}{4}}d\mathcal{H}^{d-1},
    \end{align*}
    thus proving the first inequality in lemma \ref{Poincare type inequality for u^2/r^2}.\\

    We now turn to the second inequality. Using the same parametrization, we have
    \begin{align*}
        \int_{\Sigma_{r, \infty}}u^2|x|^2e^{-\frac{|x|^2}{4}}d\mathcal{H}^{d} & = \int_{r}^{\infty}\int_{T_{n,m}}u^2(s\omega)s^{d+1}e^{-\frac{s^2}{4}}d\sigma(\omega)ds \\ & =\int_{r}^{\infty}\int_{T_{n,m}}-2u^2(s\omega)s^{d}\frac{d}{ds}(e^{-\frac{s^2}{4}})d\sigma(\omega)ds \\ & = \int_{T_{n,m}}2u^2(r\omega)r^{d}e^{-\frac{r^2}{4}}d\sigma(\omega) \\ & \ \ \ \ + \int_{r}^{\infty}\int_{T_{n,m}}2du^2(s\omega)s^{d-1} e^{-\frac{s^2}{4}}d\sigma(\omega)ds  + \int_{r}^{\infty}\int_{T_{n,m}}4u\partial_sus^{d} e^{-\frac{s^2}{4}}d\sigma(\omega)ds \\ & \leq \int_{\partial \Sigma_{r, \infty}}2|x|u^2e^{-\frac{|x|^2}{4}}d\mathcal{H}^{d-1} + 2d\int_{\Sigma_{r, \infty}}u^2e^{-\frac{|x|^2}{4}}d\mathcal{H}^{d} \\ & \ \ \ \ + 8\int_{\Sigma_{r, \infty}}|\nabla u|^2e^{-\frac{|x|^2}{4}}d\mathcal{H}^{d} + \frac{1}{2}\int_{\Sigma_{r, \infty}}u^2|x|^2e^{-\frac{|x|^2}{4}}d\mathcal{H}^{d}.
    \end{align*}
    Thus, absorbing the far right integral gives us the second Poincar\'e type inequality.
\end{proof}
    \begin{proof}[Proof of proposition \ref{inner outer estimate}]
    We first set $\overline{r}(n,m) \leq \frac{r_0(n,m)}{10}$, where $r_0(n,m) > 0$ is given in theorem \ref{inner region foliation}. We fix any $3r(\eta ; M) \leq r \leq \overline{r}(n,m)$, and consider any $r \leq l \leq 2r$. We now split $\mathbf{R}^{n+m}$ into two regions given by
    \begin{equation}
        D_{0,l} = \{(x, y) \in \mathbf{R}^n\times \mathbf{R}^m \ | \ |y| + \sqrt{\frac{n-1}{m-1}}|x| \leq (\sqrt{\frac{m-1}{n+m-2}} + \frac{n-1}{\sqrt{(n+m-2)(m-1)}})l \},
    \end{equation}
    and
    \begin{equation}
        D_{l,\infty} = \mathbf{R}^{n+m} \setminus D_{0,l}.
    \end{equation}
    In other words, $D_{0,l}$ is the set of all points whose projections onto the minimal cone are contained in $B(0, l)$.
    We define 
    \begin{equation}
        M_{0,l} = M \cap D_{0,l}, \ M_{l, \infty} = M \cap D_{l, \infty}.
    \end{equation}
    Finally, we define
    \begin{equation}
        H_{0,l}(M) = \int_{M_{0,l}}e^{-\frac{|x|^2}{4}}d\mathcal{H}^{d}, \ H_{l, \infty}(M) = \int_{M_{l, \infty}}e^{-\frac{|x|^2}{4}}d\mathcal{H}^{d},
    \end{equation}
    and define $H_{0,l}(\Sigma), E_{l, \infty}(\Sigma)$ analogously. \\

    By the inequality
    \begin{equation}
        H(M) \leq H(\Sigma),
    \end{equation}
    we obtain
    \begin{equation}
        H_{0,l}(M) - H_{0,l}(\Sigma) \leq H_{l,\infty}(\Sigma) - H_{l, \infty}(M).
    \end{equation}
    We first estimate 
    \begin{equation}
        H_{l,\infty}(\Sigma) - H_{l, \infty}(M).
    \end{equation}
    By the assumption that $3r(\eta ; M) \leq r \leq l$, we see that $M_{l, \infty}$ is a normal graph over $\Sigma_{l, \infty}$ with graph norm $\|u\|^*_{C^1(\Sigma_{l, \infty})} \leq \eta \leq \overline{\eta}(n,m)$, where $\overline{\eta}(n,m) > 0$ will be determined. This means that we can use the parametrization for $M_{l, \infty}$:
    \begin{equation}
        \Sigma_{l, \infty} \ni x = (\sqrt{\frac{n-1}{n+m-2}}|x|w_1, \sqrt{\frac{m-1}{n+m-2}}|x|w_2) \to x + u(x)\nu_{\Sigma}(x) \in M_{l, \infty},
    \end{equation}
    where $\nu_{\Sigma}(x) = (-\sqrt{\frac{m-1}{n+m-2}}w_1, \sqrt{\frac{n-1}{n+m-2}}w_2)$ for $w_1 \in \mathbf{S}^{n-1},  w_2 \in \mathbf{S}^{m-1}$. Then we can write
\begin{equation}
    H_{l,\infty}(\Sigma) - H_{l, \infty}(M)  = \int_{\Sigma_{l, \infty}}e^{-\frac{|x|^2}{4}}[1 -  J(x, u, \nabla u)e^{-\frac{u^2}{4}}]d\mathcal{H}^{d},
\end{equation}
where $J(x, u, \nabla u)$ is the Jacobian of the parametrization. By directly computing $J$, we see that for any $\epsilon > 0$, there exists $\overline{\eta}_0(n,m,\epsilon) \in (0, \frac{1}{100})$ so that whenever $0 < \eta \leq \overline{\eta}_0$, we have
\begin{equation}
    J(x,u, \nabla u) \geq (1 - \sqrt{\frac{m-1}{n-1}}\frac{u}{|x|})^{n-1}(1 + \sqrt{\frac{n-1}{m-1}}\frac{u}{|x|})^{m-1}(1 + (\frac{1}{2} - \epsilon)|\nabla u|^2).
\end{equation}
The above inequality implies that by possibly making $\overline{\eta}_0(n,m,\epsilon) > 0$ smaller, we have
\begin{align*}
    1 - J(x,u, \nabla u)e^{-\frac{u^2}{4}} & \leq 1 - e^{-\frac{u^2}{4}}(1 - \frac{n+m-2+\epsilon}{2}\frac{u^2}{|x|^2})(1 + (\frac{1}{2} - \epsilon)|\nabla u|^2) \\ & \leq 1 - (1 - \frac{u^2}{4})(1 - \frac{n+m-2+\epsilon}{2}\frac{u^2}{|x|^2})(1 + (\frac{1}{2} - \epsilon)|\nabla u|^2) \\ & \leq  \frac{1}{3}u^2 + \frac{n+m-1}{2}\frac{u^2}{|x|^2} - (\frac{1}{2} - 2\epsilon)|\nabla u|^2.
\end{align*}
By integrating the above estimate, we see that for any $\epsilon > 0$, $0 < \eta \leq \overline{\eta}_0(n,m, \epsilon)$, $3r(\eta, M) \leq r \leq l \leq 2r \leq 2\overline{r}(n,m) \leq \frac{r_0(n,m)}{5}$, we have
\begin{equation}\label{righthandside first estimate}
    H_{l, \infty}(\Sigma) - H_{l, \infty}(M)  \leq \int_{\Sigma_{l, \infty}}(\frac{1}{3}u^2 + \frac{d}{2}\frac{u^2}{|x|^2} - (\frac{1}{2} - 2\epsilon)|\nabla u|^2)e^{-\frac{|x|^2}{4}}d\mathcal{H}^{d}.
\end{equation}
Applying the first inequality in lemma \ref{Poincare type inequality for u^2/r^2} to \eqref{righthandside first estimate}, we obtain
    \begin{align*}
         H_{l, \infty}(\Sigma) - H_{l, \infty}(M)  &\leq \int_{\Sigma_{l, \infty}}(\frac{1}{3}u^2 + \frac{d}{2}\frac{u^2}{|x|^2} - (\frac{1}{2} - 2\epsilon)|\nabla u|^2)e^{-\frac{|x|^2}{4}}d\mathcal{H}^{d} \\ & \leq (\frac{8}{3d - 8} - (\frac{1}{2} - 2\epsilon))\int_{\Sigma_{l, \infty}}|\nabla u|^2e^{-\frac{|x|^2}{4}}d\mathcal{H}^{d} \\ & \ \ \ \ + 2\int_{\Sigma_{l, \infty}}u^2e^{-\frac{|x|^2}{4}}d\mathcal{H}^{d} - 2\int_{\partial \Sigma_{l, \infty}}\frac{u^2}{|x|}e^{-\frac{|x|^2}{4}}d\mathcal{H}^{d-1}.
    \end{align*}
    Since $d\geq 9$, $\frac{8}{3d - 8} < \frac{1}{2}$, therefore we can find $\epsilon(n,m) > 0$ so that
    \begin{equation}
        \frac{8}{3d - 8} - (\frac{1}{2} - 2\epsilon) = -C_0(n,m) < 0.
    \end{equation}
Now that $\epsilon(n,m) > 0$ is determined, $\overline{\eta}_0(n,m) = \overline{\eta}_0(n, m,\epsilon(n,m)) \in (0, \frac{1}{100})$ is also determined; thus we see that for any $\eta \in (0, \overline{\eta}_0(n,m))$, $3r(\eta ; M) \leq r \leq l \leq 2r \leq 2\overline{r} \leq  \frac{r_0(n,m)}{5}$, 
\begin{align*}
     H_{l, \infty}(\Sigma) - H_{l, \infty}(M)  & \leq - C_0(n,m)\int_{\Sigma_{l, \infty}}|\nabla u|^2e^{-\frac{|x|^2}{4}}d\mathcal{H}^{d} + 2\int_{\Sigma_{l, \infty}}u^2e^{-\frac{|x|^2}{4}}d\mathcal{H}^{d}  - 2\int_{\partial \Sigma_{l, \infty}}\frac{u^2}{|x|}e^{-\frac{|x|^2}{4}}d\mathcal{H}^{d-1}.
\end{align*}
We now estimate $H_{0,l}(M) - H_{0,l}(\Sigma)$. The idea is to use the fact that the inner self-shrinkers constructed in proposition \ref{inner family of self shrinkers} foliate $B(0, r_0(n,m))$ to run a calibration argument as in \cite{angenent2019unique}. Note that by possibly making $\overline{r}(n,m) > 0$ and $\overline{\eta}(n,m) > 0$ smaller, we can ensure that one can indeed apply the calibration argument. \\

By theorem \ref{inner region foliation}, the ball $B(0, r_0(n,m))$ is foliated by the self-shrinkers
\begin{equation}
        \{\Gamma_{\theta}\}_{\theta \in (0, \sqrt{2(\min(n,m)-1)})}\cup \{\Lambda_{\theta}\}_{\theta \in (0, \sqrt{2(\min(n,m)-1)})} \cup \{\Sigma\},
    \end{equation}
where 
\begin{equation*}
    \Gamma_{\theta} = \{ (xw_1, u_{\theta}w_2) \ | \ x \in [0, \sqrt{2(\min(n,m)-1)}]\}, \ \Lambda_{\theta} = \{(v_{\theta}(x)w_1, xw_2) \ | \ x \in [0, \sqrt{2(\min(n,m)-1)}]\},
\end{equation*}
with $u_{\theta}$ being the profile function constructed in proposition \ref{inner family of self shrinkers}, and $v_{\theta}$ being the profile function constructed in proposition \ref{inner family of self shrinkers} with $n,m$ reversed. 
We can thus define a smooth vector field in $B(0, r_0) \setminus \{0\}$ given by the unit normal vectors of the above self-shrinkers. More precisely, we can define for each $x \in B(0, r_0) \setminus \{0\}$, 
\begin{equation}
\nu_{fol}(x) = 
    \begin{cases}
        \frac{1}{\sqrt{1 + (u_{\theta}')^2}}(-u_{\theta}'w_1, w_2) \text{ if }x = (sw_1, u_{\theta}w_2) \in \Gamma_{\theta} \\ 
        (-\sqrt{\frac{m-1}{n+m-2}}w_1, \sqrt{\frac{n-1}{n+m-2}}w_2) \text{ if }x = (sw_1, \sqrt{\frac{m-1}{n-1}}sw_2) \in \Sigma \\ 
        \frac{1}{\sqrt{1 + (v_{\theta}')^2}}(-w_1,v_{\theta}' w_2) \text{ if }x = (v_{\theta}w_1, sw_2) \in \Lambda_{\theta}.
    \end{cases}
\end{equation}
We can then apply the divergence theorem. More specifically, due to the fact that $\nu_{fol}$ is not defined up to the origin, we make an arbitrarily small perturbation of both $M$ and $\Sigma$ near the origin so that the origin is not contained in the perturbed surfaces. We then apply the divergence theorem to the perturbed surface, and then take the limit back to the original $M$ and $\Sigma$. Keeping in mind that $3r(\eta ; M) \leq r \leq l \leq 2r \leq \frac{r_0}{5}$ with $\eta > 0$ small, and using the formula of $\nu_{fol}$, we have the lower bound
\begin{equation}\label{conseuqence of divergence theorem}
    H_{0,l}(M) - H_{0, l}(\Sigma) \geq -\int_{\omega \in T_{n,m}}\int_0^{|u(l\omega)|}|w \cdot \nu_{fol}|e^{-\frac{l^2}{4}}l^{d-1}dtd\sigma(\omega),
\end{equation}
where $\nu_{fol}$ is evaluated at 
\begin{equation}\label{signchoice}
    lw \pm t\nu_{n,m}(\omega)
\end{equation}
for $0 \leq t \leq |u(l\omega)|$. Here, the sign $\pm$ depends on $\omega \in T_{n,m}$, and if $w = (\sqrt{\frac{n-1}{n+m-2}}w_1, \sqrt{\frac{m-1}{n+m-2}}w_2)$ for $w_1 \in \mathbf{S}^{n-1}$, $w_2 \in \mathbf{S}^{m-1}$, then $\nu_{n,m}(w) = (-\sqrt{\frac{m-1}{n+m-2}}w_1, \sqrt{\frac{n-1}{n+m-2}}w_2)$. \\

We now estimate $|w\cdot \nu_{fol}|$. We will only be estimating the case when the sign in \eqref{signchoice} is positive, since the negative case follows from the same argument as in the positive case. We use proposition \ref{approximate formula for inner self shrinkers}. Let $\epsilon > 0$ to be determined. Then proposition \ref{approximate formula for inner self shrinkers} implies that we can find $\theta_0(n,m) > 0$, and $\delta = \delta(n,m, \epsilon) > 0$ so that whenever $x \in \partial\Sigma_{l, \infty}$, if 
\begin{equation}
    l \leq \delta, \frac{\theta}{l} \leq \delta, 0 < \theta \leq \theta_0,
\end{equation}
then
 \begin{equation}\label{C0 estimate for inner self shrinker}
        |\hat{u}_{\theta}(x) - c_0\theta^{1 - \alpha}|x|^{\alpha}| \leq \epsilon\theta^{1 - \alpha}|x|^{\alpha},
    \end{equation}
    and
    \begin{equation}\label{C1 estimate for inner self shrinker}
        |\hat{u}_{\theta}'(x) - \alpha c_0\theta^{1 - \alpha}|x|^{\alpha - 1}| \leq \epsilon \theta^{1 - \alpha}|x|^{\alpha - 1},
    \end{equation}
    where $\alpha  \in (-2, -1)$ is given by \eqref{the formula for alpha}, and $c_0(n,m) > 0$, and $\hat{u}_{\theta}$ is the normal graph function of $\Gamma_{\theta}$ over $\Sigma$. \\

     We can then first use \eqref{C0 estimate for inner self shrinker} to determine the correct $\theta > 0$ for each $x = lw + t\nu_{n,m}(w)$, and then use \eqref{C1 estimate for inner self shrinker} to estimate $|\nu_0 \cdot \nu_{fol} |$. In particular, we first see that for each $\epsilon > 0$, we can find $\overline{\eta}_1 = \overline{\eta}_1(n,m, \epsilon) \in (0, \frac{1}{100})$ so that whenever $l \leq \delta(n,m, \epsilon)$, $\eta \leq \overline{\eta}_1$, we have that for each $x$, the $\theta$ which $x \in \Gamma_{\theta}$ satisfies
    \begin{equation}
        0 \leq \frac{\theta}{l} \leq (\frac{1}{c_0 - \epsilon} \frac{|u(l \omega)|}{l})^{\frac{1}{1 - \alpha}}.
    \end{equation}
    Then, \eqref{C1 estimate for inner self shrinker} implies that
    \begin{equation}\label{estimate of foliation normal}
        |w\cdot \nu_{fol}| = \frac{|\hat{u}_{\theta}'|}{\sqrt{1 + (\hat{u}_{\theta}')^2}} \leq 
   (|\alpha|c_0 + \epsilon)(\frac{\theta}{l})^{1 - \alpha} \leq \frac{|\alpha|(c_0 + \epsilon)}{c_0 - \epsilon}\frac{|u(l\omega)|}{l}.
    \end{equation}
    Combining this with \eqref{conseuqence of divergence theorem} implies that for each $\epsilon > 0$, we can find $\delta = \delta(n,m, \epsilon) > 0$, $\overline{\eta}_1 = \overline{\eta}_1(n,m, \epsilon) > 0$ so that whenever $\eta \in (0, \overline{\eta}_1(n,m, \epsilon))$, $3r(\eta ; M) \leq r \leq l \leq 2r \leq \min \{ \frac{r_0}{5}, \delta\}$, we have
\begin{align*}
     H_{0,l}(M) - H_{0, l}(\Sigma) &\geq -\int_{\omega \in T_{n,m}}\int_0^{|u(l\omega)|}|\nu_0 \cdot \nu_{fol}|e^{-\frac{l^2}{4}}l^{d-1}dtd\sigma(\omega) \\ & \geq - \int_{\omega \in T_{n,m}}\int_0^{|u(l\omega)|}\frac{|\alpha|(c_0 + \epsilon)}{c_0 - \epsilon}\frac{|u(l\omega)|}{l}e^{-\frac{l^2}{4}}l^{d-1}dtd\sigma(\omega) \\ & = - \int_{\omega \in T_{n,m}}\frac{|\alpha|(c_0 + \epsilon)}{c_0 - \epsilon}\frac{|u(l\omega)|^2}{l}e^{-\frac{l^2}{4}}l^{d-1}d\sigma(\omega)\\ & = - \frac{|\alpha|(c_0 + \epsilon)}{c_0 - \epsilon}\int_{\partial \Sigma_{l, \infty}}\frac{u^2}{|x|}e^{-\frac{|x|^2}{4}}d\mathcal{H}^{d-1}.
\end{align*}
Combining this with the previous estimate, we obtain
\begin{align*}
     (2 - |\alpha|\frac{c_0 + \epsilon}{c_0 - \epsilon})\int_{\partial \Sigma_{l, \infty}}\frac{u^2}{|x|}e^{-\frac{|x|^2}{4}}d\mathcal{H}^{d-1} + C_0(n,m)\int_{\Sigma_{l, \infty}}|\nabla u|^2e^{-\frac{|x|^2}{4}}d\mathcal{H}^{d}  \leq 2\int_{\Sigma_{l, \infty}}u^2e^{-\frac{|x|^2}{4}}d\mathcal{H}^{d},
\end{align*}
provided $0 < \eta \leq \min \{\overline{\eta}_0(n,m), \overline{\eta}_1(n,m, \epsilon)\}$, $3r(\eta ; M) \leq r \leq l \leq 2r \leq \min \{\frac{r_0(n,m)}{5},  \delta(n,m, \epsilon) \}$. Since $n +m\geq 10$, we have $|\alpha| < 2$ and $c_0(n,m) > 0$; thus we can finally make the choice $\epsilon = \epsilon(n,m) > 0$ so that
\begin{equation}
    2 - |\alpha|\frac{c_0 + \epsilon}{c_0 - \epsilon} = c_1(n,m) > 0,
\end{equation}
which then allows us to choose $\overline{\eta}(n,m) \in(0, \frac{1}{100})$, and $\overline{r}(n,m) > 0$ so that whenever $0 < \eta \leq \overline{\eta}(n,m)$, $3r(\eta ; M) \leq r \leq l \leq 2r \leq 2\overline{r}(n,m)$, we have
\begin{equation}\label{alm final form of innerouter}
     \int_{\partial \Sigma_{l, \infty}}\frac{u^2}{|x|}e^{-\frac{|x|^2}{4}}d\mathcal{H}^{d-1} + \int_{\Sigma_{l, \infty}}|\nabla u|^2e^{-\frac{|x|^2}{4}}d\mathcal{H}^{d} \leq C_1(n,m)\int_{\Sigma_{l, \infty}}u^2e^{-\frac{|x|^2}{4}}d\mathcal{H}^{d}.
\end{equation}
Integrating over $r \leq l \leq 2r$,  and dropping the gradient term, we obtain
\begin{equation}
    \frac{1}{2r}\int_{\Sigma_{r, 2r}}u^2e^{-\frac{|x|^2}{4}}d\mathcal{H}^{d} \leq C_1(n)r\int_{\Sigma_{r, \infty}}u^2e^{-\frac{|x|^2}{4}}d\mathcal{H}^{d}.
\end{equation}
By possibly slightly shrinking $\overline{r}(n,m) > 0$, we can absorb some parts of the integral on the right hand side and finally obtain 
\begin{equation}\label{inner outer integral1}
    \frac{1}{r^2}\int_{\Sigma_{r, 2r}}u^2e^{-\frac{|x|^2}{4}}d\mathcal{H}^{d} \leq C_2(n,m)\int_{\Sigma_{2r, \infty}}u^2e^{-\frac{|x|^2}{4}}d\mathcal{H}^{d}.
\end{equation}
The above estimate combined with \eqref{alm final form of innerouter} immediately implies 
\begin{equation}\label{inner outer integral2}
    \int_{\Sigma_{r, \infty}}|\nabla u|^2e^{-\frac{|x|^2}{4}}d\mathcal{H}^{d} \leq C_3(n,m)\int_{\Sigma_{2r, \infty}}u^2e^{-\frac{|x|^2}{4}}d\mathcal{H}^{d}.
\end{equation}
To control the remaining integrals involving $u^2(|x|^2 + \frac{1}{|x|^2})$, we combine the previous inequality with lemma \ref{Poincare type inequality for u^2/r^2}. As for the integral involving $\frac{u^2}{|x|^2}$, we use the first inequality in lemma \ref{Poincare type inequality for u^2/r^2}, drop the boundary integral, and then use previous inequalities \eqref{inner outer integral1}, \eqref{inner outer integral2}. Finally, to control the integral involving $u^2|x|^2$, we use the second inequality in lemma \ref{Poincare type inequality for u^2/r^2}. Since $r \leq \overline{r}(n,m)$, the boundary integral can be controlled by \eqref{alm final form of innerouter}. The remaining integrals can be controlled by \eqref{inner outer integral1}, \eqref{inner outer integral2}, thus completing the proof of proposition \ref{inner outer estimate}.
\end{proof}

\section{Unique asymptotics in the parabolic region}\label{uniqueasymptoticssection}
In this section, we consider a smooth, properly embedded ancient mean curvature flow $(M_t)_{t \in (-\infty, 0)}$ that satisfies assumptions \ref{strong convergence to simonscone assumption1}, \ref{additional geoemtric assumption}. By using the graphical radius estimate (proposition \ref{graphical radius estimate proposition}), and the inner outer estimate (corollary \ref{inner outer estimate for RMCF corollary}), we suitably control the nonlinear error terms, which allows us to isolate the dynamics of the flow onto a few eigenmodes of the linearized rescaled mean curvature operator \eqref{linearzed RMCF at cone} (proposition \ref{evolution of projections of localized functions}). We then use the assumption \ref{additional geoemtric assumption} to show that the first eigenmode ($\varphi_1$ in subsection \ref{eigenfunctions list}) has to be dominant, which leads to unique asymptotics in the parabolic region (theorem \ref{unique asymptotics of one sided flow}). \\

Throughout this section, for each $x = (\sqrt{\frac{n-1}{n+m-2}}rw_1, \sqrt{\frac{m-1}{n+m-2}}rw_2) \in \Sigma$ for some $r > 0$, $w_1 \in \mathbf{S}^{n-1}, w_2 \in \mathbf{S}^{m-1}$, we always define the unit normal $\nu_{\Sigma}(x) = (-\sqrt{\frac{m-1}{n+m-2}}w_1, \sqrt{\frac{n-1}{n+m-2}}w_2)$.\\

Let us first compute the rescaled mean curvature flow of normal graphs over the $O(n)\times O(m)$ symmetric minimal quadratic cone $\Sigma$. Following the arguments in lemma 3.6 in \cite{Chodosh2024-qv}, we can write down the structure of the nonlinear terms.
\begin{lemma}\label{the nonlinear error term computation}
    Let $(\Tilde{M}_{\tau})_{\tau \in I}$ be a smooth solution to rescaled mean curvature flow. If $\Tilde{M}_{\tau}$ can be written as a normal graph of $u_{\tau} \in C^{\infty}(\mathcal{D}_{\tau} \subset \Sigma \setminus 
    \{0\})$, then there exists $\overline{\eta} = \overline{\eta}(n,m) \in (0, \frac{1}{100})$ so that whenever $\|u_{\tau}\|^*_{C^2(\mathcal{D}_{\tau})} \leq \overline{\eta}$, then
    \begin{equation}
        \partial_{\tau}u_{\tau} = Lu_{\tau} + Qu_{\tau},
    \end{equation}
    where $L$ is the linearized operator given in definition \ref{linearized rescaled mean curvature flow equation over the cone}, and $Qu_{\tau} = Q(x, u_{\tau}, \nabla u_{\tau}, \nabla^2u_{\tau})$ can be decomposed into
\begin{align*}
    Q(x, u, \nabla u, \nabla^2u) = &uQ_1(x, u, \nabla u, \nabla^2u) + \nabla u \cdot Q_2(x, u, \nabla u, \nabla^2u) \\ &+ u  Q_3(x, u, \nabla u) + \nabla u \cdot Q_4(x, u, \nabla u).
\end{align*}
  Here, $Q_1$, $Q_2$, $Q_3$, $Q_4$ satisfy the estimates
  \begin{equation}\label{scaling behavior of errorterm1}
      |x|^{2 +i+ j -l}|\nabla^i_x\nabla^j_z\nabla^k_q\nabla^l_AQ_1(x,z,q,A)|\leq C(n,m,i)(\frac{|z|}{|x|} + |q| + |x||A|)^{\max(0, 1 - j-k-l)},
  \end{equation}
  \begin{equation}\label{scaling behavior of errorterm2}
      |x|^{1 +i + j -l}|\nabla^i_x\nabla^j_z\nabla^k_q\nabla^l_AQ_2(x,z,q,A)|\leq C(n,m,i)(\frac{|z|}{|x|} + |q| + |x||A|)^{\max(0, 1 - j-k-l)},
  \end{equation}
   \begin{equation}\label{scaling behavior of errorterm3}
      |x|^{ i + j }|\nabla^i_x\nabla^j_z\nabla^k_qQ_3(x,z,q)|\leq C(n,m,i)(\frac{|z|}{|x|} + |q| )^{\max(0, 1 - j-k)},
  \end{equation}
  and
  \begin{equation}\label{scaling behavior of errorterm4}
      |x|^{-1 +i+ j}|\nabla^i_x\nabla^j_z\nabla^k_qQ_4(x,z,q)|\leq C(n,m,i)(\frac{|z|}{|x|} + |q|)^{\max(0, 1 - j-k)}.
  \end{equation}
\end{lemma}
\begin{proof}[Proof of lemma \ref{the nonlinear error term computation}]
We can parametrize $\Tilde{M}_{\tau}$ by
\begin{equation}
    \mathcal{D}_{\tau} \ni x \to x + u_{\tau}(x)\nu_{\Sigma}(x) = X_{\tau}(x)\in \Tilde{M}_{\tau}.
\end{equation}
    We define $\nu_{\tau}$ to be the unit normal vector field along $\Tilde{M}_{\tau}$ so that $\nu_{\Sigma}\cdot \nu_{\tau} > 0$. Then the rescaled mean curvature flow is equivalent to 
\begin{equation}
    \partial_{\tau}u_{\tau} = v_{\tau}(-H_{\tau} + \frac{1}{2}X_{\tau}\cdot \nu_{\tau}) = (\nu_{\tau}\cdot \nu_{\Sigma})^{-1}(-H_{\tau} + \frac{1}{2}X_{\tau}\cdot \nu_{\tau}).
\end{equation}
We can then write the right hand side as
\begin{align*}
    v_{\tau}(-H_{\tau} + \frac{1}{2}X_{\tau}\cdot \nu_{\tau}) &= Lu_{\tau} + (-v_{\tau}H_{\tau} - \Delta_{\Sigma}u_{\tau} - \frac{n+m-2}{|x|^2}u_{\tau}) + \frac{1}{2}(v_{\tau}X_{\tau}\cdot \nu_{\tau} + x\cdot \nabla^{\Sigma}u_{\tau} - u_{\tau})\\ &= Lu_{\tau} + Q^Hu_{\tau} + Q^{X\cdot \nu}u_{\tau}.
\end{align*}
Since $v_{\tau}H_{\tau}$ is a second order quasilinear operator in $u_{\tau}$, and $\Delta_{\Sigma}u_{\tau} + \frac{n+m-2}{|x|^2}u_{\tau}$ is the linearization of $-v_{\tau}H_{\tau}$ in $(u_{\tau}, \nabla u_{\tau}, \nabla^2u_{\tau})$ at $(0,0,0)$, we see that
\begin{equation}
    Q^Hu= uQ_1(x, u, \nabla u, \nabla^2u) + \nabla u \cdot Q_2(x, u, \nabla u, \nabla^2u)
\end{equation}
for some $Q_1, Q_2$. By the scaling behavior of the mean curvature, one can check that for any $x \in \Sigma \setminus \{0\}, z \in \mathbf{R}, p \in T_x\Sigma, A \in T_x\Sigma\otimes T_x^*\Sigma$ with small enough $\frac{|z|}{|x|}, |p|, |x||A|$, we have
\begin{equation}
    Q_1(x,z,p,q) = |x|^{-2}Q_1(\frac{x}{|x|}, \frac{z}{|x|}, p, |x|A),
\end{equation}
and
\begin{equation}
    Q_2(x,z,p,q) = |x|^{-1}Q_2(\frac{x}{|x|}, \frac{z}{|x|}, p, |x|A).
\end{equation}
Then by using the scaling behavior above, together with the fact that $Q^H$ is the at least quadratic error term of $vH$ for small $\frac{|z|}{|x|} + |p| + |x||A|$, we have the estimates \eqref{scaling behavior of errorterm1}, \eqref{scaling behavior of errorterm2}.\\

$Q^{X\cdot \nu}u = Q^{X\cdot \nu}(x, u, \nabla u)$ can be dealt with in the same manner. Since this is the at least quadratic error term of $v_{\tau}X_{\tau}\cdot \nu_{\tau}$, we can write it as
\begin{equation}
    Q^{X\cdot \nu}(x, z, p) = zQ_3(x,z,p) + pQ_4(x,z,p).
\end{equation}
By the scaling behavior of $X$, we see that
\begin{equation}
    Q_3(x,z,p) = Q_3(\frac{x}{|x|}, \frac{z}{|x|}, p),
\end{equation}
and
\begin{equation}
    Q_4(x,z,p) = |x|Q_4(\frac{x}{|x|}, \frac{z}{|x|}, p).
\end{equation}
Above scaling relations immediately imply \eqref{scaling behavior of errorterm3}, \eqref{scaling behavior of errorterm4}. 
\end{proof}
Next, we prove the following weighted $C^2$ estimate for $u_{\tau}$ that will be useful when estimating the nonlinear error terms. This is a consequence of the avoidance principle against the self-shrinkers we construct in proposition \ref{inner family of self shrinkers}, and proposition \ref{Existence of trumpets}, and interpolation.
\begin{lemma}\label{weighted C2 estimate for the graph function giving us higher power nonlinearity}
    Let $(M_t)_{t \in (-\infty, 0)}$ be a smooth, properly embedded ancient mean curvature flow which satisfies assumptions \ref{strong convergence to simonscone assumption1}, \ref{additional geoemtric assumption}, and let $(\Tilde{M}_{\tau})_{\tau \in (-\infty, 0)}$ be the rescaled mean curvature flow. Then for any $\delta \in (0, \frac{1}{100})$, there exist $0 < \overline{r}(n,m) < \overline{L}(n,m) < \infty$, $\overline{\tau} < 0$, $C = C(n,m,  \delta) > 0$ so that the following holds: if we define 
\begin{equation}\label{definition of graphical radius function for localization}
    r(\tau) = r_{\delta}(\tau) = \rho_0^{\frac{1}{1 - \alpha} - \delta} , \ \ \rho_0(\tau) = \sup_{s \leq \tau} \|u_{s}\|_{L^2_w(\Sigma_{\frac{\overline{r}}{2}, 2\overline{L}})}
\end{equation}
for each $\tau \leq \overline{\tau}$, then
\begin{equation}
   \|u_{\tau}\|^*_{C^2(\Sigma_{r(\tau), \infty})} \leq C\rho_0^{\frac{\delta(1 - \alpha)}{3}}(\tau).
\end{equation}
\end{lemma}
\begin{proof}[Proof of lemma \ref{weighted C2 estimate for the graph function giving us higher power nonlinearity}]
We select $\overline{r}(n,m) > 0$, $\overline{\eta}(n,m)$, $\overline{\tau} < 0$ from the graphical radius estimate (proposition \ref{graphical radius estimate proposition}) with $\eta = \overline{\eta}(n,m)$. Also, by possibly taking the maximum with $10 \overline{r}(n,m) $, we set $\overline{L}(n) > 10\overline{r}(n,m)$ to be the constant from corollary \ref{normal graph function estimate of trumpets}.\\

First, by possibly shrinking $\overline{\eta}(n,m) > 0$, we may apply interior Schauder estimates, and obtain
\begin{equation}\label{standard interior schauder}
    \sup_{s \leq \tau,  x \in \Sigma_{\overline{r}, \overline{L}}}|u_s(x)| \leq c_0(n,m)\rho_0(\tau).
\end{equation}
In view of proposition \ref{graphical radius estimate proposition}, and that $\rho_0(\tau) \to 0$ as $\tau \to -\infty$ by assumption \ref{strong convergence to simonscone assumption1}, by possibly shrinking $\overline{\tau} < 0$, we may further assume that
\begin{equation}\label{compariosn inequality with graphical radius estimate}
     10r(\tau ; \overline{\eta}) \leq 10C_0(\sup_{s \leq \tau} \|u_{s}\|_{L^2_w(\Sigma_{\frac{\overline{r}}{2}, 2\overline{r}})})^{\frac{1}{1 - \alpha}} \leq 10C_0\rho_0^{\frac{1}{1 - \alpha}} \leq r_{\delta}(\tau) < \frac{\overline{r}(n,m)}{100},
\end{equation}
where the constant $C_0 = C_0(n,m ) > 0$ is from proposition \ref{graphical radius estimate proposition} with $\eta = \overline{\eta}(n,m)$. 
 We claim that 
    \begin{equation}\label{C0 estimate}
   \|u_{\tau}\|^*_{C^0(\Sigma_{\frac{r(\tau)}{2},\infty} )} \leq C(n, m,\delta)\rho_0^{\delta(1 - \alpha)}(\tau).
\end{equation}
Then lemma \ref{weighted C2 estimate for the graph function giving us higher power nonlinearity} follows by interpolation (proposition \ref{interpolation inequalities of weighted Ck norms}) between the above $C^0$ estimate, and the weighted $C^3$ bound given by $\overline{\eta}(n,m) > 0$.\\

We split $\Sigma_{\frac{r(\tau)}{2}, \infty}$ into three regions, namely $D_1 = \Sigma_{\frac{r(\tau)}{2}, \overline{r}(n,m)}$, $D_2 = \Sigma_{\overline{r}, \overline{L}}$, and $D_3 = \Sigma_{\overline{L}, \infty}$. \\

We first consider $D_1$. By \eqref{compariosn inequality with graphical radius estimate}, we can apply proposition \ref{inner region weighted C0 estimate}. Then for $x \in D_1$, $\tau \leq \overline{\tau} < 0$, we have 
\begin{equation}
    \frac{|u_{\tau}(x)|}{|x|} \leq c(n,m)\rho_0(\tau)(\frac{r_{\delta}(\tau)}{2})^{\alpha - 1}\leq c(n,m)\rho_0^{1 + (\frac{1}{1 - \alpha} - \delta)(\alpha - 1)} = c(n,m)\rho_0^{\delta(1 - \alpha)}(\tau),
\end{equation}
which implies that
\begin{equation}
    \|u_{\tau}\|^*_{C^0(D_1)} \leq c(n,m)\rho_0^{\delta(1 - \alpha)}(\tau).
\end{equation}
For $D_2$, since the weighted norm and the unweighted norm are uniformly comparable, \eqref{standard interior schauder} immediately implies 
\begin{equation}
    \|u_{\tau}\|^*_{C^0(D_2)} \leq c(n,m)\rho_0^{\delta(1 - \alpha)}(\tau).
\end{equation}
In the $D_3$ region, we apply corollary \ref{normal graph function estimate of trumpets} (more precisely, corollary \ref{normal graph function estimate of trumpets} with $n,m$ reversed since we will use $\Tilde{T}_{\sigma}$ in remark \ref{remark, trumpets from reflection} as barriers). First, note that $\rho_0(\tau )> 0$. If this is not the case for some $\tau < 0$, then 
\begin{equation}
    u_s(x) \equiv 0 \text{ for all }s \leq \tau, \ x \in \Sigma_{\overline{r}, \overline{L}}
\end{equation}
Due to the one-sidedness assumption \ref{additional geoemtric assumption}, together with connectedness of the flow (corollary \ref{connectedness}), the strong maximum principle implies that $\Tilde{M}_{\tau} \equiv \Sigma$, which contradicts smoothness of the flow. Combining this with \eqref{sigma estimate for trumpet}, we see that by possibly making $\overline{\tau} < 0$ smaller, we may assume that if for any $|x| = \overline{L}$, 
\begin{equation}\label{choice of sigma for comparison}
    \sup_{s \leq \tau, |y| = \overline{L}} |u_{s}(y)| \leq c_0(n,m)\rho_0(\tau) = \hat{v}_{\sigma}(x), 
\end{equation}
then $\sigma \in (\frac{1}{2}\sqrt{\frac{n-1}{m-1}}, \sqrt{\frac{n-1}{m-1}})$. Since both the `trumpet' $\Tilde{T}_{\sigma}$, and $\Tilde{M}_{\tau}$ are normal graphs over $\Sigma_{\overline{L}, \infty}$ for all $\tau \leq \overline{\tau}$ with at most linear growth in $|x|$, we can apply the weak maximum principle to the normal graph functions. In view of proposition \ref{Rough graphicality property}, and the fact that the `trumpet' has a strict linear growth, we can find some very negative time $s_* < \tau \leq \overline{\tau}$ so that 
\begin{equation}
    |u_{s}(x)| \leq \hat{v}_{\sigma}(x) \text{ for all }x \in \Sigma_{\overline{L}, \infty}, \ s \leq s_*.
\end{equation}
In view of \eqref{choice of sigma for comparison}, the weak maximum principle with $s_* < \tau$ as the initial time implies that 
\begin{equation}
    |u_{s}(x)| \leq \hat{v}_{\sigma}(x) \text{ for all }x \in \Sigma_{\overline{L}, \infty}, \ s \leq \tau.
\end{equation}
Combining this with corollary \ref{normal graph function estimate of trumpets} implies that
\begin{equation}
    \sup_{s \leq \tau}\sup_{x \in \Sigma_{\overline{L}, \infty}}\frac{|u_{s}(x)|}{|x|} \leq c(n,m)\frac{|\hat{v}_{\sigma}(\overline{L})|}{\overline{L}} = C(n,m)\rho_0(\tau) \leq C(n,m)\rho_0^{\delta(1 - \alpha)}(\tau),
\end{equation}
thus proving the desired estimate in $D_3$.\\

Combining the three domains, we obtain \eqref{C0 estimate}, thus proving lemma \ref{weighted C2 estimate for the graph function giving us higher power nonlinearity}.
\end{proof}
We now define the localized graph function $\Tilde{u}_{\tau}$. We follow the setup in lemma \ref{weighted C2 estimate for the graph function giving us higher power nonlinearity} with $\delta = \delta(n,m) \in (0, \frac{1}{100})$ that will be determined later. We first fix a cutoff function
\begin{equation}
    \xi \in C^{\infty}([0, \infty)), \ \xi \equiv 0 \text{ in }[0,1], \ \xi \equiv 1 \text{ in }[2, \infty), \xi ' \geq 0,
\end{equation}
and define for each $x \in \Sigma$
\begin{equation}\label{definition of cutoff}
    \xi_{\tau}(x) = \xi(\frac{|x|}{r(\tau)}).
\end{equation}
Note that $r(\tau) > 0$. If this is not the case, then there exists $\tau  > -\infty$ so that for all $s \leq \tau$, $\Tilde{M}_{s}$ agrees with $\Sigma$ on an open set $\Sigma_{\overline{r}, \overline{L}} \subset \mathcal{D}$. Since $\Tilde{M}_{s}$ is smooth, connected (corollary \ref{connectedness}), and one-sided by assumption \ref{additional geoemtric assumption}, the strong maximum principle implies that the flow has to be identical to the cone globally. This contradicts the smoothness of the flow, thus $r(\tau) > 0$. This ensures that $\xi_{\tau}$ is well defined. Moreover, since the $L^2_w$ norm of the graph function over the fixed annulus region asymptotically vanishes as $\tau \to -\infty$, positivity of $r(\tau)$ implies that the supremum is actually is attained, i.e
\begin{equation}
    r(\tau) = (\sup_{s \leq \tau}\|u_s\|_{L^2_w(\Sigma_{\frac{\overline{r}}{2}, 2\overline{L}})})^{\frac{1}{1 - \alpha} - \delta} = (\max_{s \leq \tau}\|u_s\|_{L^2_w(\Sigma_{\frac{\overline{r}}{2}, 2\overline{L}})})^{\frac{1}{1 - \alpha} - \delta}.
\end{equation}
This ensures that $r(\tau) > 0$ is locally Lipschitz in $\tau$, almost everywhere differentiable with
\begin{equation}\label{derivative of graphical radius function for localization}
    \frac{d}{d\tau}r(\tau) = \frac{d}{ds}(\|u_s\|^{\frac{1}{1 - \alpha} - \delta}_{L^2_w(\Sigma_{\frac{\overline{r}}{2}, 2\overline{L}})}),
\end{equation}
where the $s$-derivative is evaluated at any $s \leq \tau$ where $r(\tau)$ is attained. By looking at the proof of lemma \ref{weighted C2 estimate for the graph function giving us higher power nonlinearity}, we see that on $\Sigma_{\frac{\overline{r}}{4}, 4\overline{L}}$, the unweighted $C^0$ norm of $u_{s}$ is controlled linearly by $ \rho_0(\tau)$ for all $s \leq \tau$. Thus by applying interior Schauder estimates, we see that for all sufficiently negative $\tau$, 
\begin{equation}\label{bound on the derivative of graphical radius}
    | \frac{d}{d\tau}r(\tau)| \leq C(n, m,\delta)\rho_0^{\frac{1}{1 - \alpha} - \delta}. 
\end{equation}
 The localized graph function is given by
\begin{equation}\label{definiton of localized graph function}
    \Tilde{u}_{\tau} = u_{\tau}\xi_{\tau}.
\end{equation}
It is clear from the definition of graphical radius (definition \ref{definition of graphical radius}) that $\Tilde{u}_{\tau} \in C^{\infty}(\Sigma \setminus \{0\}) \cap L^2_w(\Sigma)$. Therefore we can define its projection onto the eigenmodes of the linearized operator $L$. We define for each $\tau \leq \overline{\tau}$, the projection of $\Tilde{u}_{\tau}$ onto the positive / zero / negative eigenmodes 
\begin{equation}\label{definition of projection onto eigenmodes}
    U^+_{\tau} = \sum_{i = 1}^3\sum_{j = 1}^{J(i)}\langle \Tilde{u}_{\tau}, \varphi_{i,j}\rangle_{L^2_w}\varphi_{i,j}, \ U^0_{\tau}  = \sum_{j = 1}^{J(4)}\langle \Tilde{u}_{\tau}, \varphi_{4,j}\rangle_{L^2_w}\varphi_{4,j},\ U^-_{\tau} = \Tilde{u}_{\tau} - U^+_{\tau} - U^0_{\tau},
\end{equation}
where $\varphi_{i,j}$ are defined in subsection \ref{eigenfunctions list}.
We now compute the evolution of the projections. 
\begin{proposition}\label{evolution of projections of localized functions}
    There exist $\delta(n,m) \in (0, \frac{1}{100})$, $p(n,m) > 0$, $\overline{\tau}< 0$ so that the projections satisfy the following inequalities for all $\tau \leq \overline{\tau}$: 
    \begin{equation}
        \| \frac{d}{d\tau}U^*_{\tau} - LU^*_{\tau} \|_{L^2_{w}} \leq C(n,m)\rho_0^{p(n,m)}(\| U^+_{\tau}\|_{L^2_w} + \| U^0_{\tau}\|_{L^2_w} + \| U^-_{\tau} \|_{L^2_w} +\rho_0),
    \end{equation}
    where $ * = +, 0$, and
    \begin{equation}
        \frac{d}{d\tau}\|U^-_{\tau}\|_{L^2_w}^2 \leq \frac{\lambda_5(n,m)}{2}\|U^-_{\tau}\|_{L^2_w}^2 +  C(n,m)\rho_0^{2p(n,m)}(\| U^+_{\tau}\|_{L^2_w}^2 + \| U^0_{\tau}\|^2_{L^2_w}+ \| U^-_{\tau} \|_{L^2_w}^2 + \rho_0^2),
    \end{equation}
where $\rho_0$ is given by lemma \ref{weighted C2 estimate for the graph function giving us higher power nonlinearity}. Here, $\lambda_5(n,m) < 0$ is the first negative eigenvalue of $L$.
\end{proposition}
\begin{proof}[Proof of proposition \ref{evolution of projections of localized functions}]
     We first compute the first inequality. To do so, set $\varphi = \varphi_{i,j}$ given in subsection \ref{eigenfunctions list}, and define 
    \begin{equation}
        a_{\varphi}(\tau) = \langle \Tilde{u}_{\tau} , \varphi \rangle_{L^2_w}.
    \end{equation}
    From now on, we omit the subscript $L^2_w$ in the inner products, and norms whenever the integral is defined on the entire $\Sigma$. By lemma \ref{the nonlinear error term computation}, 
    \begin{align*}
        \frac{d}{d\tau}a_{\varphi}(\tau) & = \langle \partial_{\tau}\Tilde{u}_{\tau}, \varphi\rangle = \langle L\Tilde{u}_{\tau} , \varphi \rangle + \langle \xi_{\tau}Qu_{\tau}, \varphi \rangle + \langle E_{\tau}, \varphi \rangle, 
    \end{align*}
    where $E_{\tau} = u_{\tau}\partial_{\tau} \xi_{\tau} + [\xi_{\tau}, L]u_{\tau} $. We now estimate each term separately. By definition of $\varphi = \varphi_{i,j}$, we have 
   \begin{equation}\label{estimate of linear term positive mode}
       \langle L\Tilde{u}_{\tau} , \varphi \rangle = \lambda_i \langle \Tilde{u}_{\tau}, \varphi \rangle = \lambda_ia_{\varphi}(\tau).
   \end{equation}
   We use the inner outer estimate (corollary \ref{inner outer estimate for RMCF corollary}), and lemma \ref{weighted C2 estimate for the graph function giving us higher power nonlinearity} to estimate the second term. By lemma \ref{the nonlinear error term computation}, we can write
   \begin{align*}
       \langle \xi_{\tau}Qu_{\tau}, \varphi \rangle = \langle \xi_{\tau}u_{\tau}(Q_1 + Q_3), \varphi\rangle + \langle \xi_{\tau}\nabla u_{\tau}\cdot (Q_2 + Q_4) , \varphi \rangle. 
   \end{align*}
   By lemma \ref{weighted C2 estimate for the graph function giving us higher power nonlinearity},  $\|u_{\tau}\|^*_{C^2(\Sigma_{r(\tau), \infty})} \leq C(n,m, \delta)\rho_0^{\frac{\delta(1 - \alpha)}{3}}$. By using the nonlinear structure given in lemma \ref{the nonlinear error term computation}, we have the pointwise estimates
   \begin{equation}
       |\xi_{\tau}u_{\tau}(Q_1 + Q_3)| \leq C(n,m, \delta)\rho_0^{\frac{\delta(1 - \alpha)}{3}}|u_{\tau}|(\frac{1}{|x|^2} + 1)\chi_{\Sigma_{r(\tau), \infty}},
   \end{equation}
   and
   \begin{equation}
       |\xi_{\tau}\nabla u_{\tau}\cdot (Q_2 + Q_4)| \leq C(n,m, \delta)\rho_0^{\frac{\delta(1 - \alpha)}{3}}|\nabla u_{\tau}|(|x| + \frac{1}{|x|})\chi_{\Sigma_{r(\tau), \infty}}.
   \end{equation}
   By the inner outer estimate (corollary \ref{inner outer estimate for RMCF corollary}) together with the fact that 
   \begin{equation}
       \int_{\Sigma}\varphi^2(\frac{1}{|x|^2} + |x|^2)e^{-\frac{|x|^2}{4}}d\mathcal{H}^{d} \leq C(n,m)
   \end{equation}
   for any $\varphi$ given in subsection \ref{eigenfunctions list}, we have
   \begin{align*}
       |\langle \xi_{\tau}u_{\tau}(Q_1 + Q_3), \varphi \rangle| & \leq C(n,m, \delta)\rho_0^{\frac{\delta(1 - \alpha)}{3}}\int_{\Sigma_{r(\tau), \infty}}\frac{|u_{\tau}|}{|x|} (|x| + \frac{1}{|x|})|\varphi| e^{-\frac{|x|^2}{4}}d\mathcal{H}^{d} \\ & \leq C(n,m, \delta)\rho_0^{\frac{\delta(1 - \alpha)}{3}}(\int_{\Sigma_{r(\tau), \infty}}\frac{|u_{\tau}|^2}{|x|^2} e^{-\frac{|x|^2}{4}}d\mathcal{H}^{d})^{\frac{1}{2}} \\ & \leq C(n,m, \delta)\rho_0^{\frac{\delta(1 - \alpha)}{3}}(\int_{\Sigma_{2r(\tau), \infty}}u_{\tau}^2 e^{-\frac{|x|^2}{4}}d\mathcal{H}^{d})^{\frac{1}{2}} \\ & \leq C(n,m, \delta)\rho_0^{\frac{\delta(1 - \alpha)}{3}} (\| U^+_{\tau} \| + \| U^0_{\tau} \| +\| U^-_{\tau} \|).
   \end{align*}
   Likewise 
   \begin{align*}
       |\langle \xi_{\tau}\nabla  u_{\tau} \cdot(Q_2 + Q_4), \varphi\rangle| & \leq C(n, m,\delta)\rho_0^{\frac{\delta(1 - \alpha)}{3}}\int_{\Sigma_{r(\tau), \infty}}|\nabla  u_{\tau}| (|x| + \frac{1}{|x|})|\varphi| e^{-\frac{|x|^2}{4}}d\mathcal{H}^{d} \\ & \leq C(n,m, \delta)\rho_0^{\frac{\delta(1 - \alpha)}{3}}(\int_{\Sigma_{r(\tau), \infty}}|\nabla  u_{\tau}|^2e^{-\frac{|x|^2}{4}}d\mathcal{H}^{d})^{\frac{1}{2}} \\ & \leq C(n, m,\delta)\rho_0^{\frac{\delta(1 - \alpha)}{3}}(\int_{\Sigma_{2r(\tau), \infty}}u_{\tau}^2 e^{-\frac{|x|^2}{4}}d\mathcal{H}^{d})^{\frac{1}{2}} \\ & \leq C(n, m,\delta)\rho_0^{\frac{\delta(1 - \alpha)}{3}} (\| U^+_{\tau} \| +\| U^0_{\tau} \| + \| U^-_{\tau} \|).
   \end{align*}
   Combining both estimates, we obtain
   \begin{equation}\label{est of nonlinear term from nonlinearity of equation in positive mode}
       |\langle \xi_{\tau}Qu_{\tau}, \varphi \rangle| \leq C(n, m,\delta)\rho_0^{\frac{\delta(1 - \alpha)}{3}} (\| U^+_{\tau} \| +\| U^0_{\tau} \| + \| U^-_{\tau} \|).
   \end{equation}
   Finally using the fact that $E_{\tau}$ is supported in $\Sigma_{r(\tau), 2r(\tau)}$, together with \eqref{bound on the derivative of graphical radius}, lemma \ref{the nonlinear error term computation}, lemma \ref{weighted C2 estimate for the graph function giving us higher power nonlinearity}, we have for $x \in \Sigma_{r(\tau), 2r(\tau)}$
\begin{equation}
    |E_{\tau}|(x) \leq C(n,m)(|r'(\tau)|\frac{|u_{\tau}|}{r(\tau)} + \frac{|u_{\tau}|}{r^2(\tau)} + \frac{|\nabla u_{\tau}|}{r(\tau)})  \leq C(n,m, \delta)r^{-1}(\tau).
\end{equation}
Integrating the above estimate over $\Sigma_{r(\tau), 2r(\tau)}$ gives us
\begin{equation}
    \langle E_{\tau}, \varphi \rangle \leq \|E_{\tau} \| \leq C(n,m, \delta)r^{-1}(\tau)r(\tau)^{\frac{n+m-2}{2}} \leq C(n,m, \delta)\rho_0^{(\frac{n+m}{2}-2)(\frac{1}{1-\alpha} - \delta)}.
\end{equation}
By the definition of $\alpha\in (-2, -1)$ given by \eqref{the formula for alpha}, one can see that as long as $n +m\geq 10$, we can choose $\delta(n,m) \in (0, \frac{1}{100})$ small enough so that
\begin{equation}
   (\frac{n+m}{2}-2)(\frac{1}{1-\alpha} - \delta) = 1 + p(n,m) > 1, 
\end{equation}
which gives us 
\begin{equation}\label{est fo error term fromcutoff positive mode}
    \langle E_{\tau}, \varphi \rangle \leq \|E_{\tau} \|  \leq C(n,m)\rho_0^{1 + p(n,m)}(\tau).
\end{equation}
By possibly replacing $p(n,m)$ with $\min(p(n,m), \frac{\delta(n,m)(1 - \alpha)}{3}) > 0$, and combining \eqref{estimate of linear term positive mode}, \eqref{est of nonlinear term from nonlinearity of equation in positive mode}, \eqref{est fo error term fromcutoff positive mode}, we obtain
\begin{equation}
    |\frac{d}{d\tau}a_{\varphi}(\tau) - \lambda_ia_{\varphi}(\tau)| \leq C(n,m)\rho_0^{p(n,m)}(\tau)(\| U^+_{\tau} \| + \| U^0_{\tau} \| +\| U^-_{\tau} \| +\rho_0(\tau)).
\end{equation}
By considering all $\varphi = \varphi_{i,j}$, we have the first inequality in proposition \ref{evolution of projections of localized functions}. \\

The second inequality also follows from a similar line of computation. We first compute the evolution of $\|U^-_{\tau}\|^2_{L^2_w}$.
\begin{equation}
    \frac{d}{d\tau}\|U^-_{\tau}\|^2_{L^2_w} = 2\langle \partial_{\tau}\Tilde{u}_{\tau}, U^-_{\tau} \rangle = 2\langle L\Tilde{u}_{\tau} ,U^-_{\tau}\rangle  + 2\langle \xi_{\tau}Qu_{\tau}, U^-_{\tau} \rangle +2\langle E_{\tau}, U^-_{\tau} \rangle. 
\end{equation}
We now estimate each term separately. First, since the negative eigenvalues of $L$ are bounded from above by $\lambda_5(n,m) < 0$, we have
\begin{equation}\label{est of linear term negative mode}
    2\langle L\Tilde{u}_{\tau} ,U^-_{\tau}\rangle =  2\langle LU^-_{\tau} ,U^-_{\tau}\rangle \leq 2\lambda_5(n,m)\| U^-_{\tau} \|^2.
\end{equation}
To estimate $2\langle \xi_{\tau}Qu_{\tau}, U^-_{\tau} \rangle$, we further decompose $Q$ via lemma \ref{the nonlinear error term computation}. This gives us
\begin{align*}
       \langle \xi_{\tau}Qu_{\tau}, U^-_{\tau} \rangle = \langle \xi_{\tau}u_{\tau}(Q_1 + Q_3), U^-_{\tau} \rangle + \langle \xi_{\tau}\nabla u_{\tau}\cdot (Q_2 + Q_4) ,U^-_{\tau} \rangle. 
   \end{align*}
   As before, we can estimate each term on the right hand side as follows by using the inner outer estimate (corollary \ref{inner outer estimate for RMCF corollary}), and lemma \ref{weighted C2 estimate for the graph function giving us higher power nonlinearity}. In the following computations, we continue to use the constants $\delta(n,m), p(n,m) > 0$ determined previously.
   \begin{align*} 
       | \langle \xi_{\tau}u_{\tau}(Q_1 + Q_3), U^-_{\tau} \rangle| & \leq C(n,m)\rho_0^{ p(n,m)}\int_{\Sigma_{r(\tau), \infty}}|u_{\tau}|(|x| + \frac{1}{|x|})\frac{|U^-_{\tau}|}{|x|}e^{-\frac{|x|^2}{4}}d\mathcal{H}^{d}\\ & \leq C(n,m)\rho_0^{ p(n,m)}(\int_{\Sigma_{r(\tau), \infty}}\frac{|u_{\tau}|^2}{|x|^2}e^{-\frac{|x|^2}{4}}d\mathcal{H}^{d})^{\frac{1}{2}}\cdot(\int_{\Sigma_{r(\tau), \infty}}(|x|^2 + \frac{1}{|x|^2})|U^-_{\tau}|^2e^{-\frac{|x|^2}{4}}d\mathcal{H}^{d})^{\frac{1}{2}} \\ & \leq C(n,m)\rho_0^{ p(n,m)}\| \Tilde{u}_{\tau}\|^2_{L^2_w} + C(n,m)\rho_0^{ p(n,m)} \int_{\Sigma_{r(\tau), \infty}}(|x|^2 + \frac{1}{|x|^2})|U^-_{\tau}|^2e^{-\frac{|x|^2}{4}}d\mathcal{H}^{d},
   \end{align*}
   and
    \begin{align*}
       | \langle \xi_{\tau}\nabla u_{\tau}\cdot (Q_2 + Q_4), U^-_{\tau} \rangle | & \leq C(n,m)\rho_0^{ p(n,m)}\int_{\Sigma_{r(\tau), \infty}}|\nabla u_{\tau}|(|x| + \frac{1}{|x|})|U^-_{\tau}|e^{-\frac{|x|^2}{4}}d\mathcal{H}^{d}\\ & \leq C(n,m)\rho_0^{ p(n,m)}(\int_{\Sigma_{r(\tau), \infty}}|\nabla u_{\tau}|^2e^{-\frac{|x|^2}{4}}d\mathcal{H}^{d})^{\frac{1}{2}}\cdot(\int_{\Sigma_{r(\tau), \infty}}(|x|^2 + \frac{1}{|x|^2})|U^-_{\tau}|^2e^{-\frac{|x|^2}{4}}d\mathcal{H}^{d})^{\frac{1}{2}} \\ & \leq C(n,m)\rho_0^{ p(n,m)}\| \Tilde{u}_{\tau}\|^2_{L^2_w}+ C(n,m)\rho_0^{ p(n,m)}\int_{\Sigma_{r(\tau), \infty}}(|x|^2 + \frac{1}{|x|^2})|U^-_{\tau}|^2e^{-\frac{|x|^2}{4}}d\mathcal{H}^{d}.
   \end{align*}
   To estimate the last integrals in both estimates, we use the two Poincar\'e type inequalities in lemma \ref{Poincare type inequality for u^2/r^2}. Note that since $U^-_{\tau} = \Tilde{u}_{\tau} - U^+_{\tau}- U^0_{\tau}$, we see that $U^-_{\tau}$ does indeed have polynomial growth in $|x|$ as $|x| \to \infty$; hence we can apply the Poincar\'e type inequalities. This gives us
\begin{align*}
     \int_{\Sigma_{r(\tau), \infty}}(|x|^2 + \frac{1}{|x|^2})|U^-_{\tau}|^2e^{-\frac{|x|^2}{4}}d\mathcal{H}^{d} &\leq  C(n,m)\|U^-_{\tau}\| ^2_{H^1_w}  + (8r(\tau) - \frac{1}{C(n,m)r(\tau)})\int_{\partial\Sigma_{r(\tau), \infty}}(U^-_{\tau})^2e^{-\frac{|x|^2}{4}}d\mathcal{H}^{d-1}.
\end{align*}
Since $r(\tau) \to 0$ as $\tau \to -\infty$, by possibly shrinking $\overline{\tau}  < 0$, we can drop the boundary integral for all $\tau \leq \overline{\tau}$. Therefore, the final estimate we get is
\begin{equation}\label{est of nonlinear term coming from nonlinearity negative}
    |2\langle \xi_{\tau}Qu_{\tau}, U^-_{\tau} \rangle| \leq C(n,m)\rho_0^{ p(n,m)} \| \Tilde{u}_{\tau}\|^2_{L^2_w} + C(n,m)\rho_0^{ p(n,m)} \| U^-_{\tau} \|^2_{H^1_w}.
\end{equation}
The last term is immediately estimated by \eqref{est fo error term fromcutoff positive mode}, and we obtain
\begin{equation}\label{est of error term from cutoff negatvie}
    |2\langle E_{\tau}, U^-_{\tau} \rangle| \leq 2\| E_{\tau}\|\| U^-_{\tau}\| \leq \frac{-\lambda_5(n,m)}{4}\| U^-_{\tau}\|^2 + C(n,m)\|E_{\tau}\|^2 \leq \frac{-\lambda_5(n,m)}{4}\| U^-_{\tau}\|^2 + C(n,m)\rho_0^{2 + 2p(n,m)}(\tau).
\end{equation}
Therefore, combining \eqref{est of linear term negative mode}, \eqref{est of nonlinear term coming from nonlinearity negative}, \eqref{est of error term from cutoff negatvie}, using coercivity of the operator $-L$ (\eqref{spectral decomposition}), and redefining $p(n,m)$ by $\frac{p(n,m)}{2}$, we obtain
\begin{align*}
    \frac{d}{d\tau}\|U^-_{\tau}\|^2_{L^2_w} &\leq \frac{3\lambda_5(n,m)}{4}\|U^-_{\tau}\|^2_{L^2_w} + \langle LU^-_{\tau} ,U^-_{\tau}\rangle+  C(n,m)\rho_0^{ 2p} \| U^-_{\tau} \|^2_{H^1_w}  + C(n,m)\rho_0^{ 2p} \|\Tilde{u}_{\tau}\|^2 + C(n,m)\rho_0^{2 + 2p} \\ & \leq (\frac{3\lambda_5(n,m)}{4} + C(n,m)\rho_0^{ 2p})\|U^-_{\tau}\|^2_{L^2_w}  + C(n,m)\rho_0^{ 2p}(\|\Tilde{u}_{\tau}\|^2+\rho_0^{2 }). 
\end{align*}
By possibly shrinking $\overline{\tau}<0$, we can ensure that
\begin{equation}
    \frac{3\lambda_5(n,m)}{4}+ C(n,m)\rho_0^{ 2p(n,m)} \leq \frac{1}{2}\lambda_5(n,m).
\end{equation}
This proves the second inequality. 
\end{proof}
We now repeatedly use the evolution inequalities obtained in proposition \ref{evolution of projections of localized functions} together with the Merle-Zaag ODE lemma (lemma A.1 in \cite{merle1998optimal}) to isolate the dynamics onto the first eigenmodes. \\

We define for each $\tau \leq \overline{\tau}$, 
\begin{equation}
    \pi_{<}(\tau) = \sup_{s \leq \tau}\| U^-_{s}\|_{L^2_w}, \ \pi_{\geq 0}(\tau) = \sup_{s \leq \tau}\|U^+_s + U^0_s\|_{L^2_w}.
\end{equation}
Recall from \eqref{compariosn inequality with graphical radius estimate} that for $\tau \leq \overline{\tau}$, $2r(\tau) \leq 
\frac{\overline{r}}{5}$; hence we have
\begin{equation}\label{controlling rho0}
    \rho_0(\tau) \leq \pi_{<}(\tau) + \pi_{\geq }(\tau),
\end{equation}
where $\rho_0$ is given in lemma \ref{weighted C2 estimate for the graph function giving us higher power nonlinearity}. We first establish dominance of the nonnegative eigenmodes. 
\begin{lemma}\label{dominance of nonnegative modes}
    Following the setup in proposition \ref{evolution of projections of localized functions}, we have
    \begin{equation}
        \pi_{<0}(\tau) \leq  C(n,m)\pi_{\geq 0}(\tau)^{1 + p(n,m)}
    \end{equation}
    for all $\tau \leq \overline{\tau}$ after possibly shrinking $\overline{\tau} < 0$. 
\end{lemma}
\begin{proof}[Proof of lemma \ref{dominance of nonnegative modes}]
We split into two cases, either $\pi_{<0}(\tau) = 0$ for some $\tau_0 \leq \overline{\tau}$, or $\pi_{<0}(\tau) > 0$ for all $\tau \leq \overline{\tau}$. \\

In the first case, we redefine $\overline{\tau}$ to be $\tau_0$. Then this implies that $\pi_{<0}(\tau) \equiv 0$ for all $\tau \leq \overline{\tau}$; thus lemma \ref{dominance of nonnegative modes} trivially follows.\\

In the second case, note that
\begin{equation}
    \lim_{s\to -\infty}\|\Tilde{u}_s\|_{L^2_w} = 0
\end{equation}
because of lemma \ref{weighted C2 estimate for the graph function giving us higher power nonlinearity}. This together with positivity of $\pi_{<0}$ implies that 
\begin{equation}
    \pi_{<0}(\tau) = \max_{s \leq \tau}\| U^-_{s}\|_{L^2_w},
\end{equation}
and that for a.e $\tau \leq \overline{\tau}$, $\pi_{<0}$ is differentiable, and
\begin{equation}
    \frac{d}{d\tau} \pi_{<0}(\tau)^2 = \frac{d}{ds}\| U^-_{s}\|_{L^2_w}^2,
\end{equation}
where the $s$-derivative is evaluated at points where $\pi_{<0}$ is attained. Then by using the evolution inequalities in proposition \ref{evolution of projections of localized functions}, \eqref{controlling rho0}, and the fact that $\pi_{<0}(\tau)$ is nondecreasing in $\tau$, we obtain
\begin{equation}
    0 \leq \frac{d}{d\tau}\pi_{<0}^2(\tau) \leq \frac{\lambda_5(n,m)}{2}\pi_{<0}^2(\tau) + C(n,m)(\pi_{<0}^{2 + 2p(n,m)}(\tau)  + \pi_{\geq 0}^{2 + 2p(n,m)}(\tau) ).
\end{equation}
Since $\lambda_5(n,m) < 0$, by possibly shrinking $\overline{\tau} < 0$ so that $\sup_{s \leq \tau}\|\Tilde{u}_{s}\|_{L^2_w}$ is small enough, 
we see that for a.e $\tau \leq \overline{\tau}$, 
\begin{equation}
    \pi_{<0}(\tau) \leq C(n,m)\pi^{1 + p(n,m)}_{\geq 0}(\tau).
\end{equation}
By continuity in $\tau$, the above inequality is true for all $\tau \leq \overline{\tau}$, thus proving lemma \ref{dominance of nonnegative modes}.
\end{proof}
We now further split $\pi_{\geq}(\tau)$ into the neutral and positive modes. Define
\begin{equation}
    \pi_{>0}(\tau) = \sup_{s \leq \tau}\|U^+_{s}\|_{L^2_w}, \ \pi_{= 0}(\tau) = \sup_{s \leq \tau}\|U^0_s\|_{L^2_w}.
\end{equation}
We now use the one-sidedness assumption \ref{additional geoemtric assumption} to rule out the dominance of the zero eigenmode. The idea is essentially that of lemma 5.1 in \cite{Chodosh2024-qv}.
\begin{lemma}\label{dominance of positive mode}
There exists $\epsilon_0(n,m) > 0$ so that for any $\epsilon \in (0, \epsilon_0)$, we can find $\overline{\tau}< 0$ so that for all $\tau \leq \overline{\tau}$, we have
    \begin{equation}
        \pi_{=0}(\tau) \leq  \epsilon \pi_{> 0}(\tau)
    \end{equation}
    As a result, $\pi_{> 0}(\tau)$ is strictly increasing in $\tau$, and we have
    \begin{equation}
        \pi_{> 0}(\tau) = \|U^+_{\tau}\|_{L^2_w}
    \end{equation}
    for all sufficiently negative times.
\end{lemma}  
\begin{remark}
    One immediate consequence of lemma \ref{dominance of positive mode} is that 
    \begin{equation}
        \rho_0(\tau) \leq 2\|U^+_{\tau}\|_{L^2_w}
    \end{equation}
    for all sufficiently negative $\tau$. This fact combined with proposition \ref{evolution of projections of localized functions} yields a cleaner differential inequality 
    \begin{equation}\label{cleaner evolution equation for positive mode}
        \|\frac{d}{d\tau}U^+_{\tau} - LU^+_{\tau}\|_{L^2_w} \leq C(n,m)\|U^+_{\tau}\|_{L^2_w}^{1 + p(n,m)}.
    \end{equation} 
    for all sufficiently negative times.
\end{remark}
\begin{proof}[Proof of lemma \ref{dominance of positive mode}]
    First choose $\overline{\tau} < 0$ so that both proposition \ref{evolution of projections of localized functions} and lemma \ref{dominance of nonnegative modes} hold. \\

    Fix $\epsilon \in (0, \epsilon_0)$, where $\epsilon_0(n,m) > 0$ will be determined later. We claim that either
    \begin{equation}
        \pi_{<0}(\tau) + \pi_{=0}(\tau) \leq  \epsilon \pi_{> 0}(\tau), \text{ or } \pi_{<0}(\tau) +\pi_{>0}(\tau) \leq  \epsilon\pi_{= 0}(\tau)
    \end{equation}
    for all $\tau \leq  \overline{\tau}$ after possibly shrinking $\overline{\tau}$. Note that by lemma \ref{dominance of nonnegative modes}, it is enough to prove the above dichotomy without $ \pi_{<0}(\tau)$, so we focus on
    \begin{equation}
        \pi_{=0}(\tau) \leq  \epsilon \pi_{> 0}(\tau), \text{ or } \pi_{>0}(\tau) \leq  \epsilon\pi_{= 0}(\tau).
    \end{equation}
    First, if either $\pi_{=0}(\tau_0) = 0$ or $\pi_{>0}(\tau_0) = 0$ for some $\tau \leq \overline{\tau}$, then the claim trivially follows after redefining $\overline{\tau} = \tau_0$. Thus, we consider the nontrivial case where both are positive for all $\tau \leq \overline{\tau}$. \\

    In this case, by arguing as in lemma \ref{dominance of nonnegative modes}, we see that both $ \pi_{=0}(\tau)$, and $\pi_{> 0}(\tau)$ are a.e differentiable, and by using proposition \ref{evolution of projections of localized functions}, and lemma \ref{dominance of nonnegative modes}, we have
\begin{align*}
    &\frac{d}{d\tau}\pi^2_{<0} \leq \frac{\lambda_5(n,m)}{2}\pi_{<0}^2 + C(n,m)\rho_0^{p(n,m)}(\pi_{<0}^2 + \pi_{=0}^2 + \pi_{>0}^2), \\ &|\frac{d}{d\tau}\pi_{=0}^2| \leq C(n,m)\rho_0^{p(n,m)}( \pi^2_{<0} + \pi_{=0}^2 +  \pi_{>0}^2),\\&\frac{d}{d\tau}\pi_{>0}^2 \geq (-1-\alpha)\pi_{>0}^2 - C(n,m)\rho_0^{p(n,m)} (\pi_{<0}^2 + \pi_{=0}^2+  \pi_{>0}^2).
\end{align*}
    Since $\lim_{\tau \to -\infty}\rho_0(\tau) = 0$,  by applying the Merle-Zaag ODE lemma (lemma A.1 in \cite{merle1998optimal}) to the triple 
    \begin{equation}
        (X_+(\tau), X_0(\tau), X_-(\tau)) = (\pi_{>0}^2, \pi_{=0}^2, \pi_{<0}^2),
    \end{equation}
     we have the desired claim for small enough $\epsilon \leq \epsilon_0(n,m)$. \\

    We now prove that $\pi_{>0}(\tau) \leq  \epsilon \pi_{= 0}(\tau)$ cannot occur if $\epsilon_0(n,m) > 0$ is small enough. Suppose for a contradiction, that
    \begin{equation}\label{contradiction assumption, dominance of neutral mode}
        \pi_{>0}(\tau) \leq  \epsilon \pi_{= 0}(\tau)
    \end{equation}
    for all $\tau \leq \overline{\tau}$. Let $\mathcal{V}_0$ denote the space of zero eigenfunctions of $L$. First note that when the zero eigenmode is dominant, $\pi_{=0}(\tau) > 0$ and thus 
    \begin{equation}
        \pi_{=0}(\tau) = \max_{s \leq \tau}\|U^0_s\|_{L^2_w} > 0.
    \end{equation}
    Therefore, we can find $\tau_i \to -\infty$ so that
    \begin{equation}
        \pi_{=0}(\tau_i) = \|U^0_{\tau_i}\|_{L^2_w} \to 0 \text{ as }i \to \infty.
    \end{equation}
    Define
    \begin{equation}
        v_i = \frac{U^0_{\tau_i}}{\|U^0_{\tau_i}\|_{L^2_w}} = v_i^+ + v_i^-,
    \end{equation}
    where $v_i^+$, $v_i^-$ are the positive / negative parts of $v_i$ respectively. We now show that $v^-_i$ is small. By assumption \ref{additional geoemtric assumption}, we have
    \begin{equation}
        v_i = \frac{U^0_{\tau_i}}{\|U^0_{\tau_i}\|_{L^2_w}} = \frac{\Tilde{u}_{\tau_i}}{\|U^0_{\tau_i}\|_{L^2_w}} + \frac{U^0_{\tau_i} - \Tilde{u}_{\tau_i}}{\|U^0_{\tau_i}\|_{L^2_w}} \geq \frac{U^0_{\tau_i} - \Tilde{u}_{\tau_i}}{\|U^0_{\tau_i}\|_{L^2_w}} .
    \end{equation}
    This combined with the assumption \eqref{contradiction assumption, dominance of neutral mode}, and lemma \ref{dominance of nonnegative modes} implies that
    \begin{equation}
        \|v_i\|_{L^2_w} = 1, \ \|v^-_i\|_{L^2_w} \leq 2\epsilon
    \end{equation}
    for large enough $i$. On the other hand, by standard elliptic theory together with compactness (note that $\mathcal{V}_0$ is finite dimensional), there exists $C_0(n,m) > 0$ so that
    \begin{equation}
        \|\phi^-\|_{L^2_w} \geq C_0(n,m) \text{ for all }\phi \in \mathcal{V}_0, \|\phi\|_{L^2_w} = 1.
    \end{equation}
    Therefore by possibly choosing $\epsilon_0(n,m) > 0$ smaller so that
    \begin{equation}
        2\epsilon_0 \leq \frac{1}{2}C_0(n,m),
    \end{equation}
    we get a contradiction, thus establishing that
    \begin{equation}\label{dominacne of posiive mode in the proof}
         \pi_{=0}(\tau) \leq  \epsilon \pi_{> 0}(\tau)
    \end{equation}
    for all $\epsilon\in (0, \epsilon_0(n,m))$, $\tau \leq \overline{\tau}$.\\

    Once we have \eqref{dominacne of posiive mode in the proof}, $\pi_{ > 0}(\tau) > 0$, and by using proposition \ref{evolution of projections of localized functions}, we obtain
    \begin{equation}
        \frac{d}{d\tau}\pi_{>0}(\tau) \geq \frac{-1-\alpha}{2}\pi_{>0}(\tau) - \epsilon  \pi_{>0}(\tau)
    \end{equation}
    for all sufficiently negative times for each $\epsilon \in (0, \epsilon_0(n,m))$. By possibly choosing $\epsilon_0(n,m) > 0$ smaller, we obtain
\begin{equation}
        \frac{d}{d\tau}\pi_{>0}(\tau) \geq \frac{-1-\alpha}{4}\pi_{>0}(\tau)  > 0,
    \end{equation}
    thus proving that $\pi_{>0}(\tau)$ is strictly increasing. This immediately implies 
    \begin{equation}
        \pi_{> 0}(\tau) = \|U^+_{\tau}\|_{L^2_w},
    \end{equation}
    thus proving lemma \ref{dominance of positive mode}.
\end{proof}
By repeating the arguments in lemma \ref{dominance of positive mode}, we can rule out dominance of all the positive eigenmodes except the first one by using the one-sidedness assumption \ref{additional geoemtric assumption}.
\begin{lemma}\label{dominance of first eigenmode}
    There exists $\epsilon_0(n,m) > 0$ so that for any $\epsilon \in (0, \epsilon_0(n,m))$, we can find $\overline{\tau} < 0$ so that for all $\tau \leq 
    \overline{\tau}$,
\begin{equation}
    \sup_{s \leq \tau}\|\Tilde{u}_{s}\| \leq (1 + \epsilon)\langle U^+_{\tau}, \varphi_1 \rangle_{L^2_w}.
\end{equation}
If we set $a(\tau) = \langle U^+_{\tau}, \varphi_1 \rangle_{L^2_w} > 0$, then $a(\tau)$ satisfies
\begin{equation}\label{inequality of a(tau)}
    |\frac{d}{d\tau}a(\tau) - \frac{1 - \alpha}{2}a(\tau)| \leq C(n,m)a^{1 + p(n,m)}(\tau)
\end{equation}
for all sufficiently negative times.
\end{lemma}
\begin{proof}[Proof of lemma \ref{dominance of first eigenmode}]
    We first choose $\epsilon_0(n,m) > 0$ so that for any $\epsilon \in (0, \epsilon_0(n,m))$, we can find $\overline{\tau}< 0$ so that proposition \ref{evolution of projections of localized functions}, lemma \ref{dominance of nonnegative modes} and lemma \ref{dominance of positive mode} hold for all $\tau \leq \overline{\tau}$, and $\epsilon_0(n,m) < \frac{\lambda_2 - \lambda_3}{100}$. We further split the positive mode, and consider
    \begin{equation}
        U^{= \lambda_3}_\tau = \langle \Tilde{u}_{\tau}, \varphi_3\rangle_{L^2_w}\varphi_3, \ U^{> \lambda_3}_{\tau} = U^+_{\tau} - U^{= \lambda_3}_{\tau}, \ U^{<\lambda_3}_{\tau} = \Tilde{u}_{\tau} - U^+_{\tau},
    \end{equation}
    where $0 < \lambda_3 < \lambda_2$ are given in \eqref{first four eigenvalues eq}, and $\varphi_3$ is the eigenfunction corresponding to $\lambda_3$. 
    By using proposition \ref{evolution of projections of localized functions} together with lemma \ref{dominance of nonnegative modes}, lemma \ref{dominance of positive mode}, we have for each $\epsilon \in (0, \epsilon_0(n,m))$, $\tau$ small enough,
    \begin{align*}
        \frac{d}{d\tau}e^{-(2\lambda_3 - \epsilon)\tau}\|U^{<\lambda_3}_{\tau}\|^2_{L^2_w} \leq& -2\lambda_3e^{-(2\lambda_3 - \epsilon)\tau}\|U^{<\lambda_3}_{\tau}\|^2_{L^2_w}   \\ &+ 2\epsilon(e^{-(2\lambda_3 - \epsilon)\tau}\|U^{<\lambda_3}_{\tau}\|_{L^2_w}^2 + e^{-(2\lambda_3 - \epsilon)\tau}\|U^{= \lambda_3}_\tau\|_{L^2_w}^2 + e^{-(2\lambda_3 - \epsilon)\tau}\|U^{> \lambda_3}_\tau\|_{L^2_w}^2),\\|\frac{d}{d\tau }e^{-(2\lambda_3 - \epsilon)\tau}\|U^{= \lambda_3}_\tau\|_{L^2_w}^2 | \leq& 2\epsilon(e^{-(2\lambda_3 - \epsilon)\tau}\|U^{<\lambda_3}_{\tau}\|_{L^2_w}^2 + e^{-(2\lambda_3 - \epsilon)\tau}\|U^{= \lambda_3}_\tau\|_{L^2_w}^2 + e^{-(2\lambda_3 - \epsilon)\tau}\|U^{> \lambda_3}_\tau\|_{L^2_w}^2),\\\frac{d}{d\tau }e^{-(2\lambda_3 - \epsilon)\tau}\|U^{> \lambda_3}_\tau\|_{L^2_w}^2 \geq&  (2\lambda_2 - 2\lambda_3)e^{-(2\lambda_3 - \epsilon)\tau}\|U^{> \lambda_3}_\tau\|_{L^2_w}^2 \\ & -2\epsilon(e^{-(2\lambda_3 - \epsilon)\tau}\|U^{<\lambda_3}_{\tau}\|_{L^2_w}^2 + e^{-(2\lambda_3 - \epsilon)\tau}\|U^{= \lambda_3}_\tau\|_{L^2_w}^2 + e^{-(2\lambda_3 - \epsilon)\tau}\|U^{> \lambda_3}_\tau\|_{L^2_w}^2).
    \end{align*}
    By lemma \ref{dominance of nonnegative modes}, lemma \ref{dominance of positive mode}, \eqref{cleaner evolution equation for positive mode} implies that
    \begin{equation}
        \lim_{\tau \to -\infty}e^{-(\lambda_3 - \epsilon)\tau}\|\Tilde{u}_{\tau}\|_{L^2_w} = 0
    \end{equation}
    for any small $\epsilon > 0$. 
    Therefore, the Merle-Zaag ODE lemma (lemma A.1 in \cite{merle1998optimal}) applied to the triple
    \begin{equation}
        (X_+(\tau), X_0(\tau), X_-(\tau)) = (e^{-(2\lambda_3 - \epsilon)\tau}\|U^{> \lambda_3}_\tau\|_{L^2_w}^2 , e^{-(2\lambda_3 - \epsilon)\tau}\|U^{= \lambda_3}_\tau\|_{L^2_w}^2, e^{-(2\lambda_3 - \epsilon)\tau}\|U^{< \lambda_3}_\tau\|_{L^2_w}^2)
    \end{equation}
    implies that after possibly shrinking $\overline{\tau}$, either
    \begin{equation}
        \|U^{< \lambda_3}_\tau\|_{L^2_w} + \|U^{> \lambda_3}_\tau\|_{L^2_w} \leq \epsilon\|U^{= \lambda_3}_\tau\|_{L^2_w}, \text{ or }\|U^{< \lambda_3}_\tau\|_{L^2_w} + \|U^{= \lambda_3}_\tau\|_{L^2_w} \leq \epsilon \|U^{> \lambda_3}_\tau\|_{L^2_w}
    \end{equation}
    for all $\tau \leq \overline{\tau}$ for small enough $\epsilon \leq \epsilon_0(n,m)$. \\

    We argue that the first case cannot occur when $\epsilon_0(n,m) > 0$ is small enough. If not, then we can define for any $\tau_i \to -\infty$, 
    \begin{equation}
        v_i = \frac{U^{= \lambda_3}_{\tau_i}}{\| U^{= \lambda_3}_{\tau_i}\|_{L^2_w}} \in \mathcal{V}_{\lambda_3},
    \end{equation}
    where $\mathcal{V}_{\lambda_3}$ denotes the space of all $\lambda_3$-eigenfunctions of $L$. By exploiting assumption \ref{additional geoemtric assumption} as in the proof of lemma \ref{dominance of positive mode}, we see that for $i$ large enough, 
    \begin{equation}
        \|v_{i}\|_{L^2_w} = 1, \ \|v^-_{i}\|_{L^2_w} \leq 10\epsilon.
    \end{equation}
    On the other hand, by standard elliptic theory together with compactness, we get a contradiction if $\epsilon_0(n,m) >0$ is small enough. We thus get  
    \begin{equation}
        \|U^{= \lambda_3}_\tau\|_{L^2_w} \leq \epsilon \|U^{> \lambda_3}_\tau\|_{L^2_w}.
    \end{equation}
    By repeating the previous argument with $\lambda_3$ replaced by $\lambda_2$, we eventually obtain dominance of the first eigenmode. Since $\Tilde{u}_{\tau} \geq 0$, and the first eigenspace is one dimensional, we have
    \begin{equation}
    \sup_{s \leq \tau}\|\Tilde{u}_{s}\| \leq (1 + \epsilon)\langle U^+_{\tau}, \varphi_1 \rangle_{L^2_w},
\end{equation}
proving the first assertion. \\

To prove the desired inequality, first take $\epsilon = \epsilon_0(n,m)$. Combining this with \eqref{cleaner evolution equation for positive mode} implies \eqref{inequality of a(tau)}, thus proving lemma \ref{dominance of first eigenmode}.
\end{proof}
The discussions so far implies that the dynamics of $(\Tilde{M}_{\tau})_{\tau \in (-\infty, 0)}$ which satisfies assumptions \ref{strong convergence to simonscone assumption1}, \ref{additional geoemtric assumption} is determined by 
\begin{equation}
    a(\tau) = \langle \Tilde{u}_{\tau} , \varphi_1\rangle_{L^2_w},
\end{equation}
and that $a(\tau)$ satisfies
\begin{equation}
    |\frac{d}{d\tau}a(\tau) - \frac{1 - \alpha}{2}a(\tau)| \leq C(n,m)a^{1 + p(n,m)}(\tau).
\end{equation}
This allows us to compute the `leading term' of $\Tilde{u}_{\tau}$ for all sufficiently negative time. 
\begin{lemma}\label{uniqueasymptoticsinweightedL2setting}
There exist constants $c > 0$, $p(n,m) > 0$, $\overline{\tau} < 0$ so that for all $\tau \leq \overline{\tau}$, 
\begin{equation}
    \| \Tilde{u}_{\tau} - ce^{\frac{1 - \alpha}{2}\tau}\varphi_1 \|_{L^2_w} \leq O(e^{\frac{(1 - \alpha)(1 + p(n,m))}{2}\tau}).
\end{equation}
\end{lemma}
\begin{proof}[Proof of lemma \ref{uniqueasymptoticsinweightedL2setting}]
    Fix $\overline{\tau} < 0$ so that all previous results hold for every $\tau \leq \overline{\tau}$. Inequality \eqref{inequality of a(tau)} and the fact that $a(\tau) \to 0$ as $\tau \to -\infty$ imply that there exists $c \in \mathbf{R}$ so that
    \begin{equation}
        |a(\tau) - ce^{\frac{1 - \alpha}{2}\tau}|\leq O(e^{\frac{(1 - \alpha)(1 + p(n,m))}{2}\tau}).
    \end{equation}
   By using lemma \ref{dominance of first eigenmode}, and proposition \ref{evolution of projections of localized functions}, one can check that 
   \begin{equation}
       \| \Tilde{u}_{\tau} - a(\tau)\varphi_1 \|_{L^2_w} \leq O(e^{\frac{(1 - \alpha)(1 + p(n,m))}{2}\tau}).
   \end{equation}
   Combining the two inequalities, we see that
   \begin{equation}
       \| \Tilde{u}_{\tau} - ce^{\frac{1 - \alpha}{2}\tau}\varphi_1 \|_{L^2_w} \leq O(e^{\frac{(1 - \alpha)(1 + p(n,m))}{2}\tau})
   \end{equation}
   for some $c \in \mathbf{R}$. \\

   We now claim that $c > 0$. The one-sidedness assumption \ref{additional geoemtric assumption} immediately implies that $c \geq 0$. Assume for a contradiction that $c = 0$. Then the inequality \eqref{inequality of a(tau)} implies that 
   \begin{equation}
       a(\tau) \equiv 0
   \end{equation}
   for all $\tau \leq \overline{\tau}$. This implies that $\Tilde{u}_{\tau} \equiv 0$, which means that $\Tilde{M}_{\tau}$ has to agree with $\Sigma$ for all sufficiently negative time due to the strong maximum principle (note that $\Tilde{M}_{\tau}$ is one-sided). This contradicts the smoothness of $\Tilde{M}_{\tau}$, thus excluding the $c = 0$ case.
\end{proof}
Lemma \ref{uniqueasymptoticsinweightedL2setting} combined with standard parabolic interior regularity theory implies the following main result of this section.
\begin{theorem}\label{unique asymptotics of one sided flow}[Theorem \ref{maintheorem : uniqueasmptotics}] Let $(M_t)_{t \in (-\infty, 0)}$ be a smooth, properly embedded ancient mean curvature flow which satisfies assumptions \ref{strong convergence to simonscone assumption1}, \ref{additional geoemtric assumption}, and let $(\Tilde{M}_{\tau})_{\tau \in (-\infty, 0)}$ be the rescaled mean curvature flow. Then there exist $c > 0$, $p(n,m) > 0$ so that for every $k \in \mathbf{N}$, $0 < r < R < \infty$, there exists $\overline{\tau}$ so that for all $\tau \leq \overline{\tau}$, 
\begin{equation}
    \|u_{\tau} - ce^{\frac{1 - \alpha}{2}\tau}\varphi_1 \|_{C^{k}(\Sigma_{r, R})} \leq O(e^{\frac{(1 - \alpha)(1 + p(n,m))}{2}\tau}).
\end{equation}
\end{theorem}
Theorem \ref{unique asymptotics of one sided flow} has two consequences for the unrescaled flow $(M_t)_{t \in (-\infty, 0)}$. The first corollary is the `uniform in time' asymptotic behavior of the time slices of the \textit{unrescaled }flow in space. 
\begin{corollary}\label{graphicality of unrescaled flow}
    Let $\mathcal{M} = (M_t)_{t \in (-\infty, 0)}$ be a smooth, properly embedded mean curvature flow which satisfies assumptions \ref{strong convergence to simonscone assumption1}, \ref{additional geoemtric assumption}. Then there exists $\overline{\eta}(n,m) \in (0, \frac{1}{100})$ so that for any $\eta \in (0, \overline{\eta}(n,m))$, there exist $\overline{t} < 0$, and $r = r(n,\mathcal{M},  \eta) > 0$ so that for all $t \leq \overline{t}$, $M_t \setminus B(0, r)$ is a normal graph of $v_t \in C^{\infty}(\Sigma_{2r, \infty} \subset \mathcal{D})$ with $\|v_t\|^*_{C^3(\Sigma_{2r, \infty})} \leq \eta$.
\end{corollary}
\begin{proof}[Proof of corollary \ref{graphicality of unrescaled flow}]
    Let $\overline{r}(n,m) > 0$, $\overline{\eta} \in (0, \frac{1}{100})$ be constants from proposition \ref{graphical radius estimate proposition}. Theorem \ref{unique asymptotics of one sided flow} implies that for all $\tau$ small enough, 
    \begin{equation}
        \sup_{s \leq \tau}\|u_{s}\|_{L^2_{w}(\Sigma_{\frac{\overline{r}}{2}, 2\overline{r}})} \leq Ce^{\frac{1 - \alpha}{2}\tau}
    \end{equation}
    for some fixed constant $C > 0$. Therefore, proposition \ref{graphical radius estimate proposition} implies that for any $\eta \in (0, \overline{\eta}(n,m))$, we can find $\overline{\tau} < 0$ so that for all $\tau \leq \overline{\tau}$, 
    \begin{equation}
        r(\tau ; \eta) \leq Ce^{\frac{\tau}{2}}.
    \end{equation}
    Rescaling back to $(M_t)_{t \in (-\infty, 0)}$ implies corollary \ref{graphicality of unrescaled flow} since the weighted norm is scaling invariant. 
\end{proof}
The second corollary is that the unrescaled flow is `trapped' between two nearby leaves of the Hardt-Simon foliation in the parabolic region for all sufficiently negative time. 
\begin{corollary}\label{upper and lower obound of unrescaled flow in parabolic region}
    Let $(M_t)_{t \in (-\infty, 0)}$ be a smooth, properly embedded mean curvature flow which satisfies assumptions \ref{strong convergence to simonscone assumption1}, \ref{additional geoemtric assumption}. Then there exists $c > 0$ so that for any small $\epsilon > 0$, $0 < r < R < \infty$, there exists $\overline{t} < 0$ so that for any $t \leq \overline{t}$, $M_t$ contains the normal graph of $v_t$ defined over $\Sigma_{r\sqrt{-t}, R\sqrt{-t}}$, and for each $x \in \Sigma_{r\sqrt{-t}, R\sqrt{-t}}$, 
    \begin{equation}
        \hat{\psi}_{c - \epsilon}(x) \leq v_t(x) \leq \hat{\psi}_{c + \epsilon}(x).
    \end{equation}
    Here, $\hat{\psi}_{c \pm \epsilon}$ are given by \eqref{graphical normal parametrization of hardtsimonleaf}.
\end{corollary}
\begin{proof}[Proof of corollary \ref{upper and lower obound of unrescaled flow in parabolic region}]
    By rescaling theorem \ref{unique asymptotics of one sided flow} back to the unnormalized flow, we see that we can find $\lambda > 0$ so that for any $\epsilon > 0$, there exists $\overline{t}<0$ so that if $t \leq \overline{t}$, $x \in \Sigma_{r\sqrt{-t}, R\sqrt{-t}}$,
    \begin{equation}
       (\lambda - \epsilon)|x|^{\alpha} \leq v_t(x) \leq (\lambda + \epsilon)|x|^{\alpha}
    \end{equation}
    Then corollary \ref{upper and lower obound of unrescaled flow in parabolic region} follows from the asymptotic formula of $\hat{\psi}_c$ given by \eqref{graphical normal parametrization of hardtsimonleaf}.
\end{proof}
\section{Uniqueness of mean convex flows}\label{uniquenesssection}
In this section, we consider a smooth, properly embedded ancient mean curvature flow $(M_t)_{t \in (-\infty, 0)}$ which satisfies assumptions \ref{strong convergence to simonscone assumption1}, \ref{additional geoemtric assumption}, \ref{mean convexity assumption4}. The goal of this section is to prove that $(M_t)_{t \in (-\infty, 0)}$ is stationary, hence has to agree with a leaf in the Hardt-Simon foliation (theorem \ref{uniqueness of mean convex flow}). A rough sketch of the proof is as follows. We first use the assumptions, and corollary \ref{graphicality of unrescaled flow} to obtain (i) convergence of $(M_t)_{t \in (-\infty, 0)}$ to some $\Sigma^+_c$ as $t\to -\infty$ after possible reflection across $\Sigma$, and (ii) uniform in time decay rate of the normal graph function in space (lemma \ref{backwards convergence to hardtsimonfoliation}, corollary \ref{global graphicality of flow over hardtsimon leaf}). Then by using the Liouville type theorems (proposition \ref{liouvile for simons cone}, \ref{liouvlie for hardtsimonleaf}), we show that the flow is stationary, hence prove that $(M_t)_{(-\infty, 0)}$ has to be stationary (lemma \ref{Schauder type estimate}, lemma \ref{flow is stationary}). \\

Throughout this section, for each $x = (\sqrt{\frac{n-1}{n+m-2}}rw_1, \sqrt{\frac{m-1}{n+m-2}}rw_2) \in \Sigma$ for some $r > 0$, $w_1 \in \mathbf{S}^{n-1}, w_2 \in \mathbf{S}^{m-1}$, we always define the unit normal $\nu_{\Sigma}(x) = (-\sqrt{\frac{m-1}{n+m-2}}w_1, \sqrt{\frac{n-1}{n+m-2}}w_2)$. Also, note that by the one-sidedness assumption \ref{additional geoemtric assumption}, the normal graph function over $\Sigma$ denoted by $v_{t}$ is nonnegative.\\
 
We first prove backward convergence to $\Sigma^+_c$ for some $c > 0$. 
\begin{lemma}\label{backwards convergence to hardtsimonfoliation}
Let $(M_t)_{t \in (-\infty, 0)}$ be a smooth, properly embedded mean curvature flow which satisfies assumptions \ref{strong convergence to simonscone assumption1}, \ref{additional geoemtric assumption}, \ref{mean convexity assumption4}. Then there exists $c > 0$ so that
\begin{equation}
    M_{t} \to \Sigma^+_c\text{ in }C^{\infty}_{loc}(\mathbf{R}^{n+m}).
    \end{equation}
\end{lemma}
\begin{proof}[Proof of lemma \ref{backwards convergence to hardtsimonfoliation}]
    Take any $t_i \to -\infty$, and consider a sequence of mean curvature flows
    \begin{equation}
        \mathcal{M}^i = (M^i_t)_{t \in (-\infty, 1]} = (M_{t_i + t})_{t \in (-\infty, 1]}.
    \end{equation}
By the uniform entropy bound given by \eqref{entropy bound proposition eq}, after passing to a subsequence, there exists a unit regular, integral $n+m-1$-Brakke flow $\mathcal{M}^{\infty} = (\mu^{\infty}_t)_{t \in (-\infty, 1]}$ with
\begin{equation}
    \mathcal{M}^i \to \mathcal{M}^{\infty} \text{ as integral Brakke flows as }i \to \infty.
\end{equation}
We claim that there exists $c > 0$ so that for each $t \leq 1$, 
\begin{equation}
    \mu^{\infty}_t \equiv \mathcal{H}^{n+m-1}|_{\Sigma^+_c}.
\end{equation}
We first show that the limit flow is quasistatic i.e for a.e $t \in (-\infty, 1]$, $\mu^{\infty}_t$ is a pushforward of a stationary integral $n+m-1$-varifold. For each $t \in (-\infty, 0)$, define $\Omega_t$ so that
\begin{equation}
    0 \in \Omega_t, \ \partial \Omega_t = M_t.
\end{equation}
As explained in the introduction, one can always `continuously' select $\Omega_t$ under assumptions \ref{strong convergence to simonscone assumption1}, \ref{additional geoemtric assumption}. Note that the mean convexity assumption \ref{mean convexity assumption4} implies that $\{\Omega_t\}_{t \in (-\infty, 0)}$ are nested, i.e
\begin{equation}
    t_1 > t_2 \implies \Omega_{t_1}\subset \Omega_{t_2.}
\end{equation}
Take any $\phi \in C^1_c(\mathbf{R}^{n+m} ; \mathbf{R}_+)$. Due to mean convexity, we have for any $a < b \leq 1$
\begin{align*}
    \lim_{i \to \infty}\int_{a}^b\int_{M^i_t} \phi |H|d\mathcal{H}^{n+m-1}dt &= \lim_{i \to + \infty}|\int_{a}^{b}\int_{M_{t_i + t}}\phi Hd\mathcal{H}^{n+m-1}dt| \\ & = \lim_{i \to +\infty}|\int_{\Omega_{t_i + b}}\phi d\mathcal{H}^{n+m} - \int_{\Omega_{t_i + a}}\phi 
    d\mathcal{H}^{n+m}| \\ & = \int_{\Omega}\phi d\mathcal{H}^{n+m} - \int_{\Omega}\phi d\mathcal{H}^{n+m} = 0,
\end{align*}
where \begin{equation}
    \Omega = \bigcup_{t < 0}\Omega_t.
\end{equation}
This implies that for almost every $t \in (-\infty, 1]$ at which $\mu^{\infty}_t$ is a pushforward of an integral varifold $V^{\infty}(t)$ and that $M^i_t \to V^{\infty}(t)$ as integral varifolds (after possibly passing to a further subsequence), $H^{\infty}_t \equiv 0$, thus $V^{\infty}(t)$ has to be a stationary integral varifold. \\

We now prove that each $V^{\infty}(t)$ is a multiplicity one $\Sigma^+_c$ for some fixed $c > 0$ \textit{independent} of $t$. To see this, we go back to the original flow $M_t$. By corollary \ref{graphicality of unrescaled flow}, we can find $\overline{\eta}(n,m) \in (0, \frac{1}{100})$ so that for every $\eta \in (0, \overline{\eta})$, there exist $R(n, \eta) > 0$, $\overline{t} < 0$ so that $M_t \setminus B(0, R)$ is a normal graph over $\Sigma_{2R, \infty} \subset \mathcal{D}_t \subset \Sigma$ for all $t \leq \overline{t}$. Moreover, if we denote the normal graph function by $v_t$, then we have
\begin{equation}\label{asymptoticbehaviorofnegativetmes}
    \|v_t\|^*_{C^3(\Sigma_{2R, \infty})} \leq \eta.
\end{equation}
If we denote $R_0$ to be $R$ from the previous discussion with $\eta = \overline{\eta}(n,m)$, then the one-sidedness assumption \ref{additional geoemtric assumption} and the mean convexity assumption \ref{mean convexity assumption4}, together with the standard interior regularity theorem \cite{Ecker1991InteriorEF} imply that there exists a smooth, positive function $v_{-\infty} \in C^{\infty}(\Sigma_{2R_0, \infty})$ so that
\begin{equation}
    v_t \nearrow v_{-\infty} \text{ in }C^{\infty}_{loc}(\Sigma_{2R_0, \infty})\text{ as }t \to -\infty.
\end{equation}
The smooth convergence together with uniform in $t$ estimate \eqref{asymptoticbehaviorofnegativetmes} implies that 
\begin{equation}
    \lim_{R \to \infty}\|v_{-\infty}\|^*_{C^3(\Sigma_{R, \infty})} = 0.
\end{equation}
Therefore, we can conclude that for a.e $t \in (-\infty, 1]$, $\mu^{\infty}_t$ is a pushforward of a stationary integral varifold $V^{\infty}(t)$ which is smooth outside a fixed ball, and agrees with the normal graph of $v_{-\infty}$ outside a fixed ball. The asymptotic behavior of $v_{-\infty}$ implies that $V^{\infty}(t)$ is asymptotic to $\Sigma$; hence by \cite{simon1986minimal} (alternatively, one can apply the proof of corollary 3.7 in \cite{edelen2023regularity}), it has to be either the cone itself, or one of the leaves of the Hardt-Simon foliation with unit multiplicity, up to translation and rotation. Because $v_{-\infty}$ is positive, and is independent of $t$, we conclude that 
\begin{equation}
    V^{\infty}(t) \equiv \Sigma^+_c
\end{equation}
for some fixed $c > 0$ for a.e $t \in (-\infty, 1]$. Thus we obtain the desired claim
\begin{equation}
    \mu^{\infty}_t \equiv \mathcal{H}^{n+m-1}|_{\Sigma^+_c}.
\end{equation}
Then White's local regularity theory \cite{White2005-nl} implies that the convergence $\mathcal{M}^i \to \mathcal{M}^{\infty} = (\Sigma^+_c)_{t \in (-\infty, 1]}$ is smooth. \\

Finally, note that $c > 0$ is independent of choice of $t_i \to -\infty$ since $c > 0$ is determined by the function $v_{-\infty}$ which was a full limit of $v_t$ as $t\to -\infty$. This completes the proof of lemma \ref{backwards convergence to hardtsimonfoliation}.
\end{proof}
For the sake of convenience, we will from now on without loss of generality assume that $c = 1$. If $c \neq 1$, we can rescale the flow by a fixed factor $\frac{1}{c} > 0$. \\

Combining lemma \ref{backwards convergence to hardtsimonfoliation}, assumptions \ref{additional geoemtric assumption} and \ref{mean convexity assumption4}, and corollary \ref{graphicality of unrescaled flow}, we obtain the following nice description of $M_t$ as a \textit{global} normal graph over $\Sigma^+_1$ where the normal graph function denoted by $h_t$ has a nice decay at spatial infinity.
\begin{corollary}\label{global graphicality of flow over hardtsimon leaf}
    For $x = (rw_1, \psi_1(r)w_2) \in \Sigma^+_1$ for some $r \geq 0$, $w_1 \in\mathbf{S}^{n-1}, w_2 \in \mathbf{S}^{m-1}$, define the unit normal vector to be $\nu_{\Sigma^+_1}(x) = \frac{1}{\sqrt{1 + (\psi_1')^2}}(\psi_1'w_1, -w_2)$. Then there exists $\overline{t} < 0$ so that we can find $h \in C^{\infty}(\Sigma^+_1 \times (-\infty, \overline{t}])$ so that 
    \begin{equation}
        M_t = \{x + h_t(x)\nu_{\Sigma^+_1}(x) \ | \ x \in \Sigma^+_1\}, \ h_t \to 0 \text{ as }t \to -\infty \text{ in }C^{\infty}_{loc}(\Sigma^+_1),
    \end{equation}
    and
    \begin{equation}\label{properties of global graph function}
        h_t(x) \geq 0, \ \partial_th_t \geq 0, \ \sup_{t \leq \overline{t}}\sup_{x \in \Sigma^+_1}|\nabla^kh_t(x)|(1 + |x|)^{|\alpha| + k} < \infty \text{ for all } k \geq 0.
    \end{equation}
\end{corollary}
\begin{proof}[Proof of corollary \ref{global graphicality of flow over hardtsimon leaf}]
    Since the convergence $M_t \to \Sigma^+_1$ as $t \to -\infty$ is locally smooth in $\mathbf{R}^{n+m}$, we can write $M_t$ as a normal graph over $\Sigma^+_1$ inside any given large ball for sufficiently negative time. At the same time, corollary \ref{graphicality of unrescaled flow} together with the asymptotic behavior of $\Sigma^+_1$ given by \eqref{asymptotic behavior of hardsimonleafformula} implies that we also have graphicality over $\Sigma^+_1$ outside some fixed compact set for all sufficiently negative time. Thus combining the two gives us existence of a globally defined $h_t$. \\

    We now establish \eqref{properties of global graph function}. By assumption \ref{mean convexity assumption4}, and our choice of normal vector $\nu_{\Sigma^+_1}$, we have $\partial_th_t \geq 0$. Integrating in $t$ implies $h_t \geq 0$. As for the final estimate, note that it is enough to consider $|x|\geq R$ for some $R \geq 100$ to be determined. We first prove \eqref{properties of global graph function} when $k = 0$. The one-sidedness assumption \ref{additional geoemtric assumption} and mean convexity assumption \ref{mean convexity assumption4} imply that for all $t \in(-\infty, 0)$, $M_t$ is `trapped' between $\Sigma$, and $\Sigma^+_1$. Since we can find some $R_0(n,m) > 100$ so that $\Sigma \setminus B(0, R_0)$ can be parametrized as a normal graph over
    \begin{equation}
        \Sigma^+_1 \setminus B(0, 2R_0) \subset \mathcal{D}\subset \Sigma^+_1,
    \end{equation}
    for all $x \in  \Sigma^+_1 \setminus B(0, 2R_0)$, $t \leq \overline{t}$, we have
    \begin{equation}
        0 \leq h_t(x) \leq w(x).
    \end{equation}
    In view of the asymptotic behavior of $\Sigma^+_1$ given by \eqref{asymptotic behavior of hardsimonleafformula}, we see that
    \begin{equation}
        \sup_{x\in \Sigma^+_1 \setminus B(0, 2R_0)}|w(x)|(1 + |x|)^{|\alpha|} < \infty.
    \end{equation}
    This immediately implies \eqref{properties of global graph function} when $k = 0$. \\

    To prove \eqref{properties of global graph function} for other values of $k \in \mathbf{N}$, we use standard parabolic interior regularity theory, and scaling invariance. For each $x_0 \in \Sigma^+_1$ with $|x_0| \geq R_0$, $t_0 \leq \overline{t}$, we can define the function
    \begin{equation}
        \Sigma^+_{\frac{1}{|x_0|}} \cap (B(0, 2) \setminus B(0, \frac{1}{2}))\times [-1,0] \ni (y,t) \to h^{(x_0, t_0)}_t(y) = \frac{1}{|x_0|}h_{t_0 + |x_0|^2t}(|x_0|y).
    \end{equation}
    By the scaling invariance of mean curvature flow, one can see that the normal graph of $h^{(x_0, t_0)}$ over $\Sigma^+_{\frac{1}{|x_0|}} \cap (B(0, 2) \setminus B(0, \frac{1}{2}))$ is a smooth solution to the mean curvature flow. By applying corollary \ref{graphicality of unrescaled flow} together with the asymptotic behavior of the Hardt-Simon leaf, we see that after possibly enlarging $R_0 > 100$, and shrinking $\overline{t}< 0$, we may further assume that the $C^3$ norm of $h^{(x_0, t_0)}$ over $\Sigma^+_{\frac{1}{|x_0|}} \cap (B(0, 2) \setminus B(0, \frac{1}{2})) \times [-1,0] $ is sufficiently small. Then by interpreting $h^{(x_0, t_0)}$ as a smooth solution to a linear, homogeneous second order parabolic PDE with regular coefficients (see for example appendix A in \cite{huang2025mean}), we can apply interior Schauder theory, and bootstrapping, and obtain for each $k$,
   \begin{align*}
       |x_0|^{k - 1}|\nabla^k h_{t_0}(x_0)| &\leq C(n,m,k)\sup_{t \leq \overline{t}}\sup_{y \in \Sigma^+_{1} \cap (B(0, 2|x_0|) \setminus B(0, \frac{1}{2}|x_0|))}\frac{1}{|y|}|h(y, t)| \\ & \leq C(n,m,k)(1 + |x_0|)^{-|\alpha| - 1}.
   \end{align*}
   Note that the constant $C(n,m,k)$ does not depend on $x_0$ since for $|x_0|$ sufficiently large, the geometry of $\Sigma^+_{\frac{1}{|x_0|}} \cap (B(0,2) \setminus B(0,1))$ is uniformly close to that of $\Sigma \cap (B(0,2) \setminus B(0,1))$. 
   Since the above estimate is true for all $|x_0| \geq R_0$, $t_0 \leq \overline{t}$, we obtain \eqref{properties of global graph function} after taking the supremum in $(x_0, t_0)$. 
\end{proof}
We now consider the smooth function
\begin{equation}
    g_t = \partial_th_t \in C^{\infty}(\Sigma^+_{1} \times (-\infty, \overline{t}]).
\end{equation}
In view of corollary \ref{global graphicality of flow over hardtsimon leaf}, we see that
\begin{equation}\label{properties of gt}
    g_t \geq 0, \ \sup_{t \leq \overline{t}}\sup_{x \in \Sigma^+_{1}}|g_t(x)|(1 + |x|)^{|\alpha| + 2} < \infty.
\end{equation}
In the remainder of this section, we show that $g_t \equiv 0$. By essentially repeating the arguments in the proof of lemma \ref{the nonlinear error term computation}, we can obtain an evolution equation of $h_t$ for all sufficiently negative $t$. Then by differentiating in $t$, we see that $g_t$ satisfies a linear, homogeneous second order parabolic equation of the form
\begin{equation}\label{evolution equaiton for partialtht = gt}
    \partial_tg_t = L_{h_t}g_t = \Delta_{\Sigma^+_{1}}g_t + |A_{\Sigma^+_{1}}|^2g_t + E_{h_t}(g_t).
\end{equation}
Here, 
\begin{align*}
    E_{h}(g) &= a^{ij}(x, h, \nabla h, \nabla^2h)\nabla^2_{ij}g + b^i(x, h, \nabla h, \nabla^2h)\nabla_ig + c(x, h, \nabla h, \nabla^2h)g \\ &= a^{ij}(x; h)\nabla^2_{ij}g + b^i(x;h)\nabla_ig + c(x ;h)g ,
\end{align*}
where the coefficients depend smoothly on their arguments, and
\begin{equation}\label{coefficient control}
    |a^{ij}(x;h)| \leq C(n,m)\eta, \ |  b^i(x;h)| \leq \frac{C(n,m)\eta}{1 + |x|}, \ |c(x;h)| \leq \frac{C(n,m)\eta}{(1 + |x|)^2}
\end{equation}
provided 
\begin{equation}\label{weightedc2normboundfor graph functin over simons cone}
    \sup_{x \in \Sigma^+_1}(\frac{|h(x)|}{1 + |x|} + |\nabla h(x)| + (1 + |x|)|\nabla^2h(x)|) \leq \eta \leq \overline{\eta}(n,m)
\end{equation}
for some small $\overline{\eta}(n,m) > 0$. \\

The intuition to proving $g_t \equiv 0$ is as follows; for sufficiently negative time, $h_t$ is very small; hence we can expect that the error term in \eqref{evolution equaiton for partialtht = gt} will be negligible. So if we \textit{pretend} that the error term is identically zero, then due to \eqref{properties of gt}, we can apply the Liouville type theorem (proposition \ref{liouvlie for hardtsimonleaf}), and conclude that $g_t \equiv 0$. \\

To make this precise, let us define a weighted norm. 
\begin{definition}\label{finalweightednorm}
    Let $f \in C^{\infty}(\Sigma^+_1 \times (-\infty, \overline{t}]) $ for some $\overline{t} < 0$. For each $k = 0,1,2,...$, $\gamma > 0$, and $\hat{t} \leq \overline{t}$, define
    \begin{equation}
        \|f\|_{k ; \gamma,\hat{t}} = \sup_{t \leq \hat{t}}\sup_{0 \leq i \leq k}\sup_{x \in \Sigma^+_1}|\nabla ^if_t(x)|(1 + |x|)^{\gamma + i}
    \end{equation}
    whenever it is finite. 
\end{definition}
\begin{lemma}\label{Schauder type estimate}
    Let $h \in C^{\infty}(\Sigma^+_1 \times (-\infty, \overline{t}])$ be the normal graph function from corollary \ref{global graphicality of flow over hardtsimon leaf}. Let $g\in C^{\infty}(\Sigma^+_1 \times (-\infty, \overline{t}])$ so that
    \begin{equation}\label{assumptions on the function g}
        g \geq 0, \ \sup_{t \leq \overline{t}}\sup_{x \in \Sigma^+_{1}}|g_t(x)|(1 + |x|)^{|\alpha| + 2} < \infty,
    \end{equation}
    and satisfies the equation
    \begin{equation}
        \partial_tg_t = L_{h_t}g_t.
    \end{equation}
    Then after possibly shrinking $\overline{t} < 0$, there exists $C_0(n,m, \overline{t}) > 0$ so that
    \begin{equation}\label{key estimate}
        \|g\|_{2 ; |\alpha| + 2,\hat{t}} \leq C_0\| E_{h}g\|_{0 ; |\alpha| + 4,\hat{t}},
    \end{equation}
    for any $\hat{t} \leq \overline{t}$.
    Both norms involved in \eqref{key estimate} are finite. Here, $\alpha \in (-2, -1)$ is given by \eqref{the formula for alpha}.
\end{lemma}
\begin{proof}[Proof of lemma \ref{Schauder type estimate}]
We first establish finiteness of the norms involved. Recalling corollary \ref{global graphicality of flow over hardtsimon leaf}, we can choose $\overline{t} < 0$ so that for all $t \leq \overline{t}$, 
\begin{equation}\label{good behavior of ht at all negative times}
    \sum_{i = 0}^{20}\sup_{x \in \Sigma^+_1, t \leq \overline{t}}(1 + |x|)^{-1 + i}|\nabla^ih_t(x)| \leq \overline{\eta}
\end{equation}
for any given $\overline{\eta} > 0$. In view of the structure of the error term given by \eqref{weightedc2normboundfor graph functin over simons cone}, we can select $\overline{\eta}(n,m) > 0$, hence $\overline{t} < 0$ so that the operator
\begin{equation}
    \partial_t - L_{h_t}
\end{equation}
is uniformly parabolic. We can thus apply interior Schauder theory to obtain $C^k$ estimates of $g_t$ in any fixed compact set, say $B(0, 100)$. \\

To deal with the outer region where $|x| \geq 100$, we repeat the rescaling argument that we used in the proof of corollary \ref{global graphicality of flow over hardtsimon leaf}. Take any $|x_0| \geq 100$, $t_0 \leq \overline{t}$, and consider
\begin{equation}
    g^{(x_0, t_0)}_t(x) = g_{t_0 + |x_0|^2t}(|x_0|x)
\end{equation}
with $(x,t) \in \Sigma^+_{\frac{1}{|x_0|}} \cap (B(0, 2) \setminus B(0, \frac{1}{2})) \times [-1, 0]$. By scaling invariance of mean curvature flow, we see that $g^{(x_0, t_0)}$ solves the linear equation
\begin{equation}
    (\partial_t - L_{h^{(x_0, t_0)}_t})g^{(x_0, t_0)} = 0,
\end{equation}
where
\begin{equation}
    h^{(x_0, t_0)}_t(x) = \frac{1}{|x_0|}h_{t_0 + |x_0|^2t}(|x_0| x).
\end{equation}
The estimate \eqref{good behavior of ht at all negative times} implies the analogous smallness of $h^{(x_0, t_0)}$ in  $(x,t) \in \Sigma^+_{\frac{1}{|x_0|}} \cap (B(0, 2) \setminus B(0, \frac{1}{2})) \times [-1, 0]$; hence the above differential equation is also uniformly parabolic provided $\overline{\eta}(n,m) > 0$ is small enough. By applying interior Schauder theory to $g^{(x_0, t_0)}$ and rescaling back, we can cover the outer region. Ultimately, we obtain the estimate
\begin{equation}\label{regularity estimate for general g}
    \| g\|_{k ; |\alpha| + 2,\hat{t}} \leq C(n,m,k)\|g\|_{0;|\alpha| + 2,\hat{t}} \text{ for }0 \leq k \leq 10, \ \hat{t} \leq \overline{t}.
\end{equation}
Therefore, if $g$ satisfies \eqref{assumptions on the function g}, then 
\begin{equation}
    \|g\|_{2;|\alpha| + 2,\hat{t}} < \infty
\end{equation}
for all $\hat{t} \leq \overline{t}$. Also due to the structure of the error term $E_h$ given by \eqref{weightedc2normboundfor graph functin over simons cone}, by the choice of $\overline{t}$ given by \eqref{good behavior of ht at all negative times}, we see that 
\begin{equation}
    \|E_{h}(g)\|_{0, |\alpha| + 4,\hat{t}} \leq C(n,m)\overline{\eta}\|g\|_{2 ; |\alpha| + 2,\hat{t}} < \infty
\end{equation}
for all $\hat{t} \leq \overline{t}$.
Thus all the norms in estimate \eqref{key estimate} are finite provided $g$ satisfies \eqref{assumptions on the function g}.\\

We now derive \eqref{key estimate}. We will use the same choice of $\overline{t} < 0$ given by \eqref{good behavior of ht at all negative times}. We prove by contradiction. If \eqref{key estimate} is false, then there exists a sequence of functions $g^i \in C^{\infty}(\Sigma^+_1 \times (-\infty, \overline{t}])$, $\hat{t}_i \leq \overline{t}$ so that
\begin{equation}\label{basic properties of original sequence}
        g^i \geq 0, \ \sup_{t \leq \overline{t}}\sup_{x \in \Sigma^+_{1}}|g^i_t(x)|(1 + |x|)^{|\alpha| + 2} < \infty,
    \end{equation}
    and solves the equation
     \begin{equation}
        \partial_tg^i_t = L_{h_t}g^i_t,
    \end{equation}
    yet
    \begin{equation}\label{contradictory assumption original sequence}
        \|g^i\|_{2; |\alpha| + 2,\hat{t}_i} > i\|E_{h}g^i\|_{0 ; |\alpha| + 4,\hat{t}_i}.
    \end{equation}
   Since $g^i_t$ solves a homogeneous linear equation, we may without loss of generality assume that
   \begin{equation}
       \|g^i\|_{2;|\alpha| + 2,\hat{t}_i} = 1.
   \end{equation}
   By the definition of the weighted norm given by definition \ref{finalweightednorm}, we can find $x_i \in \Sigma^+_{1}$, and $t_i \leq \hat{t}_i$ so that
   \begin{equation}\label{choice of xiti}
       \sum_{j = 0}^{2}(1 + |x_i|)^{|\alpha| + 2 + j}|\nabla^jg^i_{t_i}(x_i)| \in [\frac{1}{2}, 3].
   \end{equation}
   We now split into two cases. \\

   Case (i) : $\limsup_{i \to \infty}|x_i| < \infty$. In this case, we define a new sequence of smooth functions by
   \begin{equation}
       \Tilde{g}^i_t(x) = g^i_{t_i + t}(x) \in C^{\infty}(\Sigma^+_{1}\times (-\infty, 0]).
   \end{equation}
   Since $\Tilde{g}^i$ is just a time translation of the original $g^i$, we see that
   \begin{equation}\label{properties of case1 new defined sequence gi}
       \Tilde{g}^i \geq 0, \ \|\Tilde{g}^i\|_{2, |\alpha| + 2,0} \leq 1,
   \end{equation}
   \begin{equation}\label{the equation solved by newsequence case1}
       \partial_t\Tilde{g}^i_t - \Delta_{\Sigma^+_1}\Tilde{g}^i_t - |A_{\Sigma^+_1}|^2\Tilde{g}^i_t = E_{h_{t_i + t}}(\Tilde{g}^i_t),
   \end{equation}
   and
   \begin{equation}\label{contradictory property of newsequence case1}
       \|E_{h_{t_i + t}}(\Tilde{g}^i_t)\|_{0 ; |\alpha| + 4,0} \leq \|E_{h_t}g^i_t\|_{0 ; |\alpha| + 4,\hat{t}_i} \leq \frac{1}{i}, \ \sum_{j = 0}^{2}(1 + |x_i|)^{|\alpha| + 2 + j}|\nabla^j\Tilde{g}^i_{0}(x_i)| \in [\frac{1}{2}, 3].
   \end{equation}
   In view of estimates \eqref{regularity estimate for general g} and \eqref{properties of case1 new defined sequence gi}, and the uniform boundedness of $|x_i|$, we can pass to a subsequence, and obtain $g^{\infty} \in C^{\infty}(\Sigma^+_1 \times (-\infty, 0])$, and $x \in \Sigma^+_1$ so that 
   \begin{equation}
       \Tilde{g}^i \to g^{\infty} \text{ in }C^2_{loc}(\Sigma^+_1 \times (-\infty, 0]), \ x_i \to x \text{ as }i \to \infty.
   \end{equation}
   By \eqref{properties of case1 new defined sequence gi}, we see that
   \begin{equation}
       g^{\infty} \geq 0, \ \| g^{\infty}\|_{2;|\alpha|+2,0} \leq 1 < \infty.
   \end{equation}
   On the other hand, by \eqref{the equation solved by newsequence case1}, and \eqref{contradictory property of newsequence case1}, we see that $g^{\infty} \neq 0$ solves
   \begin{equation}
       \partial_tg^{\infty} - \Delta_{\Sigma^+_1}g^{\infty} - |A_{\Sigma^+_1}|^2g^{\infty} = 0.
   \end{equation}
  This contradicts the Liouville type theorem (proposition \ref{liouvlie for hardtsimonleaf}). Thus case (i) is not possible.\\

  Case (ii) : $\limsup_{i \to \infty}|x_i| = \infty$. In this case, by dropping several terms, we may assume that $|x_i| \geq 100$, and $|x_i| \to \infty$ as $i \to \infty$. We define $\Tilde{g}^i \in C^{\infty}(\Sigma^+_{\frac{1}{|x_i|}}\times (-\infty, 0])$ by
  \begin{equation}
      \Tilde{g}^i_t = |x_i|^{|\alpha| + 2}g^i_{t_i + |x_i|^2t}(|x_i|x).
  \end{equation}
 First note that
  \begin{equation}\label{basic property of rescaled sequence second case}
      \Tilde{g}^i \geq 0, \sup_{x \in \Sigma^+_{\frac{1}{|x_i|}} , \ t \leq 0}|x|^{|\alpha| + 2 + k}|\nabla^k\Tilde{g}^i_t(x)| \leq C(n,m,k) < \infty
  \end{equation}
  for any $0 \leq k \leq 10$, $0 < r  < \infty$. The non-negativity of $\Tilde{g}^i$ is obvious by \eqref{basic properties of original sequence}. As for the second inequality, for each $x \in \Sigma^+_{\frac{1}{|x_i|}}$, by \eqref{basic properties of original sequence}, and \eqref{regularity estimate for general g}, we have
  \begin{equation}
      \sup_{t \leq 0}|x|^{|\alpha| + 2 + k}|\nabla^k\Tilde{g}^i_t(x)| \leq \sup_{t \leq \hat{t}_i}(|x_i||x|)^{|\alpha| + 2 + k}|\nabla^k g^i_t(|x_i|x)| \leq \|g^i\|_{k;|\alpha| + 2, \hat{t}_i} \leq C(n,m,k),
  \end{equation}
  thus proving \eqref{basic property of rescaled sequence second case}. 
  Also $\Tilde{g}^i$ solves the equation
  \begin{equation}\label{the equation satisfied by second sequence}
      \partial_t\Tilde{g}^i_t - \Delta_{\Sigma^+_{\frac{1}{|x_i|}}}\Tilde{g}^i_t - |A_{\Sigma^+_{\frac{1}{|x_i|}}}|^2\Tilde{g}^i_t = E_{h^i_t}\Tilde{g}^i_t,
  \end{equation}
  with 
  \begin{equation}
      h^i_t(x) = \frac{1}{|x_i|}h_{t_i + |x_i|^2t}(|x_i|x).
  \end{equation}
  By \eqref{choice of xiti}, we also have $\Tilde{x}_i = \frac{x_i}{|x_i|} \in \Sigma^+_{\frac{1}{|x_i|}}$ so that
  \begin{equation}\label{the bad points in secodn sequnce}
      \sum_{j = 0}^{2}|\nabla^j\Tilde{g}^i_{0}(\Tilde{x}_i)| \geq \frac{1}{3}.
  \end{equation}
  By \eqref{contradictory assumption original sequence}, we have
  \begin{equation}\label{vanishing error term second case}
      \sup_{r \leq |x| \leq R, \ t \leq 0}|E_{h^i_t}\Tilde{g}^i_t| \leq  \sup_{r \leq |x| \leq R, \ t \leq \hat{t}_i}|x_i|^{|\alpha| +4}|E_{h}g^i_t(|x_i|x)| \leq \frac{C(n,m, r, R)}{i} \to 0 \text{ as }i \to \infty
  \end{equation}
  for each $0 < r < R < \infty$. Finally, note that 
  \begin{equation}
      \Sigma^+_{\frac{1}{|x_i|}} \to \Sigma \text{ locally smoothly away from the origin as }i \to \infty.
  \end{equation}
  Therefore, by viewing $\Tilde{g}^i$ as smooth functions defined over large domains in $\Sigma \setminus \{0\}$, after possibly passing to a subsequence, we can obtain $g^{\infty} \in C^{\infty}(\Sigma \setminus \{0\} \times (-\infty ,0])$, and $\Tilde{x} \in \Sigma \setminus \{0\}$ so that
  \begin{equation}
      \Tilde{g}^i \to g^{\infty} \text{ in }C^2_{loc}(\Sigma \setminus \{0\}), \Tilde{x}_i \to \Tilde{x}.
  \end{equation}
  By \eqref{the equation satisfied by second sequence}, and \eqref{vanishing error term second case}, we see that
  \begin{equation}
      \partial_t g^{\infty} - \Delta_{\Sigma}g^{\infty}_t - |A_{\Sigma}|^2g^{\infty}_t \equiv 0.
  \end{equation}
  On the other hand, by \eqref{basic property of rescaled sequence second case}, and \eqref{the bad points in secodn sequnce}, we see that $g^{\infty} \neq 0$, and
  \begin{equation}
      g^{\infty} \geq 0, \ \sup_{t \leq 0}\sup_{x \in \Sigma \setminus \{0\}}|g^{\infty}_t(x)||x|^{|\alpha| + 2} < \infty.
  \end{equation}
  This contradicts Liouville type theorem (proposition \ref{liouvile for simons cone}). Thus the second case is also impossible.\\

  Since both cases are impossible, we reached a contradiction. This completes the proof of lemma \ref{Schauder type estimate}.
\end{proof}
With lemma \ref{Schauder type estimate}, we can show that $g \equiv 0$.
\begin{lemma}\label{flow is stationary}
    \begin{equation}
        g = \partial_th \equiv 0 \text{ in }\Sigma^+_1 \times (-\infty, \overline{t}].
    \end{equation}
\end{lemma}
\begin{proof}[Proof of lemma \ref{flow is stationary}]
    By lemma \ref{Schauder type estimate}, there exists a fixed constant $C_0 > 0$ so that for all $\hat{t} \leq \overline{t}$, we have
    \begin{equation}
        \|g\|_{2;|\alpha|+2,\hat{t}} \leq C_0\|E_h(g)\|_{0;|\alpha| + 4,\hat{t}}.
    \end{equation}
    Recalling the structure of $E_h$ given by \eqref{weightedc2normboundfor graph functin over simons cone}, and corollary \ref{global graphicality of flow over hardtsimon leaf}, for any $\eta \in (0, \overline{\eta}(n,m))$, we can find $\hat{t} < 0$ so that for all $t \leq \hat{t}$, 
    \begin{equation}
        \|E_h(g)\|_{0;|\alpha| + 4,\hat{t}} \leq C(n,m)\eta  \|g\|_{2;|\alpha|+2,\hat{t}}.
    \end{equation}
    Therefore by choosing $\eta > 0$ so that
    \begin{equation}
        C(n,m)C_0\eta = \frac{1}{2},
    \end{equation}
    we obtain
    \begin{equation}
         \|g\|_{2;|\alpha|+2,\hat{t}} \leq \frac{1}{2} \|g\|_{2;|\alpha|+2,\hat{t}},
    \end{equation}
    hence
    \begin{equation}
         \|g\|_{2;|\alpha|+2,\hat{t}} = 0.
    \end{equation}
    This means $g_t(x) \equiv 0$ for all $x \in \Sigma^+_1$, $t \leq \hat{t}$. Since $g_t$ satisfies a linear second order parabolic PDE, and is uniformly bounded up to time $\overline{t}$, we conclude that $g \equiv 0$, proving lemma \ref{flow is stationary}. 
\end{proof}
As a consequence, we finally obtain the desired uniqueness theorem.
\begin{theorem}\label{uniqueness of mean convex flow}[Theorem \ref{maintheorem : uniqueness in mean convex case}] 
    Let $(M_t)_{t \in (-\infty, 0)}$ be a smooth, properly embedded mean curvature flow which satisfies assumptions \ref{strong convergence to simonscone assumption1}, \ref{additional geoemtric assumption}, \ref{mean convexity assumption4}. Then $M_t \equiv \Sigma^+_1$ up to dilation.
\end{theorem}
\begin{proof}[Proof of theorem \ref{uniqueness of mean convex flow}]
    By lemma \ref{flow is stationary}, $M_t$ is stationary for all $t \leq \overline{t}$ for some $\overline{t}$. Thus by corollary \ref{global graphicality of flow over hardtsimon leaf}, 
    \begin{equation}
        M_t \equiv \Sigma^+_1
    \end{equation}
    for all $t \leq \overline{t}$. To control $t \geq \overline{t}$, we simply apply the avoidance principle to $M_t$ against nearby leaves $\Sigma^+_{1 - \epsilon}, \ \Sigma^+_{1 + \epsilon}$ for each small $\epsilon > 0$. 
\end{proof}

\section{Construction and analysis of inner self-shrinkers}\label{innershrinkersection}
Throughout this section, we take either $n+m \geq 9$ with $n,m \geq2$, or $(n,m) = (3,5), (4,4), (5,3)$. Note that these $(n,m)$ are when the $O(n)\times O(m)$ symmetric minimal quadratic cone admits a Hardt-Simon foliation \cite{hardt1985area}. The goal of this section is to construct a family of self-shrinkers, which we call inner self-shrinkers (proposition \ref{inner family of self shrinkers}). This family is a one parameter family of $O(n) \times O(m)$ symmetric self-shrinkers with boundary which are smooth and properly embedded in a fixed small ball centered at the origin. The construction of these surfaces is carried out by a careful analysis of solutions to the ODE obtained by reducing the self-shrinker PDE using the $O(n)\times O(m)$ symmetry. We then obtain an approximate formula for the profile function (proposition \ref{approximate formula for inner self shrinkers}). In the last part of this section, we show that the self-shrinkers constructed in proposition \ref{inner family of self shrinkers} together with the minimal quadratic cone foliate a fixed small ball centered at the origin. The key philosophy behind the results of this section is that near the origin, we can view the right hand side of the self-shrinker equation
\begin{equation}
    H = \frac{\langle x, \nu\rangle}{2}
\end{equation}
as a perturbation of the minimal surface equation. This suggests that these inner self-shrinkers would behave similarly to the leaves of the Hardt-Simon foliation near the origin.\\

We first construct the inner self-shrinkers.
\begin{proposition}\label{inner family of self shrinkers}
For each $0 < \theta < \sqrt{2(m-1)}$, there exists $u_{\theta} : [0, \sqrt{2(n-1)}] \to \mathbf{R}_+$ so that\\\\
(i) $u_{\theta} \in C^2([0, \sqrt{2(n-1)}])$, and satisfies the initial value problem
\begin{equation}\label{initial value problem for inner self shrinkers}
    \begin{cases}
        \frac{u_{\theta}''}{1 + (u_{\theta}')^2} = (\frac{x}{2} - \frac{n-1}{x})u_{\theta}' - \frac{u_{\theta}}{2} + \frac{m-1}{u_{\theta}} \\ u_{\theta}(0) = \theta, \ u_{\theta}'(0) = 0.
    \end{cases}
\end{equation}
In other words, the hypersurface given by the map
\begin{equation}
    X_{\theta} : [0, \sqrt{2(n-1)}] \times \mathbf{S}^{n-1} \times \mathbf{S}^{m-1} \ni (x, w_1, w_2) \to (xw_1, u_{\theta}(x)w_2) \in \Gamma_{\theta} \subset \mathbf{R}^{n+m}.
\end{equation}
is an $O(n)\times O(m)$ symmetric self-shrinker.
\end{proposition}
\begin{remark}\label{remark : construction of second family by reflection}
    By applying proposition \ref{inner family of self shrinkers} with $n,m$ reversed, and then reflecting the resulting self-shrinker via the map
    \begin{equation}
        R_{n,m} : \mathbf{R}^m\times \mathbf{R}^n \ni (y,x) \to (x,y) \in \mathbf{R}^n\times \mathbf{R}^m,
    \end{equation}
    we obtain another $O(n)\times O(m)$ symmetric self-shrinker $\Lambda_{\theta}$ for $0 < \theta < \sqrt{2(n-1)}$.
\end{remark}
\begin{proof}[Proof of proposition \ref{inner family of self shrinkers}]
It turns out that it is easier to work with the rescaled function
\begin{equation}
    \Tilde{u}_{\theta}(y) = \theta^{-1}u_{\theta}(\theta y), \ x = \theta y.
\end{equation}
Then the initial value problem \eqref{initial value problem for inner self shrinkers} is equivalent to 
\begin{equation}\label{initial value problem for inner self shrinkers, rescaled version}
    \begin{cases}
        \frac{ \Tilde{u}_{\theta}''}{1 + ( \Tilde{u}_{\theta}')^2} + \frac{n-1}{y} \Tilde{u}_{\theta}' - \frac{m-1}{ \Tilde{u}_{\theta}} = \frac{\theta^2}{2}(y\Tilde{u}_{\theta}' -  \Tilde{u}_{\theta}) \\  \Tilde{u}_{\theta}(0) = 1, \ \Tilde{u}_{\theta}'(0) = 0.
    \end{cases}
\end{equation}
We now change variable by
\begin{equation}
    y(\tau) = e^{\tau}, \ z = z(\tau; \theta) =  \Tilde{u}_{\theta}(y(\tau)), \ w = w(\tau ; \theta) =  \Tilde{u}_{\theta}'(y(\tau)).
\end{equation}
Then the initial value problem \eqref{initial value problem for inner self shrinkers, rescaled version} reduces to a first order ODE system
\begin{equation}\label{reducd ODE system}
    \begin{cases}
        \frac{dy}{d\tau} = y \\ 
        \frac{dz}{d\tau} = yw \\ 
        \frac{dw}{d\tau} = (1 + w^2)(\frac{\theta^2}{2}y^2w - \frac{\theta^2}{2}yz + \frac{m-1}{z}y - (n-1)w),
    \end{cases}
\end{equation}
with 
\begin{equation}
    \lim_{\tau \to -\infty}(y, z, w) = (0,1,0).
\end{equation}
We then obtain (short time) existence, uniqueness, and smooth dependence on $\theta$ by the unstable manifold theorem. Setting 
\begin{equation}
    F(y,z,w ; \theta) = (y, yw, (1 + w^2)(\frac{\theta^2}{2}y^2w - \frac{\theta^2}{2}yz + \frac{m-1}{z}y - (n-1)w)),
\end{equation}
the linearization of $F$ at $(0,1,0)$ is
\begin{equation}
    DF|_{(0,1,0)}(\cdot, \cdot, \cdot ; \theta) = \begin{pmatrix} 1 & 0 &  0 \\ 0 & 0 & 0 \\ -\frac{\theta^2}{2} + (m-1) & 0 & -(n-1)\end{pmatrix}.
\end{equation}
In particular, the unstable manifold of the system \eqref{reducd ODE system} at $(0,1,0)$ is one dimensional, and corresponds to the positive eigenvalue of the above matrix, which is equal to $1$. Since the corresponding eigenvector is
\begin{equation}
    \begin{pmatrix} 1  \\ 0 \\ \frac{2(m-1) - \theta^2}{2n} \end{pmatrix},
\end{equation}
$(z,w)$ can be written as a function of $y$ for small $y > 0$. Therefore the solution to the initial value problem \eqref{initial value problem for inner self shrinkers, rescaled version} corresponds to the unstable manifold of the system \eqref{reducd ODE system} at $(0,1,0)$. Therefore we see that it must uniquely exist up to some maximal existence point $\overline{y} = \overline{y}(\theta) > 0$, and depends smoothly on $\theta$. The fact that $\Tilde{u}_{\theta}$ is $C^2$ up to $y = 0$ follows from the unstable manifold theorem combined with the ODE \eqref{initial value problem for inner self shrinkers, rescaled version}. In fact, the ODE implies that
\begin{equation}
    \Tilde{u}_{\theta}''(0) = \frac{1}{n}(m-1 - \frac{\theta^2}{2})> 0.
\end{equation}
To finish the proof of proposition \ref{inner family of self shrinkers}, we need to show that $\overline{y}(\theta) > \frac{\sqrt{2(n-1)}}{\theta}$. Suppose this is not true, and $\overline{y}(\theta) \leq \frac{\sqrt{2(n-1)}}{\theta}$ for some $\theta \in (0, \sqrt{2(m-1)})$. Then in view of the ODE \eqref{initial value problem for inner self shrinkers, rescaled version}, we see that 
\begin{equation}\label{blowup criterion of rescaled profile curve}
    \limsup_{y \to \overline{y}(\theta)-} |\Tilde{u}_{\theta}(y)| + |\Tilde{u}_{\theta}'(y)| + \frac{1}{|\Tilde{u}_{\theta}(y)|} = \infty.
\end{equation}
We now claim that the above blowup is not possible when $\overline{y}(\theta) \leq  \frac{\sqrt{2(n-1)}}{\theta}$. Multiply the ODE \eqref{initial value problem for inner self shrinkers, rescaled version} by $\Tilde{u}_{\theta}'$. Then
\begin{align*}
    \frac{d}{dy}(\ln \sqrt{1 + (\Tilde{u}_{\theta}')^2}) & = \frac{\Tilde{u}_{\theta}'\Tilde{u}_{\theta}''}{1 + (\Tilde{u}_{\theta}')^2} \\ & = (\frac{\theta^2y}{2} - \frac{n-1}{y})(\Tilde{u}_{\theta}')^2 + \frac{(m-1)\Tilde{u}_{\theta}'}{\Tilde{u}_{\theta}} - \frac{\theta^2}{2}\Tilde{u}_{\theta}\Tilde{u}_{\theta}' \\ & \leq \frac{(m-1)\Tilde{u}_{\theta}'}{\Tilde{u}_{\theta}} - \frac{\theta^2}{2}\Tilde{u}_{\theta}\Tilde{u}_{\theta}' \\ &= \frac{d}{dy}(\ln \Tilde{u}_{\theta}^{m-1} - \frac{\theta^2}{4}\Tilde{u}_{\theta}^2),
\end{align*}
where the inequality follows from the assumption $\overline{y}(\theta) \leq  \frac{\sqrt{2(n-1)}}{\theta}$. Thus, by integrating in $y$, using the initial values given in \eqref{initial value problem for inner self shrinkers, rescaled version}, and removing the logarithm, we obtain 
\begin{equation}\label{uniform bound of profile curve}
    \sqrt{1 + (\Tilde{u}_{\theta}')^2} \leq \Tilde{u}_{\theta}^{m-1}e^{\frac{\theta^2}{4}}e^{-\frac{\theta^2}{4}\Tilde{u}_{\theta}^2} \leq M = M(n,m, \theta) < \infty.
\end{equation}
This implies that 
\begin{equation}
    \sup_{y \in [0, \overline{y})}|\Tilde{u}_{\theta}'(y)| \leq M < \infty.
\end{equation}
By the fundamental theorem of calculus combined with $\Tilde{u}_{\theta}(0) = 1$, we obtain
\begin{equation}
    \sup_{y \in [0, \overline{y})}|\Tilde{u}_{\theta}(y)| \leq 1 + \frac{\sqrt{2(n-1)}}{\theta}M < \infty.
\end{equation}
Finally, by \eqref{uniform bound of profile curve}, we see that
\begin{equation}
    1 \leq \Tilde{u}_{\theta}^{m-1}e^{\frac{\theta^2}{4}}e^{-\frac{\theta^2}{4}\Tilde{u}_{\theta}^2},
\end{equation}
which immediately implies that 
\begin{equation}
    \sup_{y \in [0, \overline{y})}\frac{1}{|\Tilde{u}_{\theta}(y)|} \leq M'= M'(n,m, \theta) < \infty.
\end{equation}
Therefore, we get a contradiction to \eqref{blowup criterion of rescaled profile curve}, hence $\overline{y}(\theta) > \frac{\sqrt{2(n-1)}}{\theta}$, thus completing the proof of proposition \ref{inner family of self shrinkers}.
\end{proof}
Proposition \ref{inner family of self shrinkers} implies that the self-shrinkers for each $\theta \in (0, \sqrt{2(m-1)})$ are properly embedded in $B(0, \sqrt{2(\min(n,m)-1)})$. We now show that by focusing on a possibly smaller ball, these self-shrinkers behave similarly to the leaves of the Hardt-Simon foliation of the minimal quadratic cone. \\

To make this precise, we introduce the following weighted norm which is motivated by section 4.3 in \cite{stolarski2023existence}. 
\begin{definition}\label{weighted norm for inner shrinker construction}
    Let $v \in C^2([0, y_*] ; \mathbf{R})$. We define for each $k = 0,1,2$ the weighted norm
    \begin{equation}
        \|v\|_{k,y_*} = \sup_{y \in (0, y_*]}\sum_{i = 0}^k|\frac{v^{(i)}(y)}{y^{k-i}(1 + y)^{\alpha}}|.
    \end{equation}
    Here
    \begin{equation}
        \alpha =  \frac{-(d-2) + \sqrt{d^2-8d+8}}{2} \in [-2, -1), \ d = n+m-1.
    \end{equation}
\end{definition}
Note that in the proof of proposition \ref{inner family of self shrinkers}, we worked with the rescaled profile function which satisfies the initial value problem
\begin{equation}\label{rescaled initial value problem}
    \begin{cases}
        \frac{ \Tilde{u}_{\theta}''}{1 + ( \Tilde{u}_{\theta}')^2} + \frac{n-1}{y} \Tilde{u}_{\theta}' - \frac{m-1}{ \Tilde{u}_{\theta}} = \frac{\theta^2}{2}(y\Tilde{u}_{\theta}' -  \Tilde{u}_{\theta}) \\  \Tilde{u}_{\theta}(0) = 1, \ \Tilde{u}_{\theta}'(0) = 0.
    \end{cases}
\end{equation}
If we focus only on the above initial value problem, note that the argument involving the unstable manifold theorem in the proof of proposition \ref{inner family of self shrinkers} goes through even if we take $\theta$ to be any real number in $(-\sqrt{2(m-1)}, \sqrt{2(m-1)})$. This means that we actually get a smooth family of functions $\Tilde{u}_{\theta}$ for all $\theta \in (-\sqrt{2(m-1)}, \sqrt{2(m-1)})$. Moreover, since the initial value problem \eqref{rescaled initial value problem} remains unchanged if we replace $\theta$ with $-\theta$, this implies that $\Tilde{u}_{\theta}  = \Tilde{u}_{-\theta}$. Finally, note that if $\theta = 0$, then \eqref{rescaled initial value problem} is exactly equal to the minimal surface equation reduced to the $O(n)\times O(m)$ setting. By uniqueness, this implies that
\begin{equation}
    \Tilde{u}_0(y) = \psi_1(y)=\psi_1(y;n,m),
\end{equation}
where $\psi_1(y)$ is the solution to the initial value problem \eqref{minimal surface equation}. \\

The above discussion implies that we actually get a smooth family of $C^2$ functions $(\Tilde{u}_{\theta})_{|\theta| < \sqrt{2(m-1)}}$ so that  $\Tilde{u}_{\theta} \in C^2([0, \frac{\sqrt{2(n-1)}}{|\theta|}])$ with $\Tilde{u}_{\theta} = \Tilde{u}_{-\theta}$ for $\theta \neq 0$, and $ \Tilde{u}_0(y) = \psi_1(y)$.  Then for $\theta$ small enough, we can view $\Tilde{u}_{\theta}$ as a perturbation of $\psi_1$; hence we can write $\Tilde{u}_{\theta}$ as
\begin{equation}\label{decomposition of rescaled profile curve to main and error}
    \Tilde{u}_{\theta} = \psi_1 + \theta^2E_{\theta}.
\end{equation}
We now state the key estimate which makes our claim precise. 
\begin{lemma}\label{key estimate for inner shrinkers}
    There exist $M_0 = M_0(n,m) > 3$, $\theta_0 = \theta_0(n,m)\in (0, \sqrt{2(m-1)})$ so that whenever $0 < \theta \leq \theta_0$, $M \geq M_0$
    \begin{equation}
        \|E_{\theta}\|_{2, \frac{1}{10M\theta}} \leq M,
    \end{equation}
    and 
  \begin{equation}
      \Tilde{u}_{\theta} \geq \frac{1}{2}\psi_1 \textup{ for all }0 \leq y \leq \frac{1}{10M_0\theta}.
  \end{equation}
\end{lemma}
\begin{proof}[Proof of lemma \ref{key estimate for inner shrinkers}]
In the proof below, $c(n,m)>0$ denotes any positive constant that depends only on $n,m$, and can change from line to line.   
    We first compute the equation for $E_{\theta}$. By substituting $ \Tilde{u}_{\theta} = \psi_1 + \theta^2E_{\theta}$ in the initial value problem \eqref{rescaled initial value problem}, we obtain
    \begin{equation}\label{equation of Error term inner self shrinker}
        \begin{cases}
           \frac{E_{\theta}''}{1 + (\Tilde{u}_{\theta}')^2} + \frac{n-1}{y}E_{\theta}' + \frac{(m-1)E_{\theta}}{\psi_1\Tilde{u}_{\theta}} - \frac{2\psi_1'\psi_1''E_{\theta}' + \theta^2(E_{\theta}')^2\psi_1''}{(1 + (\psi_1')^2)(1 + (\Tilde{u}_{\theta}')^2)}  = \frac{1}{2}(y\psi_1' - \psi_1) + \frac{\theta^2}{2}(yE_{\theta}' - E_{\theta}) \\
           E_{\theta}(0) = E_{\theta}'(0) = 0.
        \end{cases}
    \end{equation}
    First note that $E_{\theta} = \frac{1}{\theta^2}(\Tilde{u}_{\theta} - \psi_1) \in C^2([0, \frac{\sqrt{2(n-1)}
    }{\theta}])$, and $E_{\theta}(0) = E'_{\theta}(0) = 0, \ E_{\theta}''(0)= -\frac{1}{2n}$ by the differential equation.  This implies that for each $0 < \theta < \sqrt{2(m-1)}$, one can find some $y_* > 0$ so that 
    \begin{equation}
        \|E_{\theta}\|_{2, y_*} \leq 3, \ \Tilde{u}_{\theta} = \psi_1 + \theta^2E_{\theta} \geq \frac{1}{2}\psi_1 \textup{ in }[0, y_*].
    \end{equation}
    Now, for each $M > 3$, $0 < \theta < \sqrt{2(m-1)}$, we define
    \begin{equation}\label{def of maximal good time}
        \overline{y} = \overline{y}(\theta, M) = \sup\{ y_*  > 0\ | \ \|E_{\theta}\|_{2, y_*} \leq M, \ \Tilde{u}_{\theta} = \psi_1 + \theta^2E_{\theta} \geq \frac{1}{2}\psi_1 \textup{ in }[0, y_*]\}.
    \end{equation}
    By the previous discussion, $\overline{y} > 0$ is well defined. \\

    We claim that we can find $M_0 = M_0(n,m) > 3$, and $\theta_0 = \theta_0(n,m) > 0$ so that for all $M \geq M_0$, $0 < \theta \leq \theta_0$, \begin{equation}
         \overline{y} = \overline{y}(\theta, M) > \frac{1}{10M\theta}.
    \end{equation}
    Suppose this is not true, and that $\overline{y}(\theta, M) \leq \frac{1}{10M\theta}$. 
    To obtain a nice estimate for $E_{\theta}$, we recall the linearized minimal surface equation at $\psi_1$ given by \eqref{linearized minimal surface equation at leaf}. By using the inverse formula \eqref{Solvability of linearized operator definition eq}, one can directly check that
    \begin{equation}\label{estimate for linearized minimal surface equation}
    \|u\|_{2, y_*} \leq c(n,m)\|f\|_{0, y_*}
\end{equation}
when $u$ is the solution to the initial value problem
\begin{equation}
    \begin{cases}
        \mathcal{L}u = f \text{ in }(0, y_*] \\ u(0) = u'(0) = 0.
    \end{cases}
\end{equation}
   Thus, we can find $g \in C^2([0, \infty))$ so that 
\begin{equation}
    \begin{cases}
        \mathcal{L}g = \frac{1}{2}(y\psi_1' - \psi_1) \\ g(0) = g'(0) = 0,
    \end{cases}
\end{equation}
and 
\begin{equation}
    \|g\|_{2, \infty} \leq c(n,m)\|y\psi_1' - \psi_1\|_{0, \infty} = c(n,m) < \infty.
\end{equation}
    Since when we take $\theta = 0$ in \eqref{equation of Error term inner self shrinker}, it reduces to the initial value problem for $g$, we view $E_{\theta}$ as a perturbation of $g$, and thus consider
    \begin{equation}
        v_{\theta} = E_{\theta} - g.
    \end{equation}
    Then, by using the definition of $g$, and \eqref{equation of Error term inner self shrinker}, we see that $v_{\theta}$ solves
    \begin{equation}
        \begin{cases}
            \mathcal{L}v_{\theta} = F_{\theta} = F_{1, \theta} + F_{2, \theta} + F_{3, \theta} + F_{4, \theta} +F_{5, \theta} \\ v_{\theta}(0) = v_{\theta}'(0) = 0,
        \end{cases}
    \end{equation}
    with
    \begin{equation}
        F_{1, \theta} = \frac{E_{\theta}''(2\theta^2E'_{\theta}\psi_1' + \theta^4(E_{\theta}')^2)}{(1 + (\psi_1')^2)(1 + (\Tilde{u}_{\theta}')^2)},
    \end{equation}
    \begin{equation}
        F_{2, \theta} =  \frac{(m-1)\theta^2E^2_{\theta}}{\psi_1^2\Tilde{u}_{\theta}},
    \end{equation}
\begin{equation}
    F_{3, \theta} = \frac{\theta^2}{2}(yE'_{\theta} - E_{\theta}),
\end{equation}
\begin{equation}
    F_{4, \theta} =  \frac{\theta^2(E_{\theta}')^2\psi_1''}{(1 + (\psi_1')^2)(1 + (\Tilde{u}_{\theta}')^2)},
\end{equation}
and
\begin{equation}
    F_{5, \theta} =  -\frac{2\psi_1'\psi_1''E_{\theta}'(2\theta^2E_{\theta}'\psi_1' + \theta^4(E_{\theta}')^2)}{(1 + (\psi_1')^2)^2(1 + (\Tilde{u}_{\theta}')^2)}.
\end{equation}
We now estimate 
\begin{equation}
    \|F_{i, \theta}\|_{0, \overline{y}(\theta, M)} \textup{ for }i = 1,2,3,4,5
\end{equation}
under the assumption that $\overline{y} \leq \frac{1}{10M\theta}$, $M > 3$. Recall that $\overline{y}$ is given by \eqref{def of maximal good time}. For each $y \in (0, \overline{y}]$, 
\begin{align*}
    |\frac{F_{1, \theta}}{(1 + y)^{\alpha}}|& \leq |\frac{2\theta^2E_{\theta}'E_{\theta}''\psi_1'}{(1 + y)^{\alpha}}| + |\frac{\theta^4(E_{\theta}')^2E_{\theta}''}{(1 + y)^{\alpha}}| \\& \leq \|E\|_{2, \overline{y}}^2|\frac{2\theta^2c(m,n)y(1 + y)^{2\alpha}}{(1 + y)^{\alpha}}| + \|E\|_{2, \overline{y}}^3|\frac{\theta^4y^2(1 + y)^{3\alpha}}{(1 + y)^{\alpha}}| \\ & \leq 2c(n,m)M^2\theta^2\overline{y} + M^3\theta^4\overline{y}^2 \leq 2c(n,m)\theta^2M^2\frac{1}{10M\theta} + M^3\theta^4\frac{1}{100M^2\theta^2} \\ & \leq c(n,m)\theta M + \frac{\theta^2}{100}M.
\end{align*}
\begin{align*}
    |\frac{F_{2, \theta}}{(1 + y)^{\alpha}}| &\leq |\frac{(m-1)\theta^2E^2_{\theta}}{\psi_1^2\Tilde{u}_{\theta}(1 + y)^{\alpha}}| \\ & \leq \|E_{\theta}\|_{2, \overline{y}}^2|\frac{c(n,m)\theta^2y^4(1 + y)^{2\alpha}}{y^3(1 + y)^{\alpha}}| \\ & \leq M^2c(n,m)\theta^2\overline{y} \leq M^2c(n,m)\theta^2\frac{1}{10M\theta} \\ & \leq c(n,m)\theta M.
\end{align*}
\begin{align*}
    |\frac{F_{3, \theta}}{(1 + y)^{\alpha}}| & \leq |\frac{\theta^2yE_{\theta}'}{2(1 + y)^{\alpha}}| + |\frac{\theta^2E_{\theta}}{2(1 + y)^{\alpha}}| \\ & \leq \|E_{\theta}\|_{2, \overline{y}}\frac{\theta^2y^2(1 + y)^{\alpha}}{(1 + y)^{\alpha}} \leq M\theta^2\overline{y}^2 \leq \frac{1}{100M} \leq 1.
\end{align*}
\begin{align*}
    |\frac{F_{4, \theta}}{(1 + y)^{\alpha}}| &\leq |\frac{\theta^2(E_{\theta}')^2\psi_1''}{(1 + y)^{\alpha}}| \leq \|E_{\theta}\|_{2, \overline{y}}^2|\frac{\theta^2y^2(1 + y)^{2\alpha}c(n,m)(1 + y)^{\alpha - 2}}{(1 + y)^{\alpha}}| \\ & \leq M^2\theta^2c(n,m)\overline{y}^2 \leq c(n,m).
\end{align*}
\begin{align*}
    |\frac{F_{5, \theta}}{(1 + y)^{\alpha}}| & \leq |\frac{4(\psi_1')^2\psi_1''(E_{\theta}')^2\theta^2}{(1 + y)^{\alpha}}| + |\frac{2\psi_1'\psi_1'' \theta^4(E_{\theta}')^3}{(1 + y)^{\alpha}}| \\ & \leq 4\|E_{\theta}\|_{2, \overline{y}}^2|\frac{c(n,m)(1 + y)^{\alpha - 2}y^2(1 + y)^{2\alpha}\theta^2}{(1 + y)^{\alpha}}| + 2\|E_{\theta}\|_{2, \overline{y}}^3|\frac{\theta^4c(n,m)(1 + y)^{\alpha - 2}y^3(1 + y)^{3\alpha}}{(1 + y)^{\alpha}}| \\ & \leq 4c(n,m)M^2\theta^2\overline{y}^2 + 2c(n,m)M^3\theta^4\overline{y}^3 \leq c(n,m)(1 + \theta)\leq c(n,m).
\end{align*}
Therefore, by combining all the above estimates of the `inhomogeneous term', and using the estimate \eqref{estimate for linearized minimal surface equation}, we obtain
\begin{equation}
    \|v_{\theta}\|_{2, \overline{y}} \leq c(n,m)\|F_{\theta}\|_{0 ,\overline{y}} \leq c(n,m)(\theta M + 1,
\end{equation}
Combining this estimate with the fact that $E_{\theta} = v_{\theta} + g$, and $\|g\|_{2, \infty} \leq c(n,m) < \infty$, $0 < \theta < \sqrt{2(m-1)}$, we obtain
\begin{equation}\label{final c2 estimate of error term}
   \|E_{\theta}\|_{2, \overline{y}} \leq c(n,m)(\theta M + 1).
\end{equation}
We now choose $M_0(n,m) > 3$, and $\theta_0(n,m) \in (0, \sqrt{2(m-1)})$. 
We first choose $\theta_0 = \theta_0(n,m) < \sqrt{2(m-1)}$ so that $c(n,m)\theta_0 \leq \frac{1}{3}$, and then choose $M_0(n,m) > 3$ so that $c(n,m) \leq \frac{M_0(n,m)}{3}$. Then if $0 < \theta \leq \theta_0$, $M \geq M_0$, then estimate \eqref{final c2 estimate of error term} implies that 
\begin{equation}
    \|E_{\theta}\|_{2, \overline{y}} \leq \frac{2}{3}M < M.
\end{equation}
We now show that for all $y \in [0, \overline{y}]$, 
\begin{equation}
  \Tilde{u}_{\theta} = \psi_1 + \theta^2E_{\theta} > \frac{2}{3}\psi_1.  
\end{equation}
Recalling \eqref{def of maximal good time} and using $\overline{y} \leq \frac{1}{10M\theta}$, we have
\begin{equation}
    |\theta^2E_{\theta}(y)| \leq M\theta^2y^2(1 + y)^{\alpha} \leq (M\theta^2 \overline{y})y \leq \frac{\theta}{10}y.
\end{equation}
Since $\psi_1 > c(n,m)y$ for all $y\geq 0$, we obtain for all $y \in [0, \overline{y}]$,
\begin{equation}\label{final lower bound of tildeu}
    \Tilde{u}_{\theta}(y) = \psi_1 + \theta^2E_{\theta} \geq \psi_1 - \frac{\theta}{10}y  \geq (1 - c(n,m)\theta)\psi_1 > \frac{2}{3}\psi_1
\end{equation}
provided $0 < \theta \leq \theta_0(n,m)$.
Then estimates \eqref{final c2 estimate of error term}, and \eqref{final lower bound of tildeu} give us the desired contradiction to the maximality of $\overline{y}$ whenever $\overline{y} \leq \frac{1}{10M\theta}$. This completes the proof of lemma \ref{key estimate for inner shrinkers}.
\end{proof}
One consequence of lemma \ref{key estimate for inner shrinkers} is the following uniform graphicality of self-shrinkers constructed in proposition \ref{inner family of self shrinkers} over the $O(n)\times O(m)$ symmetric minimal quadratic cone. 
\begin{corollary}\label{graphicality of self shrinkers over simons cone}
    There exist $\theta_0 = \theta_0(n,m) \in (0, \sqrt{2(m-1)})$, $r_0 = r_0(n,m) > 0$ so that for each $0 < \theta \leq \theta_0$, if we let $u_{\theta}$ be the profile function given by proposition \ref{inner family of self shrinkers}, then there exists $\hat{u}_{\theta} : [\sqrt{\frac{m-1}{n+m-2}}\theta, r_0] \to \mathbf{R}_+$ so that 
    \begin{equation}\label{normal graph over simons cone of inner self shrinkers}
        \{v_{n,m}r + \nu_{n,m}\hat{u}_{\theta}(r)\ | \ r \in [\sqrt{\frac{m-1}{n+m-2}}\theta, r_0]\} \subset \{(x, u_{\theta}(x)) \ | \ x \in [0, \sqrt{2(n-1)}\},
    \end{equation}
    where
    \begin{equation}
        v_{n,m} = (\sqrt{\frac{n-1}{n+m-2}}, \sqrt{\frac{m-1}{n+m-2}}), \ \nu_{n,m} = (-\sqrt{\frac{m-1}{n+m-2}}, \sqrt{\frac{n-1}{n+m-2}}).
    \end{equation}
    In other words, the self-shrinker corresponding to $u_{\theta}$ can be expressed as a normal graph over the $O(n)\times O(m)$ symmetric minimal quadratic cone in a fixed small ball centered at the origin independent of small $\theta >0$.
\end{corollary}
\begin{proof}[Proof of corollary \ref{graphicality of self shrinkers over simons cone}]
    We first let $\theta_0(n,m), \ M_0(n,m) > 0$ be the constants from lemma \ref{key estimate for inner shrinkers}. Then, we can compute for each $y \in [0, \frac{1}{10M_0\theta}]$, $0 < \theta \leq \theta_0$,
  \begin{align*}
      \Tilde{u}_{\theta}(y) &= \psi_1(y) + \theta^2E_{\theta}(y) \\ & \geq \psi_1(y) - \theta^2M_0(\frac{1}{M_0\theta})^2(1 + y)^{\alpha}\\ & = \psi_1(y) - \frac{1}{M_0}(1 + y)^{\alpha},
  \end{align*}
  and
    \begin{align*}
        \Tilde{u}_{\theta}''(y) &= \psi_1''(y) + \theta^2E_{\theta}''(y) \\ & \geq c_0(n,m)(1 + y)^{\alpha - 2} - \theta^2M_0(1 + y)^{\alpha} \\ & \geq (1 + y)^{\alpha}(c_0(n,m)(1 + y)^{-2} - \theta^2M_0), 
    \end{align*}
    where $c_0(n,m) > 0$. Therefore, by possibly enlarging $M_0(n,m) >3$ in view of the asymptotic formula of $\psi_1$ given by \eqref{asymptotic behavior of hardsimonleafformula}, and then possibly shrinking $\theta_0(n,m) > 0$, we can ensure that
    \begin{equation}\label{onesidedness of inner self shrinker}
        \Tilde{u}_{\theta}(y) > \sqrt{\frac{m-1}{n-1}}y, \ \Tilde{u}_{\theta}''(y) > 0
    \end{equation}
    whenever $0 < \theta \leq \theta_0(n,m)$, $y \in [0, \frac{1}{10M_0\theta}]$. Setting $x_0(n,m) = \frac{1}{10M_0} > 0$ and rescaling back, we conclude that for every $0 < \theta \leq \theta_0$, $ 0 \leq x \leq x_0(n,m)$, 
    \begin{equation}
        u_{\theta}''(x) > 0, \ u_{\theta}(x) >\sqrt{\frac{m-1}{n-1}}x .
    \end{equation}
    Combining these estimates with the fact that $u_{\theta}'(0) = 0$, we see that the graph of $u_{\theta}|_{[0, x_0(n)]}$ can be expressed as \eqref{normal graph over simons cone of inner self shrinkers} for some $\hat{u}_{\theta} > 0$, and $r_0 = r_0(n,m, \theta) > \frac{\theta}{\sqrt{2}}$. Moreover, because $u_{\theta}(x) > \sqrt{\frac{m-1}{n-1}}x$ for $x \in [0, x_0]$, we must have $r_0(n,m, \theta) \geq x_0(n,m) > 0$; thus we may set $r_0(n,m) = x_0(n,m) > 0$ to finish the proof. 
\end{proof}

We now use lemma \ref{key estimate for inner shrinkers} and the asymptotic formula of $\psi_1$ given by \eqref{asymptotic behavior of hardsimonleafformula} to obtain an approximation formula of $\hat{u}_{\theta}$ obtained in corollary \ref{graphicality of self shrinkers over simons cone}.
\begin{proposition}\label{approximate formula for inner self shrinkers}
   Let $0 < \theta \leq \theta_0(n,m)$, where $\theta_0$ is chosen so that both lemma \ref{key estimate for inner shrinkers}, and corollary \ref{graphicality of self shrinkers over simons cone} hold. Let $u_{\theta}$ be the profile function given by proposition \ref{inner family of self shrinkers}. Then there exists $c_0 = c_0(n,m)> 0$ so that for every small $\epsilon > 0$, there exists $\delta = \delta(n, m,\epsilon) > 0$ so that whenever
    \begin{equation}
         r \leq \delta, \ \frac{\theta}{r} \leq \delta,
    \end{equation}
    then
    \begin{equation}
        |\hat{u}_{\theta}(r) - c_0\theta^{1 - \alpha}r^{\alpha}| \leq \epsilon\theta^{1 - \alpha}r^{\alpha},
    \end{equation}
    and
    \begin{equation}
        |\hat{u}_{\theta}'(r) - \alpha c_0\theta^{1 - \alpha}r^{\alpha - 1}| \leq \epsilon \theta^{1 - \alpha}r^{\alpha - 1}.
    \end{equation}
    Here, 
    \begin{equation}
        \alpha = \frac{-(d-2) + \sqrt{d^2 - 8d+8}}{2} \in [-2,-1), \ d = n+m-1.
    \end{equation}
\end{proposition}
\begin{proof}[Proof of proposition \ref{approximate formula for inner self shrinkers}]
    We first recall that $\psi_1$ satisfies the asymptotic formula
    \begin{equation}\label{asymptotic formula of psi1}
        \psi_1(y) = \sqrt{\frac{m-1}{n-1}}y + c_1(n,m)y^{\alpha} + o(y^{\alpha}) \text{ as }y \to \infty \text{ in }C^1.
    \end{equation}
    In particular, for any $\epsilon > 0$, we can find $y_0 = y_0(n,m, \epsilon) > 0$ so that for all $y \geq y_0$, 
\begin{equation}
    |\psi_1(y) - \sqrt{\frac{m-1}{n-1}}y - c_1y^{\alpha}| \leq \frac{\epsilon}{2} y^{\alpha},
\end{equation}
and
\begin{equation}
    |\psi_1'(y) - \sqrt{\frac{m-1}{n-1}} - c_1\alpha y^{\alpha - 1}| \leq \frac{\epsilon}{2} y^{\alpha - 1}.
\end{equation}
    We now use lemma \ref{key estimate for inner shrinkers}. Then, for any $\theta \in (0, \theta_0(n,m))$, $y = \frac{x}{\theta}\in [0, \frac{1}{10M_0\theta}]$,
    \begin{equation}
        |u_{\theta}(x) - \theta \psi_1(\theta^{-1}x)| = |\theta^3E_{\theta}(\theta^{-1}x)| \leq \theta^3M_0(\theta^{-1}x)^2(1 + \theta^{-1}x)^{\alpha},
    \end{equation}
    and
    \begin{equation}
        |u_{\theta}'(x) - \psi_1'(\theta^{-1}x)| = |\theta^2E_{\theta}'(\theta^{-1}x)| \leq \theta M_0x(1 + \theta^{-1}x)^{\alpha}.
    \end{equation}
    Therefore, by combining the previous two estimates, we see that for 
    \begin{align*}
        |u_{\theta}(x) - \sqrt{\frac{m-1}{n-1}}x - c_1\theta(\theta^{-1}x)^{\alpha}| &\leq \frac{\epsilon}{2} \theta^{1-\alpha}x^{\alpha} + |u_{\theta}(x) - \theta \psi_1(\theta^{-1}x)| \\ & \leq \frac{\epsilon}{2} \theta^{1 - \alpha}x^{\alpha} +  M_0x^2\theta^{1 - \alpha}x^{\alpha} \\ & = (\frac{\epsilon}{2} + M_0x^2)\theta^{1 - \alpha}x^{\alpha},
    \end{align*}
    provided $\theta^{-1}x \geq y_0(n, m,\epsilon)$, and $x \leq \frac{1}{10M_0}$. By possibly making $x$ smaller depending on $n,m$ and $\epsilon > 0$, we can ensure that $M_0x^2 < \frac{\epsilon}{2}$. Thus we see that we can find $\delta_0 = \delta_0(n, m,\epsilon) > 0$ so that whenever $x \leq \delta_0$, and $\frac{\theta}{x} \leq \delta_0$, then
    \begin{equation}\label{appx in graphical gauge}
        |u_{\theta}(x) -\sqrt{\frac{m-1}{n-1}}x -  c_1\theta^{1 - \alpha}x^{\alpha}| \leq \epsilon\theta^{1 - \alpha}x^{\alpha}.
    \end{equation}
    Likewise, the same argument implies 
    \begin{equation}\label{appx in grapical gauge derivative}
        |u_{\theta}'(x) -\sqrt{\frac{m-1}{n-1}} -  c_1\alpha\theta^{1 - \alpha}x^{\alpha-1}| \leq \epsilon\theta^{1 - \alpha}x^{\alpha - 1}
    \end{equation}
    as well. Once we have estimates \eqref{appx in graphical gauge}, and \eqref{appx in grapical gauge derivative}, proposition \ref{approximate formula for inner self shrinkers} follows by using the formula \eqref{normal graph over simons cone of inner self shrinkers} to convert estimates \eqref{appx in graphical gauge}, and \eqref{appx in grapical gauge derivative} to corresponding estimates of $\hat{u}_{\theta}$.
\end{proof}
We now show that the family of self-shrinkers we constructed in proposition \ref{inner family of self shrinkers} together with the minimal quadratic cone foliates a small ball centered at the origin in $\mathbf{R}^{d+1}$ whenever $d+1 = n+m\geq 10$. 
\begin{theorem}\label{inner region foliation}
    For each $0 < \theta < \sqrt{2(\min(n,m)-1)}$, let $\Gamma_{\theta}$ be the self-shrinker constructed in proposition \ref{inner family of self shrinkers}, and let $\Lambda_{\theta}$ be the self-shrinker constructed in remark \ref{remark : construction of second family by reflection}. Let $\Sigma$ be the $O(n)\times O(m)$ symmetric minimal quadratic cone. Then there exists $r_0 = r_0(n,m) >0$ so that 
    \begin{equation}
        \{\Gamma_{\theta}\}_{\theta \in (0, \sqrt{2(\min(n,m)-1)})}\cup \{\Lambda_{\theta}\}_{\theta \in (0, \sqrt{2(\min(n,m)-1)})} \cup \{\Sigma\}
    \end{equation}
    foliate $B(0, r_0)$.
\end{theorem}
\begin{proof}[Proof of theorem \ref{inner region foliation}]
    Recall that for each $0 < \theta < \sqrt{2(\min(n,m)-1)}$, we defined the rescaled profile function
    \begin{equation}
        \Tilde{u}_{\theta}(y) = \theta^{-1}u_{\theta}(\theta y) = \psi_1(y) + \theta^2E_{\theta}(y),
    \end{equation}
    where $\psi_1$ is the solution to the initial value problem \eqref{minimal surface equation}. By lemma \ref{key estimate for inner shrinkers}, we can find $\theta_0 = \theta_0(n,m) \in (0, \sqrt{2(\min(n,m)-1)})$, and $M_0 = M_0(n,m) > 3$ so that whenever $0 < \theta \leq \theta_0$, $M \geq M_0$ then
    \begin{equation}
        \|E_{\theta}\|_{2, \frac{1}{10M\theta}} \leq M,
    \end{equation}
    where the norm is defined in definition \ref{weighted norm for inner shrinker construction}. \\

    We claim that we can find $x_0 = x_0(n,m) > 0$ so that after possibly shrinking $\theta_0(n,m) > 0$, for all $0 \leq x \leq x_0$, $0 < \theta \leq \theta_0(n,m)$, 
    \begin{equation}\label{monotonicity in theta of inner self shrinkers}
        \frac{\partial u_{\theta}(x)}{\partial \theta} > 0.
    \end{equation}
    Note that we already established smooth dependence of $u_{\theta}$ with respect to $\theta$; thus we can indeed differentiate in $\theta$. Since $u_{\theta}(x) = \theta \Tilde{u}_{\theta}(\theta^{-1}x) = \theta \psi_1(\theta^{-1}x) + \theta^3E_{\theta}(\theta^{-1}x)$, we have
    \begin{equation}
        \frac{\partial u_{\theta}(x)}{\partial \theta} = (\psi_1(y) - y\psi_1'(y)) + \theta^2E_{\theta}(y) - \theta^2yE_{\theta}'(y) + \theta\frac{\partial \Tilde{u}_{\theta}(y)}{\partial \theta} 
    \end{equation}
    with $y = \theta^{-1}x$.
    We now estimate each term separately. First, by the asymptotic formula of $\psi_1$ given by \eqref{asymptotic behavior of hardsimonleafformula}, we have
    \begin{equation}
        \lim_{y \to \infty}\frac{\psi_1 - y\psi_1'}{(1 + y)^{\alpha}} = c_2(n,m) > 0.
    \end{equation}
    On the other hand, we have
    \begin{equation}
        \frac{d}{dy}(\psi_1 - y\psi_1') = -y\psi_1''(y) < 0.
    \end{equation}
    This implies that we can find some constant $c_3 = c_3(n,m) > 0$ so that
    \begin{equation}
        \frac{\psi_1 - y\psi_1'}{(1 + y)^{\alpha}} \geq c_3 > 0.
    \end{equation}
    By using lemma \ref{key estimate for inner shrinkers}, we estimate for each $M \geq M_0$, $0 < \theta \leq \theta_0$, $y = \theta^{-1}x \leq \frac{1}{10M\theta}$, 
    \begin{equation}
        \frac{\theta^2E_{\theta}(y)}{(1 + y)^{\alpha}} \geq - \theta^2M (\frac{1}{10M\theta})^2 = -\frac{1}{100M},
    \end{equation}
    \begin{equation}
        \frac{-\theta^2yE_{\theta}'(y)}{(1 + y)^{\alpha}} \geq -\theta^2M(\frac{1}{10M\theta})^2 \geq -\frac{1}{100M}, 
    \end{equation}
    To estimate $\frac{\partial \Tilde{u}_{\theta}}{\partial \theta}$, we proceed exactly as in lemma \ref{key estimate for inner shrinkers}. By differentiating \eqref{initial value problem for inner self shrinkers, rescaled version}, we see that $v_{\theta} = \partial_{\theta}\Tilde{u}_{\theta}$ satisfies the initial value problem
    \begin{equation}
        \begin{cases}
            \mathcal{L}v_{\theta} = \theta(y\psi_1' - \psi_1) + \theta^3(yE'_{\theta} - E_{\theta}) + \frac{\theta^2}{2}(yv_{\theta}' - v_{\theta}) + F_{\theta} \\ v_{\theta}(0) = v_{\theta}'(0) = 0,
        \end{cases}
    \end{equation}
    where
    $F_{\theta} = F_{1, \theta} + F_{2, \theta} + F_{3, \theta}$
    with
    \begin{equation}
        F_{1, \theta} = \frac{v_{\theta}''(2\theta^2E_{\theta}'\psi_1' + \theta^4(E_{\theta}')^2)}{(1 + (\Tilde{u}_{\theta}')^2)(1 + (\psi_1')^2)},
    \end{equation}
    \begin{equation}
        F_{2, \theta}= \frac{(m-1)v_{\theta}(2\theta^2E_{\theta}\psi_1 + \theta^4E_{\theta}^2)}{\Tilde{u}_{\theta}^2\psi_1^2}
    \end{equation}
    \begin{equation}
       F_{3, \theta} = v_{\theta}'(-\frac{2\psi_1'\psi_1''}{(1 + (\psi_1')^2)^2} + \frac{2(\psi_1' + \theta^2E_{\theta}')(\psi_1'' + \theta^2E_{\theta}'')}{(1 + (\psi_1' + \theta^2E_{\theta}')^2)^2}),
    \end{equation}
 and $\mathcal{L}$ is given by \eqref{linearized minimal surface equation at leaf}.
    By using lemma \ref{key estimate for inner shrinkers} to compute the `inhomogeneous term', and using \eqref{estimate for linearized minimal surface equation}, we see that for all $0 < \theta \leq \theta_0$, $M \geq M_0 > 0$, we have
    \begin{equation}
       \|v_{\theta}\|_{2, \frac{1}{10M\theta}} \leq c(n,m) \|\mathcal{L}v_{\theta}\|_{0, \frac{1}{10M\theta}} \leq C(n,m)(\frac{1}{M} + \theta)\|v_{\theta}\|_{2, \frac{1}{10M\theta}} + c(n,m)\theta.
    \end{equation} 
    Therefore, by possibly enlarging $M_0(n,m) > 3$ and shrinking $\theta_0(n,m) > 0$, we have 
    \begin{equation}
        \|\partial_{\theta}\Tilde{u}_{\theta}\|_{2, \frac{1}{10M\theta}} \leq C(n,m)\theta
    \end{equation}
    for all $M \geq M_0$, $0 < \theta \leq \theta_0$. 
    This implies that for any $M \geq M_0$, $0 \leq x \leq \frac{1}{10M}$, $0 < \theta \leq \theta_0$, we have
    \begin{equation}
        \frac{\theta\partial_{\theta}\Tilde{u}_{\theta} }{(1 + y)^{\alpha}}\geq -C(n,m)\theta^2(\frac{1}{10M\theta})^2  = -\frac{C(n,m)}{100M^2}.
    \end{equation}
    Thus, combining all together, there exists $\theta_0(n,m) > 0$, $M_0(n,m) > 0$ so that for any $M \geq M_0$, $0 < \theta \leq \theta_0$, $0 \leq x \leq \frac{1}{10M}$, 
    \begin{equation}
        \frac{\partial_{\theta}u_{\theta}(x)}{(1 + \theta^{-1}x)^{\alpha}} \geq c_3(n,m) - \frac{C(n,m)}{M}.
    \end{equation}
    Thus, by possibly enlarging $M_0(n) >0$, we obtain \eqref{monotonicity in theta of inner self shrinkers} for $x \leq \frac{1}{10M_0} = x_0(n,m)$.\\

    In view of the definitions of the self-shrinkers in theorem \ref{inner region foliation} and the monotonicity \eqref{monotonicity in theta of inner self shrinkers}, the proof is complete once we show that
    \begin{equation}
        \lim_{\theta \to 0}u_{\theta}(x) = \sqrt{\frac{m-1}{n-1}}x
    \end{equation}
    for each $x\in [0, x_0]$.
    This is equivalent to
    \begin{equation}
        \lim_{\theta \to 0}\theta^3E_{\theta}(\theta^{-1}x) = 0
    \end{equation}
    for each $x\in [0, x_0]$.
    By using lemma \ref{key estimate for inner shrinkers}, we see that
    \begin{equation}
        |\theta^3E_{\theta}(\theta^{-1}x)| \leq \theta M_0x_0^2 \to 0 \text{ as }\theta \to 0.
    \end{equation}
    This finishes the proof of theorem \ref{inner region foliation} by setting $r_0(n,m) = \min (\theta_0, x_0) > 0$.
\end{proof}
\section{Construction and analysis of `trumpets'}\label{trumpetsection}
In this section, we construct a one parameter family of self-shrinkers, which we call `trumpets'. These surfaces are $O(n) \times O(m)$ symmetric self-shrinkers with boundary that are properly embedded outside a fixed ball centered at the origin. These surfaces are analogous to the rotationally symmetric self-shrinkers constructed in \cite{kleene2014self}. We first construct the `trumpets' (proposition \ref{Existence of trumpets}), and then obtain an approximate formula of the profile function of the `trumpets'.\\

We first construct the `trumpets'.
\begin{proposition}\label{Existence of trumpets}
For each $n,m \geq 2$, $0 < \sigma < \sqrt{\frac{m-1}{n-1}}$, there exists $u_{\sigma} : [x_{s}(\sigma) , \infty) \to \mathbf{R}_+$ so that\\\\
(i) There exists $x_{s}(\sigma) \in [\sqrt{2(n-1)}, \frac{\sqrt{2(m-1)}}{\sigma}]$ so that $u_{\sigma}(x_{s}(\sigma)) = \sqrt{2(m-1)}$.\\
(ii) $u_{\sigma} \in C^0([x_{s}(\sigma), \infty)) \cap C^{\infty}((x_{s}(\sigma), \infty))$  solves 
\begin{equation}\label{selfshrinkerequationtrumptets}
    \frac{u_{\sigma}''}{1 + (u_{\sigma}')^2} = (\frac{x}{2} - \frac{n-1}{x})u_{\sigma}' - \frac{u_{\sigma}}{2} + \frac{m-1}{u_{\sigma}}.
\end{equation}
In other words, the `trumpet' defined by
\begin{equation}
    T_{\sigma} = \{(xw_1, u_{\sigma}(x)w_2) \ | \ (x, w_1, w_2) \in [x_{s}(\sigma), \infty) \times \mathbf{S}^{n-1} \times \mathbf{S}^{m - 1}\},
\end{equation}
is an $O(n)\times O(m)$ symmetric self-shrinker. \\
(iii) $u_{\sigma}$ is strictly increasing, convex, and satisfies $\sigma x \leq u_{\sigma}(x) < \sqrt{\frac{m-1}{n-1}}x$ for $x > x_{s}(\sigma)$ and is asymptotic to $y = \sigma x$. More specifically, there exists $c = c(n,m) > 0$ so that
\begin{equation}\label{asymptotics of trumpets}
    |u_{\sigma}(x) - \sigma x| \leq \frac{c(n,m)}{\sigma x}, \ \ |u_{\sigma}'(x) - \sigma| \leq \frac{c(n,m)}{\sigma x^2}
\end{equation}
hold for all $x \geq x_{s}(\sigma)$. 
\end{proposition}
\begin{remark}\label{remark, trumpets from reflection}
    As in remark \ref{remark : construction of second family by reflection}, by applying proposition \ref{Existence of trumpets} with $n,m$ reversed, and then applying the reflection $R_{n,m}$, we obtain one more `trumpet' denoted by $\Tilde{T}_{\sigma}$ for $0 < \sigma < \sqrt{\frac{n-1}{m-1}}$.
\end{remark}
\begin{proof}[Proof of proposition \ref{Existence of trumpets}]
For each $0 < \sigma < \sqrt{\frac{m-1}{n-1}}$, $a > \frac{\sqrt{2(m-1)}}{\sigma}$, we define 
    $$u_{\sigma, a} : (x_{0}(\sigma, a), a] \to \mathbf{R}_+$$ to be the solution to the initial value problem
    \begin{equation}
        \begin{cases}
            \frac{u_{\sigma, a}''}{1 + (u_{\sigma, a}')^2} = (\frac{x}{2} - \frac{n-1}{x})u_{\sigma, a}' - \frac{u_{\sigma, a}}{2} + \frac{m-1}{u_{\sigma, a}} \\
            u_{\sigma, a}(a) = \sigma a \ \ \ \ u'_{\sigma, a}(a) = \sigma
        \end{cases}
    \end{equation}
    with 
    \begin{equation}\label{definition ofx0(sigma, a)}
      x_{0}(\sigma, a) = \inf\{x_0 \geq \sqrt{2(n-1)}\ | \ u_{\sigma, a} \text{ is well defined on }(x_0, a] \text{ with }u_{\sigma, a} > \sqrt{2(m-1)} \}. 
    \end{equation}
 We claim that for $x \in (x_0(\sigma, a), a]$, the following estimates hold. 
    \begin{equation}\label{key apriori estimate of appx sol for trumpets}
       \max(\sqrt{2(m-1)}, \sigma x) \leq u_{\sigma, a}(x) < \sqrt{\frac{m-1}{n-1}}x, \ \ \phi_{\sigma, a}(x) = (\frac{x}{2} - \frac{n-1}{x})u_{\sigma, a}' - \frac{u_{\sigma, a}}{2} + \frac{m-1}{u_{\sigma, a}} > 0. 
    \end{equation}
    $u_{\sigma, a}(x) \geq \sqrt{2(m-1)}$ is automatic by the definition of $x_0(\sigma, a)$ given in \eqref{definition ofx0(sigma, a)}. Therefore, we focus on showing $\sigma x \leq u_{\sigma, a}$.
    When $x = a$, estimates \eqref{key apriori estimate of appx sol for trumpets} hold by the given initial condition. Then by continuity, there exists a small $\epsilon > 0$ so that for all $a - \epsilon \leq x \leq a$, 
    $$u_{\sigma, a}(x) < \sqrt{\frac{m-1}{n-1}}x, \ \ \phi_{\sigma, a}(x) > 0.$$
    Also, by $\phi_{\sigma, a} > 0$ and equation \eqref{selfshrinkerequationtrumptets}, we have $u_{\sigma, a}''(x) > 0$, hence $$u'_{\sigma, a}(x) < u_{\sigma,a}'(a) = \sigma$$ for $x \in [a - \epsilon, a)$. Therefore,  $\sigma x - \sigma a\leq u_{\sigma, a}(x) - u_{\sigma, a}(a) = u_{\sigma, a}(x) - \sigma a,$ hence we have
    $$\sigma x \leq u_{\sigma,a}(x)$$
    as well. This means \eqref{key apriori estimate of appx sol for trumpets} hold for all $x \in [a - \epsilon, a]$ for some $\epsilon > 0$.\\

Define 
\begin{equation}\label{definition of x0bar}
        \overline{x_0}(\sigma, a) = \inf\{x_0 > x_0(\sigma, a) \ | \ \text{\eqref{key apriori estimate of appx sol for trumpets} hold for } x \in [x_0, a]\}
    \end{equation}
    By the previous discussion, $\overline{x_0}(\sigma, a) \in [x_0(\sigma, a), a)$ is well defined. We see that \eqref{key apriori estimate of appx sol for trumpets} hold for all $x > x_0(\sigma, a)$ if and only if $\overline{x_0}(\sigma, a) = x_0(\sigma, a)$. \\

    Suppose not and $\overline{x_0}(\sigma, a) > x_0(\sigma, a)$. Then by the definition of $\overline{x}_0(\sigma, a)$ given in \eqref{definition of x0bar}, we see that 
    \begin{equation}\label{Property 1}
        \text{\eqref{key apriori estimate of appx sol for trumpets} hold for $x \in (\overline{x_0}(\sigma, a), a]$},
    \end{equation}
     and one of the three cases below must hold: either
     \begin{equation}\label{possible 2-1}
         u_{\sigma, a}(\overline{x_0}(\sigma, a)) = \sigma \overline{x_0}(\sigma, a),
     \end{equation}
      or 
      \begin{equation}\label{possible 2-2}
        u_{\sigma, a}(\overline{x_0}(\sigma, a)) = \sqrt{\frac{m-1}{n-1}}\overline{x_0}(\sigma, a) , 
      \end{equation}
      or 
      \begin{equation}\label{possible 2-3}
         \phi_{\sigma, a}(\overline{x_0}(\sigma, a)) = 0. 
      \end{equation} 
      Case \eqref{possible 2-1} is not possible; otherwise, by the mean value theorem, there exists $\overline{x_0}(\sigma, a) < x_1 < a$ so that $u_{\sigma, a}'(x_1) = \sigma$. On the other hand, by \eqref{Property 1} and equation \eqref{selfshrinkerequationtrumptets}, $u_{\sigma, a}'$ strictly increases in $(\overline{x_0}(\sigma, a), a]$. Therefore, by combining this with $u_{\sigma, a}'(a) = \sigma$ we have
      $$\sigma = u_{\sigma, a}'(x_1) < u_{\sigma, a}'(a) = \sigma $$
      which is a contradiction.\\
      
      If case \eqref{possible 2-2} holds, then by \eqref{Property 1}, $u_{\sigma, a}'(\overline{x_0}(\sigma, a)) \leq \sqrt{\frac{m-1}{n-1}}$. Actually $$u_{\sigma, a}'(\overline{x_0}(\sigma, a)) < \sqrt{\frac{m-1}{n-1}}$$ must hold. If this is not the case, by uniqueness of solution to initial value problems together with the fact that the function $u(x) = \sqrt{\frac{m-1}{n-1}}x$
      is also a solution to equation \eqref{selfshrinkerequationtrumptets}, we must have $u_{\sigma, a}(x) = \sqrt{\frac{m-1}{n-1}}x$ for $x \geq \overline{x_0}(\sigma, a)$ which contradicts \eqref{Property 1}. Combining $u_{\sigma, a}'(\overline{x_0}(\sigma, a)) < \sqrt{\frac{m-1}{n-1}}$ with $\overline{x_0}(\sigma, a) > \sqrt{2(n-1)}$, we obtain
    $$\phi_{\sigma, a}(\overline{x_0}(\sigma, a)) < (\frac{\overline{x_0}(\sigma, a)}{2} - \frac{n-1}{\overline{x_0}(\sigma, a)})\sqrt{\frac{m-1}{n-1}} - \frac{\overline{x_0}(\sigma, a)}{2} + \frac{m-1}{\sqrt{\frac{m-1}{n-1}}\overline{x_0}(\sigma, a)} = 0$$
    which contradicts \eqref{Property 1} because of continuity of $\phi_{\sigma, a}$. Therefore, we see that
    \begin{equation}\label{weak C^0 for trumpet}
        \sigma x \leq u_{\sigma, a}(x) < \sqrt{\frac{m-1}{n-1}}x
    \end{equation}
    hold for $x \in [ \overline{x_0}(\sigma, a), a]$.\\
    
    If case \eqref{possible 2-3} holds, then by solving $\phi_{\sigma, a}(\overline{x_0}(\sigma, a)) = 0$ for $u_{\sigma, a}'$, together with $\overline{x_0}(\sigma, a) >\sqrt{2(n-1)}$ and $u_{\sigma, a}(\overline{x_0}(\sigma, a)) > \sqrt{2(m-1)}$, we have
    $$u_{\sigma, a}'(\overline{x_0}(\sigma, a)) > 0.$$
    On the other hand, by using equation \eqref{selfshrinkerequationtrumptets} to compute $\phi_{\sigma, a}'$, and then using \eqref{Property 1}, assumption \eqref{possible 2-3}, and estimate \eqref{weak C^0 for trumpet}, we have
    $$0 \leq \phi_{\sigma, a}'(\overline{x_0}(\sigma, a)) = [\frac{n-1}{\overline{x_0}(\sigma, a)^2} - \frac{m-1}{u_{\sigma, a}(\overline{x_0}(\sigma, a))^2}]u_{\sigma, a}'(\overline{x_0}(\sigma, a)) < 0$$
    which is a contradiction. Thus all three cases are impossible, which means our initial assumption $\overline{x}_0(\sigma, a) > x_0(\sigma, a)$ is false. Thus 
    $$\overline{x}_0(\sigma, a) =x_0(\sigma, a)$$
    and \eqref{key apriori estimate of appx sol for trumpets} holds for all $x \in (x_0(\sigma, a), a]$.\\

    Note that equation \eqref{selfshrinkerequationtrumptets} and \eqref{key apriori estimate of appx sol for trumpets} imply that $ 0 \leq u_{\sigma, a}' \leq \sigma$, $u_{\sigma}'' \geq 0$, and $u_{\sigma}$ is bounded from below by $\sqrt{2(m-1)}$. Hence $u_{\sigma}$ extends in $C^1$ to $x = x_0(\sigma, a)$. We also obtain
\begin{equation}\label{estimate of x0(sigma, a)}
      x_0(\sigma, a) \in [\sqrt{2(n-1)}, \frac{\sqrt{2(m-1)}}{\sigma}], \ \ \ u_{\sigma, a}(x_{0}(\sigma, a)) = \sqrt{2(m-1)}.
 \end{equation}
    By \eqref{key apriori estimate of appx sol for trumpets}, if 
   $$\lim_{x \to x_0(\sigma, a)}(x, u_{\sigma, a}(x)) \in (\sqrt{2(n-1)}, \infty) \times (\sqrt{2(m-1)}, \infty)$$
    then one can strictly extend the graphical solution to parts of $x < x_0(\sigma, a)$ so that the extended solution still remains in $(\sqrt{2(n-1)}, \infty) \times (\sqrt{2(m-1)}, \infty)$, which contradicts the minimality of $x_0(\sigma, a)$ given in \eqref{definition ofx0(sigma, a)}. This together with the bound $$ u_{\sigma,a}(x) < \sqrt{\frac{m-1}{n-1}}x$$ implies estimate \eqref{estimate of x0(sigma, a)} for $x_0(\sigma, a)$. \\

    By equation \eqref{selfshrinkerequationtrumptets}, and \eqref{key apriori estimate of appx sol for trumpets}, we have the following estimates.
    \begin{align}\label{apriori estimate for trumpets}
        &\max(\sigma x, \sqrt{2(m-1)}) \leq u_{\sigma, a}(x) \leq \sqrt{\frac{m-1}{n-1}}x, \ 0 \leq u_{\sigma, a}'(x) \leq \sigma, \ \phi_{\sigma, a}(x) = \frac{u_{\sigma, a}''}{1 + (u_{\sigma, a}')^2} \geq 0
    \end{align}
    for $x \in (x_0(\sigma, a), a]$, together with 
    \begin{equation}
        \sqrt{2(n-1)} \leq x_0(\sigma, a) \leq \frac{\sqrt{2(m-1)}}{\sigma}, u_{\sigma, a}(x_0(\sigma, a)) = \sqrt{2(m-1)}.
    \end{equation}
    Then one can find $a_j \to \infty$ so that
    $$x_0(\sigma, a_j) \to x_{s}(\sigma) \in [\sqrt{2(n-1)}, \frac{\sqrt{2(m-1)}}{\sigma}]$$ and 
    $$u_{\sigma, a_j} \to u_{\sigma} \text{ in }C^2_{loc}((x_{s}(\sigma), \infty)).$$
    We then see that $u_{\sigma}$ solves equation \eqref{selfshrinkerequationtrumptets}, and estimates \eqref{apriori estimate for trumpets} pass to the limit and thus $u_{\sigma}$ satisfies
   \begin{equation}\label{final apriori estimates for trumpets}
       \max(\sqrt{2(m-1)}, \sigma x) \leq u_{\sigma}(x) \leq \sqrt{\frac{m-1}{n-1}}x, \  \ 0 \leq u_{\sigma}'(x) \leq \sigma, \  \ \phi_{\sigma} = \frac{u_{\sigma}''}{1 + (u_{\sigma}')^2}\geq 0
   \end{equation}
   for all $x \in (x_{s}(\sigma), \infty)$. In particular, we see that $u_{\sigma}$ extends in $C^1$ up to $x = x_{s}(\sigma)$ with $u_{\sigma}(x_{s}(\sigma)) \geq \sqrt{2(m-1)}$. We claim that 
   $$u_{\sigma}(x_{s}(\sigma)) = \sqrt{2(m-1)}$$
   If $u_{\sigma}(x_{s}(\sigma)) >\sqrt{2(m-1)}$, then by $0 \leq u_{\sigma}'(x) \leq \sigma$, we can find some small $\delta > 0$ so that 
   $u_{\sigma}(x) \geq \sqrt{2(m-1)} + \delta$
for all $x \in [x_{s}(\sigma), x_{s}(\sigma) +\delta]$. On the other hand, since $x_0(\sigma, a_j) \to x_{s}(\sigma)$ with $u_{\sigma, a}(x_0(\sigma, a)) = \sqrt{2(m-1)}$ and $0 \leq u_{\sigma, a}'(x) \leq \sigma$, for all sufficiently large $j$, we have $x_{s}(\sigma) - \frac{\delta}{100\sigma} < x_0(\sigma, a_j) < x_{s}(\sigma) + \frac{\delta}{100\sigma}$ and 
$$u_{\sigma, a_j}(x_{s}(\sigma) + \frac{\delta}{100\sigma}) \leq \sqrt{2(m-1)}+ \frac{\delta}{50}$$
which contradicts the $C^2_{loc}$ convergence of $u_{\sigma, a_j} \to u_{\sigma}$.\\

    By the estimates \eqref{final apriori estimates for trumpets}, 
    $\phi_{\sigma} \geq 0$. Recalling the definition of $\phi_{\sigma}$
    $$\phi_{\sigma} = (\frac{x}{2} - \frac{n-1}{x})u_{\sigma}'(x) - \frac{u_{\sigma}}{2} + \frac{m-1}{u_{\sigma}}$$
    and that $u_{\sigma} \geq \sqrt{2(m-1)}$, we see that $u_{\sigma}$ is non-decreasing. Moreover, in the region $(\sqrt{2(n-1)}, \infty) \times (\sqrt{2(m-1)}, \infty)$, $u_{\sigma}'(x) > 0$. Thus to show that $u_{\sigma}$ strictly increases, we only need to show that $u_{\sigma}(x) >\sqrt{2(m-1)}$ for $x > x_{s}(\sigma)$. If this is not true, then by uniqueness of solutions to initial value problems together with the fact that the constant function $y = \sqrt{2(m-1)}$ is a solution to equation \eqref{selfshrinkerequationtrumptets}, and $u_{\sigma}(x_{s}(\sigma)) = \sqrt{2(m-1)}$, we must have $u_{\sigma}(x) \equiv \sqrt{2(m-1)}$ which contradicts $\sigma x \leq u_{\sigma}(x)$ for large $x$. Thus $u_{\sigma}$ is strictly increasing. Convexity in the interior follows from $\phi_{\sigma}(x) = \frac{u_{\sigma}''}{1 + (u_{\sigma}')^2} \geq 0$. Finally, the inequality $\sigma x \leq u_{\sigma}(x)$ follows from \eqref{final apriori estimates for trumpets}, and the upper bound $u_{\sigma}(x) < \sqrt{\frac{m-1}{n-1}}x$ when $x > x_s(\sigma)$ follows from \eqref{final apriori estimates for trumpets} together with $u_{\sigma}(x_s(\sigma)) = \sqrt{2(m-1)}  \leq \sqrt{\frac{m-1}{n-1}}x_{s}(\sigma)$.\\

    To show the desired asymptotics of $u_{\sigma}$, we go back to the approximate solutions $u_{\sigma, a}$. By following the proof of lemma 2 of \cite{kleene2014self}, we see that each $u_{\sigma, a}$ satisfies the integral identity
    $$u_{\sigma, a}(x) = \sigma x + x\int_{x}^{a}\frac{1}{t^2}\int_{t}^{a}(\frac{(m-1)s}{u_{\sigma, a}} - (n-1)u_{\sigma, a}')(1 + (u_{\sigma, a}')^2)e^{-\frac{1}{2}\int_{t}^{s}z(1 + (u_{\sigma, a}')^2)dz}dsdt.$$
     \eqref{apriori estimate for trumpets} implies that the above integral identity also passes through the limit and thus
    \begin{equation}
        u_{\sigma}(x) = \sigma x + x\int_{x}^{\infty}\frac{1}{t^2}\int_{t}^{\infty}(\frac{(m-1)s}{u_{\sigma}} - (n-1)u_{\sigma}')(1 + (u_{\sigma}')^2)e^{-\frac{1}{2}\int_{t}^{s}z(1 + (u_{\sigma}')^2)dz}dsdt
    \end{equation}
    holds for all $x\in [x_{s}(\sigma), \infty)$. Then by using estimates \eqref{final apriori estimates for trumpets}, we have 
   \begin{equation}
       |u_{\sigma}(x) - \sigma x| \leq x\int_{x}^{\infty}\frac{1}{t^2}\int_{t}^{\infty}\frac{c(n,m)}{\sigma}e^{-1/4(s^2 - t^2)}dsdt \leq \frac{c(n,m)}{\sigma x}.
   \end{equation}
    Similar calculations yield 
    \begin{equation}
        |u_{\sigma}'(x) - \sigma| \leq \frac{c(n,m)}{\sigma x^2}.
    \end{equation}
    This completes the proof of proposition \ref{Existence of trumpets}.
\end{proof}
In order to use the `trumpets' as barriers in the proof of the main theorems, we need the following $C^0$ estimates of `trumpets'.
\begin{lemma}\label{Alternative $C^0$ for trumpets}
Let $u_{\sigma}$ be the profile function constructed in proposition \ref{Existence of trumpets}. Then
    $$\sigma x\leq u_{\sigma}(x) \leq  \sqrt{\sigma^2x^2 + 2(m-1) - \sigma^22(n-1)}$$
for all $\sigma \in (0, \sqrt{\frac{m-1}{n-1}})$, $x \in [x_{s}(\sigma), \infty)$.
\end{lemma}
A consequence of lemma \ref{Alternative $C^0$ for trumpets} is the following weighted $C^0$ norm estimate for the normal graph function of $T_{\sigma}$ (or equivalently $\Tilde{T}_{\sigma}$), which was used in the proof of lemma \ref{weighted C2 estimate for the graph function giving us higher power nonlinearity}. 
\begin{corollary}\label{normal graph function estimate of trumpets}
    For each $\sigma \in (\frac{1}{2}\sqrt{\frac{m-1}{n-1}}, \sqrt{\frac{m-1}{n-1}})$, let $u_{\sigma} : [x_s(\sigma), \infty) \to \mathbf{R}_+$ be the profile function obtained in proposition \ref{Existence of trumpets}. Then there exist $\overline{L}(n,m)> 10$, $C(n,m) > 0$ so that for any $\sigma \in (\frac{1}{2}\sqrt{\frac{m-1}{n-1}}, \sqrt{\frac{m-1}{n-1}})$, there exists $\hat{v}_{\sigma} : [\overline{L}(n,m) , \infty) \to \mathbf{R}_+$ so that
    \begin{equation}\label{normal graphicality of trumpets over cone}
       \{ rw - \hat{v}_{\sigma}(r)\nu(w) \ | \ r \geq \overline{L}, \ w \in \sqrt{\frac{n-1}{n+m-2}}\mathbf{S}^{n-1}\times \sqrt{\frac{m-1}{n+m-2}}\mathbf{S}^{m-1}\} \subset T_{\sigma}
    \end{equation}
    where, for $w = (\sqrt{\frac{n-1}{n+m-2}}w_1, \sqrt{\frac{m-1}{n+m-2}}w_2)$ with $w_1 \in \mathbf{S}^{n-1}$, $w_2 \in \mathbf{S}^{m-1}$, we define
    \begin{equation}
        \nu(w) = (-\sqrt{\frac{m-1}{n+m-2}}w_1, \sqrt{\frac{n-1}{n+m-2}}w_2).
    \end{equation}
    We also have for any $L \geq \overline{L}(n,m)$
    \begin{equation}\label{weightedc0 estimate for trumpets}
        \sup_{r \geq L}\frac{\hat{v}_{\sigma}(r)}{r} \leq C(n,m)\frac{\hat{v}_{\sigma}(L)}{L}.
    \end{equation}
\end{corollary}
\begin{proof}[Proof of corollary \ref{normal graph function estimate of trumpets}]
    By proposition \ref{Existence of trumpets} together with $\sigma \in (\frac{1}{2}\sqrt{\frac{m-1}{n-1}}, \sqrt{\frac{m-1}{n-1}})$, we see that $u_{\sigma}$ is defined for all $x \geq 2\sqrt{2(n-1)}$. The estimates $u'_{\sigma} \geq 0$, $\sigma x \leq u_{\sigma}(x) <\sqrt{\frac{m-1}{n-1}}x$ (see proposition \ref{Existence of trumpets}) imply that \eqref{normal graphicality of trumpets over cone} holds for any $\overline{L} \geq \max(10, \sqrt{4(n+m)})$. \\

    To prove \eqref{weightedc0 estimate for trumpets}, we claim that by possibly increasing $\overline{L}(n,m) > 10$, we have
    \begin{equation}\label{comparing with two linear functions}
        \sigma x \leq u_{\sigma}(x) \leq \frac{\sqrt{\frac{m-1}{n-1}} + \sigma}{2}x
    \end{equation}
    for all $\sigma \in (\frac{1}{2}\sqrt{\frac{m-1}{n-1}}, \sqrt{\frac{m-1}{n-1}})$, and $x \geq \overline{L}$. The lower bound is immediate from lemma \ref{Alternative $C^0$ for trumpets}. To show the upper bound, we compare
    \begin{equation}
        \sqrt{\sigma^2x^2 + 2(m-1) - \sigma^22(n-1)} \leq 
        \frac{\sqrt{\frac{m-1}{n-1}} + \sigma}{2}x.
    \end{equation}
    One can easily see that as long as $\sigma \in (\frac{1}{2}\sqrt{\frac{m-1}{n-1}}, \sqrt{\frac{m-1}{n-1}})$, $4\sqrt{n-1} \leq x$, the above inequality holds. Thus lemma \ref{Alternative $C^0$ for trumpets} proves the claim.\\

    Then \eqref{comparing with two linear functions} implies that for all $r \geq \overline{L}(n,m)$, $\sigma \in (\frac{1}{2}\sqrt{\frac{m-1}{n-1}}, \sqrt{\frac{m-1}{n-1}})$, 
    \begin{equation}\label{sigma estimate for trumpet}
        c_1(n,m)(\sqrt{\frac{m-1}{n-1}} - \sigma) \leq \frac{\hat{v}_\sigma(r)}{r} \leq c_2(n,m)(\sqrt{\frac{m-1}{n-1}} - \sigma)
    \end{equation}
    for some $0 < c_1(n,m) < c_2(n,m)$, which immediately implies \eqref{weightedc0 estimate for trumpets}, thus completing the proof of corollary \ref{normal graph function estimate of trumpets}.
\end{proof}
\begin{proof}[Proof of lemma \ref{Alternative $C^0$ for trumpets}]
The lower bound is immediate from proposition \ref{Existence of trumpets}. As for the upper bound, by the continuity of $u_{\sigma}$ with respect to $x$, it is enough to consider when $x >x_{s}(\sigma)$. Recall by \eqref{final apriori estimates for trumpets}, we have $\phi_{\sigma}(x) \geq 0$ for all $x \in (x_{s}(\sigma), \infty)$. By rewriting $\phi_{\sigma}(x) \geq 0$ using $x > x_{s}(\sigma) \geq \sqrt{2(n-1)}$ and $u_{\sigma} > \sqrt{2(m-1)}$, we have $$\frac{2u_{\sigma}u_{\sigma}'}{u_{\sigma}^2 - 2(m-1)} \geq 
    \frac{2x}{x^2 - 2(n-1)}.$$
    This inequality can be written as
    $$\frac{d}{dx}\ln(u_{\sigma}^2 - 2(m-1)) \geq \frac{d}{dx}\ln(x^2 - 2(n-1)).$$
    Thus integrating from $x \in (x_{s}(\sigma), \infty)$ to any large $a > 0$, and then letting $a \to \infty$, and combining the result with the asymptotics given in (iii) of proposition \ref{Existence of trumpets}, we obtain the upper bound
    \begin{equation}
        u_{\sigma}(x) \leq \sqrt{ \sigma^2x^2 + 2(m-1)-\sigma^22(n-1)}.
    \end{equation}
   This completes the proof of lemma \ref{Alternative $C^0$ for trumpets}.
\end{proof}

\section*{Acknowledgment}
The author would like to thank his advisor Prof. Nata\v sa \v Se\v sum for her constant support and guidance. He also thanks Prof. Kyeongsu Choi for his suggestion which led to the improvement of the main result.
\printbibliography

@book{ilmanen1994elliptic,
  title={Elliptic regularization and partial regularity for motion by mean curvature},
  author={Ilmanen, Tom},
  volume={520},
  year={1994},
  publisher={American Mathematical Soc.}
}

@article{ancientasymptotical,
  title={Ancient asymptotically cylindrical flows and applications},
  author={Choi, Kyeongsu and Haslhofer, Robert and Hershkovits, Or and White, Brian},
  journal={Inventiones mathematicae},
  volume={229},
  number={1},
  pages={139--241},
  year={2022},
  publisher={Springer}
}

@article{Ecker1991InteriorEF,
  title={Interior estimates for hypersurfaces moving by mean curvature},
  author={Klaus Ecker and Gerhard Huisken},
  journal={Inventiones mathematicae},
  year={1991},
  volume={105},
  pages={547-569},
  url={https://api.semanticscholar.org/CorpusID:122642136}
}

@ARTICLE{White2005-nl,
  title     = "A local regularity theorem for mean curvature flow",
  author    = "White, Brian",
  journal   = "Ann. Math.",
  publisher = "Annals of Mathematics",
  volume    =  161,
  number    =  3,
  pages     = "1487--1519",
  month     =  may,
  year      =  2005,
  language  = "en"}

@ARTICLE{Chodosh2024-qv,
  title     = "Mean curvature flow with generic initial data",
  author    = "Chodosh, Otis and Choi, Kyeongsu and Mantoulidis, Christos and
               Schulze, Felix",
  journal   = "Invent. Math.",
  publisher = "Springer Science and Business Media LLC",
  volume    =  237,
  number    =  1,
  pages     = "121--220",
  month     =  jul,
  year      =  2024,
  copyright = "https://creativecommons.org/licenses/by/4.0",
  language  = "en"
}

@article{pseudolocality,
  title={On short time existence for the planar network flow},
  author={Ilmanen, Tom and Neves, Andr{\'e} and Schulze, Felix},
  journal={Journal of Differential Geometry},
  volume={111},
  number={1},
  pages={39--89},
  year={2019},
  publisher={Lehigh University}
}

@article{kleene2014self,
  title={Self-shrinkers with a rotational symmetry},
  author={Kleene, Stephen and M{\o}ller, Niels},
  journal={Transactions of the American Mathematical Society},
  volume={366},
  number={8},
  pages={3943--3963},
  year={2014}
}

@article{angenent2019unique,
  title={Unique asymptotics of ancient convex mean curvature flow solutions},
  author={Angenent, Sigurd and Daskalopoulos, Panagiota and Sesum, Natasa},
  journal={Journal of Differential Geometry},
  volume={111},
  number={3},
  pages={381--455},
  year={2019},
  publisher={Lehigh University}
}

@article{stolarski2023existence,
  title={Existence of mean curvature flow singularities with bounded mean curvature},
  author={Stolarski, Maxwell},
  journal={Duke Mathematical Journal},
  volume={172},
  number={7},
  pages={1235--1292},
  year={2023},
  publisher={Duke University Press}
}

@article{guo2018analysis,
  title={Analysis of Vel{\'a}zquez’s solution to the mean curvature flow with a type II singularity},
  author={Guo, Siao-Hao and Sesum, Natasa},
  journal={Communications in Partial Differential Equations},
  volume={43},
  number={2},
  pages={185--285},
  year={2018},
  publisher={Taylor \& Francis}
}

@article{hardt1985area,
  title={Area minimizing hypersurfaces with isolated singularities.},
  author={Hardt, Robert and Simon, Leon},
  year={1985},
  publisher={Walter de Gruyter, Berlin/New York Berlin, New York}
}

@article{huang2025mean,
  title={Mean curvature flow converging to an minimizing cone and its Hardt-Simon foliation},
  author={Huang, Jiuzhou},
  journal={arXiv preprint arXiv:2510.11430},
  year={2025}
}

@article{huisken1990asymptotic,
  title={Asymptotic behavior for singularities of the mean curvature flow},
  author={Huisken, Gerhard},
  journal={Journal of Differential Geometry},
  volume={31},
  number={1},
  pages={285--299},
  year={1990},
  publisher={Lehigh University}
}

@article{coldingminicozzi,
author = {Colding, Tobias and Minicozzi, William},
year = {2012},
month = {03},
pages = {},
title = {Generic mean curvature flow I; Generic singularities},
volume = {2},
journal = {Annals of Mathematics. Second Series},
doi = {10.4007/annals.2012.175.2.7}
}

@article{simon1986minimal,
  title={Minimal hypersurfaces asymptotic to quadratic cones in R n+ 1},
  author={Simon, Leon and Solomon, Bruce},
  journal={Inventiones mathematicae},
  volume={86},
  number={3},
  pages={535--551},
  year={1986},
  publisher={Springer}
}

@inproceedings{bernstein2025lower,
  title={Lower bounds on density for topologically nontrivial minimal cones up to dimension six},
  author={Bernstein, Jacob and Wang, Lu},
  booktitle={Forum of Mathematics, Sigma},
  volume={13},
  pages={e122},
  year={2025},
  organization={Cambridge University Press}
}

@article{edelen2023regularity,
  title={Regularity of minimal surfaces near quadratic cones},
  author={Edelen, Nick and Spolaor, Luca},
  journal={Annals of Mathematics},
  volume={198},
  number={3},
  pages={1013--1046},
  year={2023},
  publisher={Department of Mathematics, Princeton University Princeton, New Jersey, USA}
}

@article{schulze2021introduction,
  title={An introduction to Brakke flows},
  author={Schulze, Felix},
  journal={Lecture Notes. Summer School},
  year={2021}
}

@article{huisken2015convex,
  title={Convex ancient solutions of the mean curvature flow},
  author={Huisken, Gerhard and Sinestrari, Carlo},
  journal={Journal of differential geometry},
  volume={101},
  number={2},
  pages={267--287},
  year={2015},
  publisher={Lehigh University}
}

@article{haslhofer2015uniqueness,
  title={Uniqueness of the bowl soliton},
  author={Haslhofer, Robert},
  journal={Geometry \& Topology},
  volume={19},
  number={4},
  pages={2393--2406},
  year={2015},
  publisher={Mathematical Sciences Publishers}
}

@article{haslhofer2016ancient,
  title={Ancient solutions of the mean curvature flow},
  author={Haslhofer, Robert and Hershkovits, Or},
  journal={Communications in Analysis and Geometry},
  volume={24},
  number={3},
  pages={593--604},
  year={2016}
}

@article{LangfordLynch2020,
  title={Sharp one-sided curvature estimates for fully nonlinear curvature flows and applications to ancient solutions},
  author={Langford, Mat and Lynch, Stephen},
  journal={Journal f{\"u}r die reine und angewandte Mathematik (Crelle's Journal)},
  volume={765},
  pages={1--33},
  year={2020},
  publisher={Walter de Gruyter GmbH}
}

@article{angenent2020uniqueness,
  title={Uniqueness of two-convex closed ancient solutions to the mean curvature flow},
  author={Angenent, Sigurd and Daskalopoulos, Panagiota and Sesum, Natasa},
  journal={Annals of mathematics},
  volume={192},
  number={2},
  pages={353--436},
  year={2020},
  publisher={Department of Mathematics, Princeton University Princeton, New Jersey, USA}
}

@article{BrendleChoi2021,
  title={Uniqueness of convex ancient solutions to mean curvature flow in higher dimensions},
  author={Brendle, Simon and Choi, Kyeongsu},
  journal={Geometry \& Topology},
  volume={25},
  number={5},
  pages={2195--2234},
  year={2021},
  publisher={Mathematical Sciences Publishers},
  doi={10.2140/gt.2021.25.2195},
  url={https://doi.org}
}

@article{choi2022translators,
  title={Classification of noncollapsed translators in $\mathbb{R}^4$},
  author={Choi, Kyeongsu and Hershkovits, Or and Haslhofer, Robert},
  journal={arXiv preprint arXiv:2208.14280},
  year={2022}
}

@article{choi2024classification,
  title={Classification of ancient noncollapsed flows in $\mathbb{R}^4$},
  author={Choi, Kyeongsu and Haslhofer, Robert},
  journal={arXiv preprint arXiv:2412.10581},
  year={2024}
}

@article{choi2022bubblesheet,
  title={Classification of bubble-sheet ovals in $\mathbb{R}^4$},
  author={Choi, Beomjun and Daskalopoulos, Panagiota and Du, Wenkui and Haslhofer, Robert and Sesum, Natasa},
  journal={Geometry \& Topology},
  volume={29},
  pages={931--964},
  year={2025},
  publisher={Mathematical Sciences Publishers},
  doi={10.2140/gt.2025.29.931},
  url={https://msp.org/gt/2025/29-2/gt-v29-n2-p08-p.pdf}
}

@article{daskalopoulos2010classification,
  title={Classification of compact ancient solutions to the curve shortening flow},
  author={Daskalopoulos, Panagiota and Hamilton, Richard and {\v{S}}e{\v{s}}um, Nata{\v{s}}a},
  journal={Journal of Differential Geometry},
  volume={84},
  number={3},
  pages={455--464},
  year={2010},
  publisher={International Press of Boston}
}

@article{bourni2019convex,
  title={Convex ancient solutions to curve shortening flow},
  author={Bourni, Theodora and Langford, Mat and Tinaglia, Giuseppe},
  journal={Calculus of Variations and Partial Differential Equations},
  volume={58},
  number={4},
  pages={133},
  year={2019},
  publisher={Springer}
}

@article{bourni2021collapsing,
  title={Collapsing ancient solutions of mean curvature flow},
  author={Bourni, T and Langford, M and Tinaglia, G},
  journal={Journal of Differential Geometry},
  volume={119},
  number={2},
  pages={187--219},
  year={2021},
  publisher={International Press of Boston}
}

@article{Clutterbuck_2007,
  title={Stability of translating solutions to mean curvature flow},
  author={Clutterbuck, Julie and Schnürer, Oliver C. and Schulze, Felix},
  journal={Calculus of Variations and Partial Differential Equations},
  volume={29},
  number={3},
  pages={281--293},
  year={2007},
  publisher={Springer},
  doi={10.1007/s00526-006-0033-1}
}

@article{hershkovits2021translators,
  title={Translators asymptotic to cylinders},
  author={Hershkovits, Or},
  journal={Journal f{\"u}r die reine und angewandte Mathematik (Crelle's Journal)},
  year={2021},
  publisher={Walter de Gruyter GmbH}
}

@article{haslhofer2026mcf,
  title={Mean curvature flow through singularities},
  author={Haslhofer, Robert},
  journal={Proceedings of the International Congress of Mathematicians (ICM)},
  year={2026},
  note={arXiv:2510.01355}
}

@article{DuZhu2025,
  title     = {Spectral quantization for ancient asymptotically cylindrical flows},
  author    = {Du, Wenkui and Zhu, Jingze},
  journal   = {Advances in Mathematics},
  volume    = {479},
  pages     = {110422},
  year      = {2025},
  issn      = {0001-8708},
  doi       = {10.1016/j.aim.2025.110422},
  publisher = {Elsevier}
}

@article{bamler2025pde,
  title={The PDE-ODI principle and cylindrical mean curvature flows},
  author={Bamler, Richard H. and Lai, Yi},
  journal={arXiv preprint arXiv:2512.25050},
  year={2025}
}

@misc{bamler2025classification,
  title={Classification of ancient cylindrical mean curvature flows and the Mean Convex Neighborhood Conjecture},
  author={Richard H. Bamler and Yi Lai},
  year={2025},
  eprint={2512.24524},
  archivePrefix={arXiv},
  primaryClass={math.DG}
}

@article{choi2022ancientlowentropy,
  title={Ancient low-entropy flows, mean-convex neighborhoods, and uniqueness},
  author={Choi, Kyeongsu and Haslhofer, Robert and Hershkovits, Or},
  journal={Acta Mathematica},
  volume={228},
  number={2},
  pages={217--301},
  year={2022},
  publisher={International Press of Boston}
}

@article{choi2026classification,
  title={Classification of ancient finite-entropy curve shortening flows},
  author={Choi, Kyeongsu and Seo, Dong-Hwi and Su, Wei-Bo and Zhao, Kai-Wei},
  journal={arXiv preprint arXiv:2603.09112},
  year={2026}
}

@article{lambert2021ancient,
  title={Ancient solutions in Lagrangian mean curvature flow},
  author={Lambert, Ben and Lotay, Jason D and Schulze, Felix},
  journal={Annali della Scuola Normale Superiore di Pisa, Classe di Scienze},
  volume={22},
  number={3},
  pages={1169--1205},
  year={2021},
  doi={10.2422/2036-2145.201901_016},
  url={https://doi.org}
}

@article{LotaySchulzeSzekelyhidi2024,
  title={Ancient solutions and translators of Lagrangian mean curvature flow},
  author={Lotay, Jason D. and Schulze, Felix and Sz{\'e}kelyhidi, G{\'a}bor},
  journal={Publications Math{\'e}matiques de l'IH{\'E}S},
  volume={139},
  pages={259--294},
  year={2024},
  publisher={Springer},
  doi={10.1007/s10240-023-00143-5},
  url={https://springer.com}
}

@article{Mazet2014,
  title={Minimal hypersurfaces asymptotic to Simons cones},
  author={Mazet, Laurent},
  journal={Journal of the Institute of Mathematics of Jussieu},
  volume={16},
  number={4},
  pages={821--845},
  year={2017},
  publisher={Cambridge University Press},
  url={https://arxiv.org/abs/1407.2474}
}

@ARTICLE{Chodosh2023-yl,
  title         = "Mean curvature flow with generic initial data {II}",
  author        = "Chodosh, Otis and Choi, Kyeongsu and Schulze, Felix",
  month         =  feb,
  year          =  2023,
  copyright     = "http://creativecommons.org/licenses/by/4.0/",
  archivePrefix = "arXiv",
  primaryClass  = "math.DG",
  eprint        = "2302.08409"
}

@article{merle1998optimal,
  title={Optimal estimates for blowup rate and behavior for nonlinear heat equations},
  author={Merle, Frank and Zaag, Hatem},
  journal={Communications on pure and applied mathematics},
  volume={51},
  number={2},
  pages={139--196},
  year={1998},
  publisher={New York: Interscience Publishers, c1949-}
}

@article{velazquez1994curvature,
  title={Curvature blow-up in perturbations of minimal cones evolving by mean curvature flow},
  author={Vel{\'a}zquez, Juan JL},
  journal={Annali della Scuola Normale Superiore di Pisa-Classe di Scienze},
  volume={21},
  number={4},
  pages={595--628},
  year={1994}
}

@article{stolarski2026structure,
  title={On the structure of singularities of weak mean curvature flows with mean curvature bounds},
  author={Stolarski, Maxwell},
  journal={Journal f{\"u}r die reine und angewandte Mathematik (Crelles Journal)},
  volume={2026},
  number={830},
  pages={141--177},
  year={2026},
  publisher={De Gruyter}
}
\end{document}